\documentclass[leqno,12pt,letterpaper]{article}
\usepackage{amssymb}
\usepackage{amsfonts}
\usepackage{graphicx}
\usepackage{amsmath}
\usepackage{cite}
\usepackage{mathrsfs}
\usepackage{xcolor}
\usepackage{enumitem}
\usepackage{hyperref}

\setlist{leftmargin=24pt,labelindent=24pt}
\setlist[enumerate]{wide=0pt, leftmargin=24pt, labelwidth=24pt, align=left}

\newtheorem{theorem}{Theorem}
\newtheorem{claim}{Claim}[section]
\newtheorem{corollary}{Corollary}[section]
\newtheorem{definition}{Definition}[section]

\newtheorem{lemma}{Lemma}[section]
\newtheorem{problem}{Problem}

\newtheorem{remark}{Remark}[section]
\newenvironment{proof}[1][Proof]{\noindent\textbf{#1.} }{\ \rule{0.5em}{0.5em}}
\renewcommand{\refname}{\hfil References Cited\hfil}
\newcommand{\cH}{\mathcal{H}}

\newcommand{\ti}{\tilde{i}}
\newcommand{\tj}{\tilde{j}}
\newcommand{\R}{\mathbb{R}}

\newcommand{\dist}{\text{dist}}

\newcommand{\LA}[1]{\refstepcounter{equation}\text{(\theequation)}\label{#1}}

\begin{document}

\title{$C^m$ Semialgebraic Sections Over the Plane}
\date{\today}
\author{Charles Fefferman, Garving K. Luli}
\maketitle


\section{Introduction}

In this paper we settle the two-dimensional case of a conjecture involving
unknown semialgebraic functions with specified smoothness.

Recall that a \emph{semialgebraic set } $E\subset \mathbb{R}^{n}$ is a union
of finitely many sets of the form
\begin{equation*}
\{x\in \mathbb{R}^n :P_{1}(x),P_{2}(x),\cdots ,P_{r}(x)>0,\text{ and }%
Q_{1}(x)=Q_{2}(x)=\cdots =Q_{s}(x)=0\}
\end{equation*}%
for polynomials $P_{1},\cdots ,P_{r},Q_{1},\cdots ,Q_{s}$ on $\mathbb{R}^n$. (We allow the
cases $r=0$ or $s=0$.)

A \emph{semialgebraic function} $\phi :E\rightarrow \mathbb{R}^{D}$ is a
function whose graph $\{(x,\phi (x)):x\in E\}$ is a semialgebraic set.

We define smoothness in terms of $C^m$ and $C^m_{loc}$. Here, $C^m\left( \mathbb{R}^{n},\mathbb{R}^{D}\right) $ denotes the
space of all $\mathbb{R}^{D}$-valued functions on $\mathbb{R}^{n}$ whose
derivatives up to order $m$ are continuous and bounded on $\mathbb{R}^{n}$. $%
C_{loc}^{m}\left( \mathbb{R}^{n},\mathbb{R}^{D}\right) $ denotes the space
of $\mathbb{R}^{D}$-valued functions on $\mathbb{R}^{n}$ with continuous
derivatives up to order $m$. If $D=1$, we write $C^m\left( \mathbb{R}%
^{n}\right) $ and $C_{loc}^{m}\left( \mathbb{R}^{n}\right) $ in place of $%
C^m\left( \mathbb{R}^{n},\mathbb{R}^{D}\right) $ and $C_{loc}^{m}\left(
\mathbb{R}^{n},\mathbb{R}^{D}\right) $, respectively.

To motivate our conjecture, we pose the following problems.

\begin{problem}[Semialgebraic Whitney Problem; see \cite{zobin-problem}.]
\label{problem1} Fix $m\geq 0$. Let $\phi :E\rightarrow \mathbb{R}$ be
semialgebraic. Suppose $\phi $ extends to a $C^m_{loc}$ function on $\mathbb{R}%
^{n}$. Does it necessarily extend to a $C^m_{loc}$ \underline{semialgebraic}
function on $\mathbb{R}^{n}$?
\end{problem}

\begin{problem}[Linear Equations]
\label{pr2}Fix $m\geq 0.$ Consider the linear equation
\begin{equation}
A_{1}F_{1}+\cdots +A_{D}F_{D}=f  \label{equation}
\end{equation}%
for unknowns $F_{1},\cdots ,F_{D}$ on $\mathbb{R}^{n}$, where $A_{1},\cdots
,A_{D}$, $f$ are given semialgebraic functions. If equation $\left( \ref%
{equation}\right) $ admits a $C^m_{loc}$ solution $F_{1},\cdots ,F_{D}$, does it necessarily admit a $C^m_{loc}$ \underline{%
semialgebraic} solution?
\end{problem}

More generally, in place of (\ref{equation}) we can consider underdetermined
systems of linear equations.

Problem \ref{problem1} was raised by Bierstone and Milman in \cite{zobin-problem}.

Note that $m$ is fixed in the above problems so we are not allowed to lose
derivatives.

Problems \ref{problem1} and \ref{pr2} are instances of a more general
question. The purpose of this paper is to settle that question, and in
particular provide affirmative answers to Problems \ref{problem1} and \ref%
{pr2}, in the case of $C^m_{loc}\left( \mathbb{R}^{2}\right) $.

To pose our more general question, we set up notations and give a few basic
definitions.

Fix $m\geq 0$. If $F\in C^m_{loc}(\mathbb{R}^{n})$ and $x\in \mathbb{R}^{n}$, we
write $J_{x}(F)$ (the \textquotedblleft jet" of $F$ at $x$) to denote the $m$%
-th degree Taylor polynomial of $F$ at $x$.

Thus, $J_{x}(F)$ belongs to $\mathcal{P}$, the vector space of all such
polynomials.

For $x\in \mathbb{R}^{n}$, $P,Q\in \mathcal{P}$, we define $P\odot
_{x}Q=J_{x}(PQ)$. The multiplication $\odot _{x}$ makes $\mathcal{P}$ into a
ring, denoted by $\mathcal{R}_{x}$, the \textquotedblleft ring of $m$-jets
at $x$". We have $J_{x}\left( FG\right) =J_{x}\left( F\right) \odot
_{x}J_{x}\left( G\right) $ for $F,G\in C^m_{loc}\left( \mathbb{R}^{n}\right) $.

We consider vector-valued functions $F=\left( F_{1},\cdots ,F_{D}\right) :%
\mathbb{R}^{n}\rightarrow \mathbb{R}^{D}$, and we write $F\in C^m_{loc}\left(
\mathbb{R}^{n},\mathbb{R}^{D}\right) $ if each $F_{i}\in C^m_{loc}\left( \mathbb{%
R}^{n}\right) $. We define $J_{x}F=\left( J_{x}F_{1},\cdots
,J_{x}F_{D}\right) \in \mathcal{P\oplus \cdots \oplus P}$. Under the natural
multiplication
\begin{equation*}
Q\odot _{x}\left( P_{1},\cdots ,P_{D}\right) :=\left( Q\odot
_{x}P_{1},\cdots ,Q\odot _{x}P_{D}\right) \text{,}
\end{equation*}%
the vector space $\mathcal{P\oplus \cdots \oplus P}$ becomes an $\mathcal{R}%
_{x}$ module, which we denote by $\mathcal{R}_{x}^{D}$.

We will discuss $\mathcal{R}_{x}$-submodules of $\mathcal{R}_{x}^{D}$; we
allow both $\left\{ 0\right\} $ and $\mathcal{R}_{x}^{D}$ as submodules of $%
\mathcal{R}_{x}^{D}$.

Fix $m,n,D$, and a subset $E\subset \mathbb{R}^{n}$. For each $x \in E$, let
\begin{equation*}
H\left( x\right) =f\left( x\right) +I\left( x\right) \subset \mathcal{R}%
_{x}^{D}
\end{equation*}%
be given, where $f\left( x\right) \in \mathcal{R}_{x}^{D}$ and $I\left(
x\right) \subset \mathcal{R}_{x}^{D}$ is an $\mathcal{R}_{x}$-submodule.
Then the family
\begin{equation}
\mathcal{H}=(H(x))_{x\in E}  \label{bundle}
\end{equation}%
is called a \textquotedblleft bundle" over $E$. $H(x)$ is called the
\underline{fiber} of $\mathcal{H}$ at $x$.
\begin{remark}\label{classicalbundle}
We remark that our notion of bundle differs from the notion of a bundle considered
previously (e.g, \cite{FL04}). In the present version, we do not require $E$ to be
compact and we require all the fibers $H\left( x\right) $ to be non-empty.\end{remark}

When $m,n,D$ are not clear from context, we speak of a \textquotedblleft bundle
with respect to $C^m_{loc}\left( \mathbb{R}^{n},\mathbb{R}^{D}\right) $".

If $\mathcal{H}$ is given by (\ref{bundle}) and $E^{\prime }\subset E$, then
we write $\left. \mathcal{H}\right\vert _{E^{\prime }}$ to denote the bundle
$\left( H\left( x\right) \right) _{x\in E^{\prime }}$, and refer to $\mathcal{H}|_{E^\prime}$ as the \underline{restriction of $\mathcal{H}$ to $E^\prime$}.

A \textquotedblleft section" of the bundle $\mathcal{H}$ in \eqref{bundle} is a vector-valued
function $F\in C^m_{loc}(\mathbb{R}^{n},\mathbb{R}^{D})$ such that $J_{x}F\in
H(x)$ for all $x\in E$.

Note that sections $F$ belong to $C^m_{loc}\left( \mathbb{R}^{n},\mathbb{R}%
^{D}\right) $ by definition.

The bundle (\ref{bundle}) is called \textquotedblleft semialgebraic" if
\begin{equation*}
\left\{ \left( x,P_{1},\cdots ,P_{D}\right) :\mathbb{R}^{n}\oplus \mathcal{%
P\oplus \cdots \oplus \mathcal{P}}:x\in E,\left( P_{1},\cdots ,P_{D}\right)
\in H\left( x\right) \right\}
\end{equation*}%
is a semialgebraic set.

We can now state our general problem.

\begin{problem}
\label{problem2} Let $\mathcal{H}=(H(x))_{x\in E}$ be a semialgebraic bundle
with respect to $C^m_{loc}\left( \mathbb{R}^{n},\mathbb{R}^{D}\right) $. If $%
\mathcal{H}$ has a section, does it necessarily have a semialgebraic
section?
\end{problem}

Again, we note that sections of $\mathcal{H}$ must belong to $C^m_{loc}$ for
fixed $m$, so we are not allowed to lose derivatives.

One checks easily that Problems \ref{problem1} and \ref{pr2} are instances
of Problem \ref{problem2}.

Indeed, suppose $\phi :E\rightarrow \mathbb{R}$ is semialgebraic, as in
Problem \ref{problem1}. Set $\mathcal{H=}\left( H\left( x\right) \right)
_{x\in E}$, where
\begin{equation*}
H\left( x\right) =\left\{ P\in \mathcal{P}:P\left( x\right) =\phi \left(
x\right) \right\} \text{.}
\end{equation*}%
Then $\mathcal{H}$ is a semialgebraic bundle, and a section of $\mathcal{H}$ is precisely
a function $F\in C^m_{loc}\left( \mathbb{R}^{n}\right) $ such that $F=\phi $ on $%
E$.

Similarly, given an equation (\ref{equation}) as in Problem \ref{pr2}, set $%
\mathcal{H=}\left( H\left( x\right) \right) _{x\in \mathbb{R}^{n}}$ with
\begin{equation*}
H\left( x\right) =\left\{ \left( P_{1},\cdots ,P_{D}\right) \in \mathcal{P}%
^{D}:A_{1}\left( x\right) P_{1}\left( x\right) +\cdots +A_{D}\left( x\right)
P_{D}\left( x\right) =f\left( x\right) \right\} \text{.}
\end{equation*}%
Then $\mathcal{H}$ is a semialgebraic bundle, and a section of $\mathcal{H}$ is precisely
a solution $F=\left( F_{1},\cdots ,F_{D}\right) \in C^m_{loc}\left( \mathbb{R}%
^{n},\mathbb{R}^{D}\right) $ of equation (\ref{equation}).

In this paper, we settle the two-dimensional case of Problem \ref%
{problem2}.

\begin{theorem}
\label{main-theorem} Let $\mathcal{H}$ be a semialgebraic bundle with
respect to $C^m_{loc}\left( \mathbb{R}^{2},\mathbb{R}^{D}\right) .$ If $\mathcal{%
H}$ has a section, then it has a semialgebraic section.
\end{theorem}

We give a quick sketch of the proof of Theorem \ref{main-theorem}.

By a change of coordinates and a partition of unity, we may localize the problem
to a small thin wedge
\begin{equation*}
\Gamma (c)=\{(x_{1},x_{2})\in \mathbb{R}^{2}:x_{1}\in \left[ 0,c\right]
,0 \leq x_{2}\leq x_1 \}. \end{equation*}

More precisely, it is enough to prove that $\mathcal{H}|_{\Gamma(c')}$ has a section for sufficiently small $c'$.

We may assume also that our bundle $\mathcal{H=}\left( H\left(
x_{1},x_{2}\right) \right) _{\left( x_{1},x_{2}\right) \in \Gamma \left(
c\right) }$ satisfies $H\left( \left( 0,0\right) \right) =\left\{ 0\right\} $%
.

We analyze what it means for a given $F=\left( F_{1},\cdots ,F_{D}\right)
\in C^m_{loc}\left( \mathbb{R}^{n},\mathbb{R}^{D}\right) $ with $J_{(0,0)}F=0$ to be a section of $%
\mathcal{H}$. Our analysis produces finitely many semialgebraic curves $%
\gamma _{1},\gamma _{2},\cdots ,\gamma _{s_{\max }}$ in $\Gamma \left(
c\right) $, and we find that $F$ is a section of $\mathcal{H}$ if and only if

\begin{itemize}
\item $F\left( x_{1},x_{2}\right) $ and its $x_{2}$-derivatives up to order $%
m$ satisfy finitely many linear equations on the $\gamma
_{s}$ and

\item $F$ satisfies finitely many linear equations on $\Gamma(c) \setminus
\left( \gamma _{1}\cup \cdots \cup \gamma _{s_{\max }}\right) .$
\end{itemize}

The curves $\gamma _{s}$ have the form $\gamma _{s}=\left\{ \left(
x,\psi _{s}\left( x\right) \right) :x\in \left[ 0,c\right] \right\} $
for semialgebraic functions $\psi _{1},\cdots ,\psi _{s_{\max }}$ of one
variable.

The heart of our proof is to use the above characterization to produce
finitely many linear equations and inequalities for unknown functions $\xi
_{sk}^{l}\left( x\right) $ of one variable ($l=0,\cdots ,m;k=1,\cdots
,D;s=1,\cdots ,s_{\max }$) with the following properties:

\begin{description}
\item[(A)] If $F=\left( F_{1},\cdots ,F_{D}\right) \in C^m_{loc}\left( \mathbb{R}%
^{2},\mathbb{R}^{D}\right) $ is a section of $\mathcal{H}$ then the
functions
\begin{equation}
\xi _{sk}^{l}\left( x_{1}\right) =\left. \partial _{x_{2}}^{l}F_{k}\left(
x_{1},x_{2}\right) \right\vert _{x_{2}=\psi _{s}\left( x_{1}\right) }
\label{eq3}
\end{equation}%
satisfy the above equations and inequalities for $x\in \left[ 0,c\right] $;
and conversely

\item[(B)] If \underline{semialgebraic} functions $\xi _{sk}^{l}\left(
x\right) $ satisfy the above equations and inequalities for $x\in \left[ 0,c%
\right] $, then for some small $c^{\prime }<c$ there exists a semialgebraic
section $F=\left( F_{1},\cdots ,F_{D}\right) $ of $\mathcal{H}|_{\Gamma {(c}%
^{\prime }{)}}$ such that (\ref{eq3}) holds for $x\in \left[ 0,c^{\prime }%
\right] $.
\end{description}

We can easily deduce Theorem \ref{main-theorem} from (A) and (B), as follows.

Because $\mathcal{H}|_{\Gamma \left( c \right) }$ has a section,
(A)\ tells us that the relevant equations and inequalities for the $\xi
_{sk}^{l}$ admit a solution.

Because all functions appearing in those equations and inequalities are
semialgebraic (except perhaps the unknowns $\xi _{sk}^{l}$), it follows
easily that we may take the $\xi _{sk}^{l}\left( x\right) $ to depend
semialgebraically on $x$. Thanks to (B), we obtain a semialgebraic section
of $\mathcal{H}|_{\Gamma \left( c'\right) }$, completing the proof
of Theorem \ref{main-theorem}. See Section \ref{proofmainlemma} for details.

Let us recall some of the literature regarding Problems \ref{problem1}, \ref%
{pr2}, \ref{problem2}. The literature on Whitney's extension problem goes back to the seminal
works of H.~Whitney \cite{Whitney1,Whitney2}, and includes fundamental contributions by G. Glaeser \cite{Glaeser}, Yu. Brudnyi and P. Shvartsman \cite{brudnyi1,brudnyi2,brudnyi3,Brudnyi4}, E. Bierstone, P. Milman,
and W. Paw{\l }ucki \cite{BMP03,BMP06,Bierstone5}, as well as our own
papers \cite{F05-J,F4,F05-Sh,F06,F07-L,F07-St,F09-Data-3,F09-omega,FIL13,FIL15-S1,FIL15-S2,FIL15-S3,FIL16,FIL16+}. In
the semialgebraic (and $o$-minimal) setting , the analogue of the classical Whitney extension
theorem is due to K. Kurdyka and W. Paw{\l }ucki \cite{kp} and A. Thamrongthanyalak \cite{tham}.

Problem \ref{problem1} in the setting of $C^1_{loc}\left( \mathbb{R}^{n}\right) $
was settled affirmatively by M. Aschenbrenner and A. Thamrongthanyalak \cite{at}.
Our results on Problem \ref{problem2} imply an affirmative solution for $%
C^m_{loc}\left( \mathbb{R}^{2}\right) $. For $C^m_{loc}\left( \mathbb{R}^{n}\right) $
with $m\geq 2$ and $n\geq 3$, Problems \ref{problem1}, \ref{pr2}, \ref{problem2} remain open.

The problem of deciding whether a (possibly underdetermined) system of
linear equations of the form (\ref{equation}) admits a $C^0_{loc}$ solution was
proposed by Brenner \cite{bren}, and Epstein-Hochster \cite{Hochster}. Two independent solutions
to this problem appear in Fefferman-Koll\'ar \cite{Feff-Kollar}. Fefferman-Luli \cite{cf-luli-dfq-cm} solved the
analogous problem for $C^m_{loc}$ $\left( m\geq 1\right) $. See also \cite{cf-luli-generator}.

Koll\'ar-Nowak \cite{kn-continuous} proved by example that an equation of the form (\ref{equation}) may fail to admit a solution by $C^0_{loc}$-rational functions,
even though $A_{1},\cdots ,A_{D}$ and $f$ are polynomials and a $C^0_{loc}$
solution $\left( F_{1},\cdots ,F_{D}\right) $ exists. They showed that $x_1^3x_2f_1 +(x_1^3-(1+x_3^2)x_2^3)f_2 = x_1^4$ has a continuous semialgebraic solution but no continuous rational solution $(f_1,f_2) \in C^0_{loc}(\mathbb{R}^3,\mathbb{R}^2)$. However, \cite{ath} shows that a semialgebraic $C^0_{loc}$
solution exists, and \cite{kn-continuous} shows that a solution by $C^0_{loc}$ semialgebraic
functions exists for Problems \ref{problem1} and \ref{pr2} posed over $%
\mathbb{R}^{2}$, again provided $A_{1},\cdots ,A_{D},f$ are polynomials.

A recent paper of Bierstone-Campesato-Milman \cite{bcm} shows that given a system of
equations (\ref{equation}) with semialgebraic data $A_{i}$, $f$, there exists a function $r: \mathbb{N} \rightarrow \mathbb{N}$ independent of $f$ such that if the system \eqref{equation} admits a $C^{r(m)}_{loc}$ solution, then it admits a semialgebraic $C^m_{loc}$ solution. The result of Bierstone-Campesato-Milman is more general than the
version stated  above; it applies to suitable $o$-minimal  structures.

\textbf{Acknowledgement. } We are grateful to Matthias Aschenbrenner, Edward Bierstone, Jean-Baptiste Campesato, Fushuai (Black) Jiang, Bo'az Klartag, J\'anos Koll\'ar, Pierre
Milman, Assaf Naor, Kevin O'Neill, Wies{\l }aw Paw{\l }ucki, and Pavel Shvartsman for their interest
and valuable comments. We would also like to thank the participants of the
11-th Whitney workshop for their interest in our work, and we thank
Trinity College Dublin, for hosting the workshop. The first author is supported by the Air Force Office of Scientific Research (AFOSR), under
award FA9550-18-1-0069, the National Science Foundation (NSF), under grant DMS-1700180, and the US-Israel 
Binational Science Foundation (BSF), under grant 2014055. The second author is supported by NSF Grant
DMS-1554733 and the UC Davis Chancellor’s Fellowship.

\section{Notation and Preliminaries}

\label{section-notation-prelim-old}

A function $f: \mathbb{R}^n \rightarrow \mathbb{R}$ is called a Nash function if it is real-analytic and semialgebraic.

Write $B(x,r)$ to denote the ball of radius $r$ about $x$ in $\R^n$.

The \underline{dimension} of a semialgebraic set $E\subset \mathbb{R}^{n}$
is the maximum of the dimensions of all the imbedded (not necessarily
compact) submanifolds of $\mathbb{R}^{n}$ that are contained in $E$.

We recall a few definitions from the Introduction.

Fix $m,n,D$, and a subset $E\subset \mathbb{R}^{n}$. For each $x \in E$, let
\begin{equation}
H\left( x\right) =f\left( x\right) +I\left( x\right) \subset \mathcal{R}%
_{x}^{D} \label{bunddle}
\end{equation}%
be given, where $f\left( x\right) \in \mathcal{R}_{x}^{D}$ and $I\left(
x\right) \subset \mathcal{R}_{x}^{D}$ is an $\mathcal{R}_{x}$-submodule.
Then the family
\begin{equation*}
\mathcal{H}=(H(x))_{x\in E}
\end{equation*}%
is called a \underline{bundle} over $E$. $H(x)$ is called the \underline{%
fiber} of $\mathcal{H}$ at $x$.

When $m,n,D$ are not clear from context, we speak of a ``bundle with
respect to $C^m_{loc}\left( \mathbb{R}^{n},\mathbb{R}^{D}\right) $".

If $\mathcal{H}$ is given by (\ref{bunddle}) and $E^{\prime }\subset E$, then
we write $\left. \mathcal{H}\right\vert _{E^{\prime }}$ to denote the bundle
$\left( H\left( x\right) \right) _{x\in E^{\prime }}$, and refer to it as \underline{the restriction of $\mathcal{H}$ to $E^\prime$}. If $\mathcal{H} = (H(x))_{x \in E}$ and $\mathcal{H}' = (H'(x))_{x \in E}$ are bundles, $\mathcal{H}'$ is called a \underline{subbundle} of $\mathcal{H}$ if $H'(x) \subset H(x)$ for all $x \in E$. We write $\mathcal{H}\supset \mathcal{H}'$ to denote that $\mathcal{H}'$ is a subbundle of $\mathcal{H}$.

What we called a ``bundle" in \cite{FL04} we now call a ``classical bundle".

The definition is as follows. Fix $m,n,D$. Let $E\subset \mathbb{R}^{n}$ be compact. A
\underline{classical bundle} over $E$ is a family $\mathcal{{H}}=\left( {H}%
\left( x\right) \right) _{x\in E}$ of (possibly empty) affine subspaces $%
{H}\left( x\right) \subset \mathcal{P}^{D}$, parametrized by the
points $x\in E$, such that each non-empty ${H}\left( x\right) $ has the
form
\begin{equation*}
{H}\left( x\right) =\vec{P}^{x}+\vec{I}\left( x\right)
\end{equation*}%
for some $\vec{P}^{x}\in \mathcal{P}^{D}$ and some $\mathcal{R}%
_x$-submodule $\vec{I}\left( x\right) $ of $\mathcal{P}^{D}$.

When $m,n,D$ are not clear from context, we speak of a ``classical bundle with respect to $C^m(\R^n,\R^D)$".

We remark again that our notion of bundle differs from the notion of bundles considered previously (e.g., \cite{FL04}). In the present version, we do not require that $E$ be compact and we require all the fibers $H(x)$ to be non-empty.

A \underline{section} of the bundle $\mathcal{H}$ is a vector-valued
function $F\in C_{loc}^{m}(\mathbb{R}^{n},\mathbb{R}^{D})$ such that $J_{x}F\in
H(x)$ for all $x\in E$. A \underline{section} of a classical bundle $\mathcal{H}$ is a vector-valued
function $F\in C^m(\mathbb{R}^{n},\mathbb{R}^{D})$ such that $J_{x}F\in
H(x)$ for all $x\in E$.

\section{Tools}

\subsection{Glaeser Refinements, Stable Glaeser Refinements}

Given a bundle $\mathcal{H}=(H(x))_{x\in E}$ for $C^m_{loc}(\mathbb{R}^{n},%
\mathbb{R}^{D})$ or a classical bundle $\mathcal{H}=(H(x))_{x \in E}$ for $C^m(\R^n,\R^D)$, we
define the \underline{Glaeser refinement} $\mathcal{H}^{\prime }=(H^{\prime
}(x))_{x\in E}$ as follows:

\begin{description}
\item[(GR)] Let $x_0 \in E$. A given $P_0 \in H(x_0)$ belongs to $H^{\prime }(x_0)$ if and only if the
following holds. Given $\epsilon > 0$, there exists $\delta >0$ such that
for all $x_1, \cdots, x_k \in B(x_0,\delta) \cap E$, where $k$ is a large enough constant depending only on $m$, $n$, and $D$, there exist $P_i \in
H(x_i)$ ($i = 1,\cdots, k$), such that
\begin{equation*}
\left|\partial^\alpha (P_i - P_j)(x_i) \right|\leq \epsilon |x_i -
x_j|^{m-|\alpha|},
\end{equation*}
for all $|\alpha|\leq m, 0 \leq i,j \leq k$.
\end{description}

A bundle or a classical bundle $\mathcal{H}$ is \underline{Glaeser stable} if $\mathcal{H}^{\prime
}=\mathcal{H}$.

Note that the Glaeser refinement $\mathcal{H}^{\prime }$ of $\mathcal{H}$
may have empty fibers, even if $\mathcal{H}$ has none. In that case, we know
that $\mathcal{H}$ has no sections. If $\mathcal{H}$ is a classical bundle, then so
is $\mathcal{H}'$. If $ \cH$ is a bundle and no fibers of $\cH'$ are empty, then $\cH'$ is a
bundle. Both for  bundles and for classical bundles, every section of
$\cH$ is a section of $\cH'$. (See  \cite{FL04}  for the case  of classical bundles;
the elementary proofs  carry over unchanged for bundles.) Note in
particular that if a given bundle $\cH$ has a section, then $\cH'$ has no
empty fibers, hence  $\cH'$ is a bundle and $\cH'$ has a section.

Starting from a classical bundle $\cH$, or a bundle $\cH$ with a section, we can perform iterated Glaeser
refinement to pass to ever smaller subbundles $\mathcal{H}^{\left( 1\right) }
$, $\mathcal{H}^{\left( 2\right) }$, etc., without losing sections. We set $%
\mathcal{H}^{\left( 0\right) }=\mathcal{H}$, and for $l\geq 0$, we set $%
\mathcal{H}^{\left( l+1\right) }=\left( \mathcal{H}^{\left( l\right)
}\right) ^{\prime }$. Thus, by an obvious induction on $l$, we have $%
\mathcal{H=\mathcal{H}}^{\left( 0\right) }\supset \mathcal{H}^{\left(
1\right) }\supset \cdots $, yet $\mathcal{H}$ and $\mathcal{H}^{\left(
l\right) }$ have the same sections for all $l\geq 0$.

If $\mathcal{H}=(H(x))_{x \in E}$ is a semialgebraic bundle with respect to $C^m_{loc}(\mathbb{R}^n,\R^D)$, by an obvious induction on $l$, we have $H^{(l)}(x)$ depends semialgebraically on $x$, where $\mathcal{H}^{(l)}=(H^{(l)}(x))_{x \in E}.$

In principle, each $\mathcal{H}^{\left( l\right) }$ can be computed from $%
\mathcal{H}$. We remark that iterated Glaeser refinement stabilizes after finitely many iterations (i.e. for a large enough integer $l^*$ determined by $m,n,D$, we have
$\mathcal{H}^{(l^*+1)}=\mathcal{H}^{(l^*)}$;
thus $\mathcal{H}^{(l^*)}$ is Glaeser stable. See \cite{FL04} for the case of classical
bundles;  the argument, which goes back to Glaeser \cite{Glaeser} and Bierstone-Milman-Paw\l{}ucki \cite{BMP03,BMP06}, applies unchanged for bundles. We call $\mathcal{H}^{(l^*)} $
the \underline{stable Glaeser refinement} of $\mathcal{H}$.)

The main results of \cite{FL04} give the following

\begin{theorem}\label{cm-glaeserthm} For a large enough integer constant $l_{\ast }$ determined by $m,n,$ and $D$,
the following holds. Let $\mathcal{H}$ be a classical bundle with
respect to $C^m\left( \mathbb{R}^{n},\mathbb{R}^{D}\right)$. Let $\mathcal{%
H}^{\left( 0\right) },\mathcal{H}^{\left( 1\right) },\mathcal{H}^{\left(
2\right) },\cdots $ be its iterated Glaeser refinements. Then $\mathcal{H}$
has a section if and only if $\mathcal{H}^{\left( l_{\ast
}\right) }$ has no empty fibers. Suppose $\mathcal{H}^{\left( l_{\ast
}\right) }$ has no empty fibers. Let $x_0 \in E$ and let $P_0$ belong to
the fiber of $\mathcal{H}^{\left( l_{\ast
}\right) }$ at $x_0$.  Then there exists a section $F$ of the
bundle $\mathcal{H}$, such that $J_{x_0}(F)=P_0$. Moreover, there exists a constant $k^\#$ depending only on $m,n,$ and $D$ such that the following holds: Suppose  $\mathcal{H}=(H(x))_{x \in E}$ is a Glaeser stable classical bundle. Assume the following holds for some constant $M>0$:
\begin{itemize}
    \item Given $x_1,\cdots x_{k^\#} \in E$, there exist polynomials $P_1,\cdots, P_{k^\#} \in \mathcal{P}^D$, with $P_i \in H(x_i)$ for $1 \leq i \leq k^\#$; $|\partial^\alpha P_i(x_i)|\leq M$ for all $|\alpha|\leq m, 1\leq i \leq k^\#$; and $|\partial^\alpha (P_i-P_j)(x_j)|\leq M|x_i-x_j|^{m-|\alpha|}$ for all $|\alpha| \leq m,1 \leq i,j\leq k^{\#}$.
\end{itemize}
 Then there exists $F\in C^m(\mathbb{R}^n,\mathbb{R}^D)$ with $\|F\|_{C^m(\mathbb{R}^n,\mathbb{R}^D)}\leq C(m,n, D)M$ and $J_x(F) \in H(x)$ for all $x \in E$.

\end{theorem}

\subsection{Puiseux Series}

We will use the following elementary result regarding semialgebraic functions. For a
proof, see \cite{hormander}.

\begin{lemma} Suppose $f:\mathbb{R}\rightarrow \mathbb{R}$ is semialgebraic.
Then there exists a polynomial $P\left( z,x\right) \not\equiv 0$ on $\mathbb{R}^2$ such that $%
P\left( f\left( x\right) ,x\right) \equiv 0$. Moreover, for each $x_{0}\in
\mathbb{R}$ there exists $\delta>0$ such that $f\left( x\right) $ for $x \in (x_0,x_0+\delta)$ is given by a convergent Puiseux series.\end{lemma}

\begin{corollary}\label{corollarytolemma_onevariablegrowth}
Let $F(x)$ be a semialgebraic function of one variable,
satisfying $|F(x)|=O(x^p)$ on $(0, c]$ for some given $p$. Then the
derivatives of $F$ satisfy
$|F^{(k)}(x)|=O(x^{p-k})$ on $(0, c']$ for some $c'$. Similarly, if $F(x)=o(x^p)$ for $x$  in $(0, c)$, then $F^{(k)}(x)=o(x^{p-k})$
for $x$ in $(0, c')$. More generally, $|F^{(k)}(x)|=O(|F(x)|/x^k)$ on $(0, c')$.
\end{corollary}

\begin{corollary}\label{corollary3.2}
Let $F$ be a semialgebraic function in $C^m_{loc}(\Omega_1)$, where
$\Omega_\delta=\{(x,y) \in \R^2:0\leq y \leq x< \delta\}$ for $\delta>0$. Then for  small
enough $\delta$, $F|_{\Omega_\delta}$ extends to a $C^m$  semialgebraic function on
$\R^2$.
\end{corollary}

\begin{proof}[Sketch of proof]
The result follows in one line from known results,
but we sketch an elementary proof.

Without loss of generality, we may suppose that $J_{(0,0)}F=0$.  Then
$\partial_{x_2}^k F(x_1,0)=o(x_1^{m-k})$ for $k\leq m$, hence
$\partial_{x_1}^l\partial_{x_2}^k F(x_1,0)=o(x_1^{m-k-l})$ for $0\leq k,l\leq m$.

We set $\tilde{F}(x_1,x_2)$ equal to the m-th degree Taylor polynomial of $x_2 \mapsto F(x_1,x_2)$ about $x_2=0$ for each fixed $x_1$. The  above  estimates
for derivatives of $F$ show that $\tilde{F}$ is $C^m$  on
$\tilde{\Omega}_{\delta}=\{(x_1,x_2): 0\leq -x_2 \leq x_1 \leq \delta\}$, and its
$x_2$-derivatives up to order $m$ agree with those of $F$  on the $x_1$-axis.
In particular, $J_{(0,0)}\tilde{F}=0$.

Similarly, we set $F^{\#}(x_1,x_2)$ equal to the m-th degree Taylor polynomial
of $x_2 \mapsto F(x_1,x_2)$ about $x_2=x_1$  for each fixed $x_1$.  Then $F^{\#}$ is
$C^m$ on
$\Omega^{\#}_\delta=\{(x_1,x_2):0\leq x_1 \leq x_2 \leq 2x_1 \leq 2\delta\}$, and its $x_2$-derivatives
up to order $m$ agree with those of $F$ on the line $x_1=x_2$. In particular,
$J_{(0,0)}F^{\#}=0$.

Setting $F^+=\begin{cases} F &\mbox{on } \Omega_\delta \\
\tilde{F} & \mbox{on } \tilde{\Omega}_{\delta} \\
F^{\#} & \mbox{on } \Omega^{\#}_\delta
\end{cases}$, we  see that $F^+$ is  a $C^m$ semialgebraic function on
$\{(x_1,x_2): x_1 \in [0, \delta], -x_1\leq x_2 \leq 2x_1\}, F^+=F$ on $\Omega_\delta$, and $J_{(0,0)}F^+=0$.

Next, let $\theta(t)$ be a $C^m$ semialgebraic function of one variable,
equal to 1 in $[0,1]$ and supported in $[-1,2]$. Then, for small enough
$\delta$, the function $F^{++}(x_1,x_2)=\theta(\frac{x_2}{x_1})\cdot F^+(x_1,x_2)$ for $x_1>0$, $F^{++}(x_1,x_2)=0$ otherwise, is a $C^m$ semialgebraic function on the disc $B(0, \delta)$ that agrees with our given $F$ on $\Omega_\delta$.

Finally, multiplying $F^{++}$ by a semialgebraic cutoff  function supported in a small disc about $(0,0)$ and equal to $1$ in a smaller disc, we
obtain a $C^m$ semialgebraic function  on $\R^2$ that agrees with $F$ on
$\Omega_\delta$ for small enough $\delta$.
\end{proof}

\subsection{Singularities of Semialgebraic Sets and Functions}

We recall a few standard properties of semialgebraic sets and functions.
\begin{itemize}
    \item Let $U \subset \R^n$ be an open semialgebraic set, and let 
$F:U\rightarrow\R^k$ be semialgebraic. Then there exists a semialgebraic subset $X 
\subset U$
of dimension less than $n$ (the ``singular set" of $F$) such that $F$ is 
real-analytic on $U\setminus X$. (See Chapter 8 in \cite{coste-book}.)
\item A zero-dimensional semialgebraic set is finite. A one-dimensional 
semialgebraic set is a union of finitely many real-analytic arcs and 
finitely many points. (See Chapter 2 in \cite{coste-book}.)
\end{itemize}

\subsection{Existence of Semialgebraic Selections}

For sets $X,Y$, we denote a map $\Xi$ from $X$ to the power set of $Y$ by $\Xi: X \rightrightarrows Y$ and call such $\Xi$ a set-valued map; a set-valued map $\Xi$ is semialgebraic if $\{(x,y): y \in \Xi(x) \}$ is a semialgebraic set. Let $E \subset\mathbb{R}^n$ and $\Xi :E \rightrightarrows\mathbb{R}^D$. A \underline{selection} of $\Xi$ is a map $f: E \rightarrow \mathbb{R}^D$ such that $f(x) \in \Xi(x) $ for every $x \in E$. We recall the following well-known result regarding semialgebraic selection (see, for example, \cite{book-semialgebraic}).

\begin{theorem}\label{semialgebraic-selection}
Let $\Xi:E \rightrightarrows \mathbb{R}^D$ be semialgebraic. If each $\Xi(x)$ is nonempty, then $\Xi$ has a semialgebraic selection.
\end{theorem}

\subsection{Growth of Semialgebraic Functions }

Recall from \cite{cf-luli-dfq-cm} the following result

\begin{lemma}[Growth Lemma]
\label{growthlemma}Let $E\subset \mathbb{R}^{n_{1}}$ and $E^{+}\subset
E\times \mathbb{R}^{n_{2}}$ be compact and semialgebraic, with $\dim E^+
\geq 1$. Let $A$ be a semialgebraic function on $E^{+}$. Then there exist an
integer $K\geq 1$, a semialgebraic function $A_{1}$ on $E$, and a compact
semialgebraic set $\underline{E}^{+}\subset E^{+}$, with the following
properties.

\begin{description}
\item[(GL1)] $\dim \underline{E}^{+}<\dim E^{+}$.

For $x\in E$, set $E^{+}\left( x\right) =\left\{ y\in \mathbb{R}%
^{n_{2}}:\left( x,y\right) \in E^{+}\right\} $ and $\underline{E}^{+}\left(
x\right) =\left\{ y\in \mathbb{R}^{n_{2}}:\left( x,y\right) \in \underline{E}%
^{+}\right\} $. Then, for each $x\in E$, the following hold.

\item[(GL2)] If $\underline{E}^{+}\left( x\right) $ is empty, then
\begin{equation*}
\left\vert A\left( x,y\right) \right\vert \leq A_{1}\left( x\right) \text{
for all }y\in E^{+}\left( x\right) .
\end{equation*}

\item[(GL3)] If $\underline{E}^{+}\left( x\right) $ is non-empty, then
\begin{equation*}
\left\vert A\left( x,y\right) \right\vert \leq A_{1}\left( x\right) \cdot
\left[ \dist\left( y,\underline{E}^{+}\left( x\right) \right) \right] ^{-K}%
\text{ for all }y\in E^{+}\left( x\right) \setminus \underline{E}^{+}\left(
x\right) .
\end{equation*}
\end{description}
\end{lemma}

The Growth Lemma follows easily from a special of a theorem of {\L }ojasiewicz
and Wachta \cite{wachta}, as explained in \cite{cf-luli-dfq-cm}. We thank W. Paw{\l }ucki for teaching us that
implication.

We will apply the Growth Lemma to prove the following.

\begin{lemma}
\label{semialgebraic-function-lemma} 
Let $F\left( x,y\right) $ be a bounded semialgebraic function on $\left[ -1,1%
\right] \times (0,1],$ and suppose that
\begin{equation}
\lim_{y\rightarrow 0^{+}}F\left( x,y\right) =0\text{ for each }x\in \left[
-1,1\right] \text{.}  \label{lemma3.2.a}
\end{equation}%
Then there exist a positive integer $N$ and a semialgebraic function $%
A\left( x\right) $ on $\left[ -1,1\right] $ such that
\[
F\left( x,y\right) \leq A\left( x\right) y^{\frac{1}{N}}\text{ for all }%
\left( x,y\right) \in \left[ -1,1\right] \times (0,1]\text{.}
\]
\end{lemma}

\begin{proof}
It is enough to show that for some positive integer $N$ we have
\begin{equation}
\sup_{y\in (0,1]}\frac{\left\vert F\left( x,y\right) \right\vert }{y^{1/N}}%
<\infty \text{ for all }x\in \left[ -1,1\right] \text{,}  \label{lemma3.2.b}
\end{equation}%
for we may then set $A\left( x\right) =\sup_{y\in (0,1]}\frac{\left\vert
F\left( x,y\right) \right\vert }{y^{1/N}}$, and $A\left( x\right) $ will
depend semialgebraically on $x$.

For each fixed $x$, the function $y\mapsto F\left( x,y\right) $ is bounded
and given near $\left( 0,0\right) $ by a convergent Puiseux series that tends
to zero as $y\rightarrow 0^{+}$. Hence, for some positive integer $N_{x}$ we
have
\begin{equation}
\sup_{y\in (0,1]}\frac{\left\vert F\left( x,y\right) \right\vert }{%
y^{1/N_{x}}}<\infty \text{.}  \label{3.2.c}
\end{equation}%
Our task is to show that $N_{x}$ may be taken independent of $x.$ Thanks to (%
\ref{3.2.c}), we may exclude from consideration any given finite set of
\textquotedblleft bad" $x\in \left[ -1,1\right] $.

We recall our main hypothesis (\ref{lemma3.2.a}). For each $\left(
x,\varepsilon \right) \in \left[ -1,1\right] \times (0,1]$ there exists $%
\delta \in (0,1]$ such that $\left( x,\varepsilon ,\delta \right) $ belongs
to the semialgebraic set
\[
\left\{ \left( x,\varepsilon ,\delta \right) \in \left[ -1,1\right] \times
(0,1]\times (0,1]:\left\vert F\left( x,y\right) \right\vert \leq \varepsilon
\text{ for all }y\in (0,\delta ]\right\} .
\]%
Hence, there exists a semialgebraic function $\delta \left( x,\varepsilon
\right) $ mapping $\left[ -1,1\right] \times (0,1]$ into $(0,1]$ such that
\begin{equation}
\left\vert F\left( x,y\right) \right\vert \leq \varepsilon \text{ for }y\in
(0,\delta \left( x,\varepsilon \right) ],x\in \left[ -1,1\right]
,\varepsilon \in (0,1].  \label{ttt3}
\end{equation}

We set $\delta \left( x,0\right) =1$ for $x\in \left[ -1,1\right] $. Then $%
\delta :\left[ -1,1\right] \times \left[ 0,1\right] \rightarrow (0,1]$ is
semialgebraic and satisfies (\ref{ttt3}).

We now apply Lemma \ref{growthlemma} to the function $\frac{1}{\delta \left( x,\varepsilon
\right) }$.

Thus, we obtain a semialgebraic set $\underline{E}\subset \left[ -1,1\right]
\times \left[ 0,1\right] $, a positive integer $N,$ and a positive
semialgebraic function $\underline{\delta }\left( x\right) $ on $\left[ -1,1%
\right] $, with the following properties.
\begin{itemize}
\item $\dim \underline{E}\leq 1$.
\item For $x\in \left[ -1,1\right] $, let $\underline{E}\left( x\right) =\left\{
\varepsilon :\left( x,\varepsilon \right) \in \underline{E}\right\} $.
\end{itemize} 
Then%
\begin{equation}
\delta \left( x,\varepsilon \right) \geq \underline{\delta }\left( x\right)
\text{ (all }\varepsilon >0\text{) if }\underline{E}=\emptyset   \label{ttt4}
\end{equation}

and%
\begin{equation}
\delta \left( x,\varepsilon \right) \geq \underline{\delta }\left( x\right)
\cdot \left[ \dist\left( \varepsilon ,\underline{E}\left( x\right) \right) %
\right] ^{N}\text{ (all }\varepsilon \not\in \underline{E}(x)\text{) if }%
\underline{E}\not=\emptyset \text{.}  \label{ttt5}
\end{equation}%
Because $\dim \underline{E}\leq 1,$ there are at most finitely many $x\in %
\left[ -1,1\right] $ for which $\underline{E}\left( x\right) $ is infinite.

As explained above, we may discard those \textquotedblleft bad" $x$, it
is enough to prove (\ref{lemma3.2.b}) for all $x$ such that $\underline{E}%
\left( x\right) $ is finite.

From now on, we restrict attention to \textquotedblleft good" $x,$ i.e.,
those $x$ for which $\underline{E}\left( x\right) $ is finite.

Set
\[
\underline{\mathcal{\varepsilon }}\left( x\right) =\left\{
\begin{array}{l}
\frac{1}{2}\min \left( \underline{E}\left( x\right) \setminus \left\{
0\right\} \right)  \\
1%
\end{array}%
\right.
\begin{array}{l}
\text{if }\underline{E}\left( x\right) \text{ contains points other than }0
\\
\text{otherwise}%
\end{array}%
\text{.}
\]%
So $\underline{\mathcal{\varepsilon }}\left( x\right) >0$ for all
\textquotedblleft good" $x$.

If $\underline{E}\left( x\right) \not=\emptyset $, then $\dist\left(
\varepsilon ,\underline{E}\left( x\right) \right) \geq \varepsilon $ for $%
0<\varepsilon \leq \underline{\mathcal{\varepsilon }}\left( x\right) $,
hence (\ref{ttt5}) gives
\begin{equation}
\delta \left( x,\varepsilon \right) \geq \underline{\delta }\left( x\right)
\varepsilon ^{N}\text{ for }0<\varepsilon \leq \underline{\varepsilon }%
\left( x\right) \text{.}  \label{ttt6}
\end{equation}%
If instead $\underline{E}\left( x\right) =\emptyset $, then because $%
\underline{\mathcal{\varepsilon }}\left( x\right) =1,$ (\ref{ttt4}) again
gives (\ref{ttt6}). Thus, (\ref{ttt6}) holds in all cases.

Now suppose $0<y<\underline{\delta }\left( x\right) \cdot \left( \underline{%
\varepsilon }\left( x\right) \right) ^{N}$.

Then, setting $\varepsilon =\left( \frac{y}{\underline{\delta }\left(
x\right) }\right) ^{1/N}$ and applying (\ref{ttt6}), we find that $\delta
\left( x,\varepsilon \right) \geq y.$ The defining property of $\delta
\left( x,\varepsilon \right) $ therefore tells us that
\[
\left\vert F\left( x,y\right) \right\vert \leq \varepsilon =\left( \frac{y}{%
\underline{\delta }\left( x\right) }\right) ^{1/N}\text{.}
\]

Thus, for any \textquotedblleft good" $x,$ we have shown that
\begin{equation}
\frac{\left\vert F\left( x,y\right) \right\vert }{y^{1/N}}\leq \left(
\underline{\delta }\left( x\right) \right) ^{-1/N}\text{ for }0<y<\underline{%
\delta }\left( x\right) \cdot \left( \underline{\varepsilon }\left( x\right)
\right) ^{N}\text{.}  \label{ttt7}
\end{equation}%
On the other hand, recall that $F$ is bounded; say, $\left\vert F\left(
x,y\right) \right\vert \leq M$ for all $\left( x,y\right) \in \left[ -1,1%
\right] \times (0,1]$.

Hence,
\begin{equation}
\frac{\left\vert F\left( x,y\right) \right\vert }{y^{1/N}}\leq \frac{M}{%
\left( \underline{\delta }\left( x\right) \right) ^{1/N}\underline{%
\varepsilon }\left( x\right) }\text{ for }\underline{\delta }\left( x\right)
\cdot \left( \underline{\varepsilon }\left( x\right) \right) ^{N}\leq y\leq 1%
\text{.}  \label{ttt8}
\end{equation}%
Our desired estimate (\ref{lemma3.2.b}) is now immediate from (\ref{ttt7}) and
(\ref{ttt8}).

The proof of Lemma \ref{semialgebraic-function-lemma} is complete.
\end{proof}

Similar ideas can be used to prove an $n$-dimensional version of Lemma \ref{semialgebraic-function-lemma},
but we don't discuss it here.

\subsection{Logarithmic Derivatives of Semialgebraic Functions}
\label{seciton_log-derivative}
Let $V$ be a semialgebraic subset of $\mathbb{R}^n \times \mathbb{R}^m$.
Given $x \in \mathbb{R}^n$, we write $V(x)$ to denote the set of all $t \in
\mathbb{R}^m$ such that $(x,t) \in V$. Given $(x,t) \in \mathbb{R}^n \times
\mathbb{R}^m$, we write $\delta_V(x,t)$ to denote the distance from $t$ to $%
V(x)$. We take $\delta_V(x,t) = + \infty$ if $V(x)$ is empty. For a smooth
function $F(x,t)$ on $\mathbb{R}^n \times \mathbb{R}^m$, we write $\nabla_t
F(x,t)$ to denote the gradient of the function $t \mapsto F(x,t)$.

The following theorem is proven by A. Parusinski in \cite{AP1,AP2}. We
thank Edward Bierstone, Jean-Baptiste Campesato, Pierre Milman, and Wieslaw Paw{\l }ucki for pointing out the references, and thus helping us
remove 10 pages from our paper.

\begin{theorem}
\label{th-log} \label{theorem1} Let $F(x,t)$ be a (real-valued) subanalytic
function of $(x,t)\in \mathbb{R}^{n}\times \mathbb{R}^{m}$. Then there
exist a closed codimension 1 subanalytic set $V\subset \mathbb{R}^{n}\times \mathbb{%
R}^{m}$ and a constant $C>0$ such that outside $V$ the function $F$ is
smooth and moreover,
\begin{equation}
|\nabla _{t}F(x,t)|\leq C\frac{\left\vert F\left( x,t\right) \right\vert }{%
\delta _{V}\left( x,t\right) }\text{.}  \label{eq1}
\end{equation}
If $F$ is semialgebraic, then we can take $V$ to be semialgebraic.
\end{theorem}

As a special case of Theorem \ref{th-log}, we have the following.

\begin{theorem}
\label{log-derivative-theorm} Let $F\left( x\right) $ be a semialgebraic
function on $\mathbb{R}^{n}$. Then there exist a closed semialgebraic $%
V\subset \mathbb{R}^{n}$ of dimension at most $\left( n-1\right) $, and a
constant $C$, such that $F$ is $C^m_{loc}$ outside $V$, and
\begin{equation*}
\left\vert \nabla F\left( x\right) \right\vert \leq C\left\vert F\left(
x\right) \right\vert \cdot \left[ \text{dist}\left( x,V\right) \right] ^{-1}
\end{equation*}%
for $x\in \mathbb{R}^{n}\setminus V$.
\end{theorem}

\subsection{Variant of Helly's Theorem}

\label{hltv}

We recall the following result from convex geometry. Surely more precise versions of the result are well known,
but we had trouble tracking down a reference so we will provide a proof.

\begin{theorem}[Helly's Theorem Variant]
\label{helly-theorem} Let $(p_{\omega })_{\omega \in \Omega }$ be a family
of seminorms on a vector space $V$ of dimension $D$. Assume that $%
\sup_{\omega \in \Omega }p_{\omega }(v)<\infty $ for every $v\in V$. Then
there exist $\omega _{1},\cdots ,\omega _{L}\in \Omega $, with $L$ depending
only on $D $, such that
\begin{equation*}
\sup_{\omega \in \Omega }p_{\omega }(v)\leq C\cdot \max\{p_{\omega
_{1}}(v),\cdots ,p_{\omega _{L}}(v)\}\text{ for all }v\in V,
\end{equation*}%
with $C$ also depending only on $D$.
\end{theorem}

We use the following variant of the classical Helly theorem (see Section 3 in \cite{F4}) from
elementary convex geometry.
\begin{lemma}\label{classical-helly}
Let $(K_\omega)_{\omega \in \Omega}$ be a collection of compact
convex symmetric subsets of $\R^D$. Suppose the intersection of all the $K_\omega$ has
nonempty interior. Then there exist $\omega_1,\cdots,\omega_L$ such that
$K_{\omega_1} \cap \cdots \cap K_{\omega_L} \subset
C\cdot\bigcap_{\omega\in \Omega} K_\omega$, where $C$ and $L$ depend only
on $D$.
\end{lemma}

The proof of the ``Lemma on Convex Sets" in Section 3 of \cite{F4} applies here
and proves Lemma \ref{classical-helly}, even though our present hypotheses differ
slightly from those of \cite{F4}.

We apply Lemma \ref{classical-helly} to prove Theorem \ref{helly-theorem}.

\begin{proof}[Proof of Theorem \ref{helly-theorem}]
Suppose first that each $p_\omega$ is a norm, not just a seminorm. Then
the conclusion of Theorem \ref{helly-theorem} follows by applying Lemma \ref{classical-helly} to the
family of convex sets $K_\omega=\{v\in V: p_\omega(v) \leq 1\}$, ${\omega\in
\Omega}$.

Now  suppose each $p_\omega$  is a seminorm. Let $H(\omega)=\{v\in
V:p_\omega(v)=0\}$, and let $H$ be the intersection of all the $H(\omega)$.
Each $H(\omega)$ is a vector subspace of $V$. Consequently there exist
$\lambda_1,\cdots,\lambda_s \in \Omega$, with $s \leq D$, such that $H=H(\lambda_1) \cap\cdots \cap H(\lambda_s)$.

For $\omega\in \Omega$ and $v\in V$, set
$p^*_\omega(v)=p_{\lambda_1}(v)+\cdots+p_{\lambda_s}(v)+p_{\omega}(v)$.
Then $p^*_\omega$ is a seminorm on $V$, and $p^*_\omega(v)=0$ if and only if $v\in H$.
Regarding each $p^*_\omega$ as a norm on $V/H$, and applying Theorem \ref{helly-theorem} for
collections of norms, we complete the proof of Theorem \ref{helly-theorem}. \end{proof}

\section{Preliminary Reductions}

\label{prelim}

The purpose of this section is to reduce Theorem \ref{main-theorem} to the following:

\begin{lemma}[Main Lemma]
\label{main-proposition} Let $\mathcal{H}=(H(x))_{x\in \mathbb{R}^{2}}$ be a
semialgebraic bundle for $C^m_{loc}(\mathbb{R}^{2},\mathbb{R}^{D})$. Assume $%
\mathcal{H}$ is Glaeser stable. Assume $H(0)=\{0\}$.
Then, for small enough $c>0$, $\mathcal{H}|_{\Gamma (c)}$ has a
semialgebraic section, where $\Gamma (c)=\{(x_{1},x_{2})\in \mathbb{R}^{2}:x_{1}\in \left[ 0,c\right]
,0 \leq x_{2}\leq x_1 \}.$
\end{lemma}

To deduce Theorem \ref{main-theorem} from
Lemma \ref{main-proposition} we argue as follows.

Suppose we are given a Glaeser stable bundle $\mathcal{H}=(H(x))_{x \in \mathbb{R}^2}$ for $C^m_{loc}(\mathbb{R}^2,\mathbb{R}^D)$ with $H(x) \subset
\mathcal{P}^D$ depending semialgebraically on $x$. Assume $H(0)=\{0\}$.

Let $\Gamma (c)=\{(x_{1},x_{2})\in \mathbb{R}^{2}:x_{1}\in \left[ 0,c\right]
,0 \leq x_{2}\leq x_1 \}$. Theorem \ref{cm-glaeserthm} tells us that $\mathcal{H}|_{\Gamma(c)}$ has a section $F_c$.
The main lemma asserts that for $c$ small enough $\mathcal{H}|_{\Gamma(c)}$ has a
semialgebraic section.

We will cover a full neighborhood of $0$ by rotating wedges of the form $\Gamma(c)$. Using a partition of unity subordinate to the cover and the fact that $H(0)=\{0\}$, we can then patch
together sections of $\mathcal{H}$, and obtain a semialgebraic section over
a full neighborhood of $0$.

We may drop the restriction $H(0)=\{0\}$, because without loss of generality
our given section $F_c$ has jet $0$ at the origin, so we may just cut down $H(0)$
to $\{0\}$. We can also drop the restriction that $\mathcal{H}$ is Glaeser
stable (assuming $\mathcal{H}$ has a section) since we can always pass to the stable Glaeser refinement. Thus, any
semialgebraic bundle having a section has a semialgebraic section over some
neighborhood of $0$. We can use compactness and a partition of unity to
conclude that $\mathcal{H}$ admits a semialgebraic section over any given
compact set.

\begin{lemma}
\label{lemmma-prelim-reduction} Suppose $H(z)$ depends semialgebraically on $%
z\in \mathbb{R}^{2}$. If $\mathcal{H}=(H(z))_{z\in \mathbb{R}^{2}}$ has a
section, then $\mathcal{H}$ has a section $F\in C^m_{loc}(\mathbb{R}^{2},\mathbb{%
R}^{D})$ such that for all $|\alpha |\leq m$, $|\partial ^{\alpha }F(x)|\leq
C(1+|x|)^{K}$ on $\mathbb{R}^{2}$, for some $C$ and $K$.
\end{lemma}

\begin{proof}
To prove this lemma, we may assume that $\mathcal{H}$ is Glaeser stable.

Taking $E_{R}=\left\{ x\in \mathbb{R}^{2}:\left\vert x\right\vert \leq
R\right\} $ with $R\geq 1$, and applying Theorem \ref{cm-glaeserthm}, we obtain a
section $F_{R}$ of $\mathcal{H}|_{E_{R}}$, with $\left\vert \left\vert
F_{R}\right\vert \right\vert _{C^m}\leq C\left( R\right) ^{K}$, because
the \textquotedblleft $M$ \textquotedblright\ in the result quoted above
applied to $\mathcal{H}|_{E_{R}}$ can be taken to depend semialgebraically on $R$. (That's
where we use the fact that the bundle $\mathcal{H}$ is semialgebraic.)

We can now easily use a partition of unity to patch together $F_{2^{k}}$, $%
k=1,2,3,\cdots $, into a section $F$ as in the conclusion of Lemma \ref%
{lemmma-prelim-reduction}.
\end{proof}

Fix $K$ as in the conclusion of Lemma \ref{lemmma-prelim-reduction}. Let $\Phi :\text{ Open Disc }\Delta \rightarrow \mathbb{R}^{2}$ be a
semialgebraic diffeomorphism, for example, $\Phi(x) = \frac{x}{1-|x|^2}$. Let $\theta (x)>0$ be a semialgebraic
function on $\mathbb{R}^{2}$ that tends to zero so rapidly that
\begin{equation*}
\partial ^{\alpha }[(\theta F)\circ \Phi ](y)\rightarrow 0\text{, for all }%
|\alpha |\leq m\text{ as }y\rightarrow \partial \Delta ,
\end{equation*}%
whenever $|\partial ^{\alpha }F(x)|\leq C(1+|x|)^{K}$ on $\mathbb{R}^{2}$, $%
|\alpha |\leq m$.

We can now form a bundle $\mathcal{H}^{*}$ as follows: For $x$ in $\Delta$, the fiber $H^{*}(x)$ consists of all $J_x((\theta F)\circ \Phi)$ for sections $F$ of the bundle $\mathcal{H}$.

The fibers of $\mathcal{H}^*$ over points not in $\Delta$
are $\{0\}$.

Then $\mathcal{H}^*$ is a semialgebraic bundle admitting a section.

We have seen that semialgebraic bundles with sections have
semialgebraic sections over any compact set. In particular, $\mathcal{H}^*$ has a semialgebraic section $\mathcal{F}$
over $\Delta^\text{closure}$. Then $\frac{\mathcal{F}\circ \Phi^{-1}(x)}{%
\theta(x)}$ is a semialgebraic section of $\mathcal{H}$ over $\mathbb{R}^2$.

Consequently, we can deduce Theorem \ref{main-theorem} from Lemma \ref{main-proposition}.

The rest of the paper is devoted to the proof of Lemma \ref{main-proposition}.

\section{Characterization of Sections}\label{characterization-section-section}

\subsection{Semialgebraic Bundles}

\label{section-semialgebraic-bundle}

Fix $U \subset \mathbb{R}^n$ open, semialgebraic. Fix $\psi: U \rightarrow
\mathbb{R}^k$ Nash. Let $\hat{\psi}(x) =
(x,\psi(x)) \in \mathbb{R}^n \times \mathbb{R}^k$ for $x \in U$. We set $%
\hat{U} = \hat{\psi}(U)$. Let $\mathcal{P}$ denote the vector space of polynomials
of degree at most $m$ on $\mathbb{R}^n \times \mathbb{R}^k$. We write $z=(x,y)$ to denote a point of $\R^n \times \R^k$. We write $%
\mathcal{R}_z$ to denote the ring
obtained from $\mathcal{P}$ by multiplication of $m$-jets at $z$. We fix a
bundle $\mathcal{H}= (H(z))_{z \in \hat{U}}$, where, for each $z =
\hat{\psi}(x) \in \hat{U}$ we have $H(z)=f^x+I(x)$, $f^x \in \mathcal{P}^D$, $%
I(x)$ an $\mathcal{R}_{\hat{\psi}(x)}$-submodule of $\mathcal{P}^D$. (We point out that $\mathcal{H}$ is a bundle, not a classical bundle, see Remark \ref{classicalbundle}.)

We suppose $\mathcal{H}$ is Glaeser stable. We assume that $H(z)$ depends
semialgebraically on $z \in \hat{U}$. (We sometimes abuse notion by writing $%
I(z)$ for $I(x)$, where $z =\hat{\psi}(x)$.)

Under the above assumptions and definitions, we will prove the following
result.

\begin{lemma}
\label{semialgebraic-bundle-lemma}

There exist a semialgebraic set $U_{\text{bad}} \subset \mathbb{R}^n$ of
dimension less than $n$; Nash functions $%
A_{j\beta}^i,G^i$ on $U \setminus U_{\text{bad}}$ ($i=1,\cdots, i_{\max}, j
=1,\cdots, D, \beta$ a multiindex of order $\leq m$ for $\mathbb{R}^k$) with
the following property. Let $B \subset U \setminus U_{\text{bad}}$ be a
closed ball. Set $\hat{B}=\hat{\psi}(B)$. Let $F = (F_1, \cdots, F_D) \in
C^m_{loc} (\mathbb{R}^n \times \mathbb{R}^k, \mathbb{R}^D)$. Then $F$ is a section
of $\mathcal{H}|_{\hat{B}}$ if and only if $\sum_{|\beta| \leq
m}\sum_{j=1}^D A_{j\beta}^i(x) \cdot (\partial_y^\beta F_j(x,\psi(x)))=G^i (x)$ for all
$x \in B$ (each $i$).
\end{lemma}

\begin{proof}
We may suppose that $f^x$ and $I(x)$ depend semialgebraically on $x \in U$.
We write $f^x= (f_1^x,\cdots, f_D^x)$ and $\psi(x)=(\psi_1(x),\cdots,%
\psi_k(x))\quad (x \in U)$.

For $l=1, \cdots, n$, we introduce the vector field
\begin{equation*}
X_l = \frac{\partial}{\partial x_l} + \sum_{p=1}^k \frac{\partial \psi_p(x)}{%
\partial x_l} \frac{\partial}{\partial y_p} \text{on } U \times \mathbb{R}^k.
\end{equation*}

On $U \times \mathbb{R}^k$, then $X_l$ are Nash,
and $[X_l,X_{l^{\prime }}]=0$. For $\alpha= (\alpha_1,\cdots, \alpha_n)$, we
write $X^\alpha = X_1^{\alpha_1} \cdots X_n^{\alpha_n}$.

The $X_{1},\cdots ,X_{n}$, $\frac{\partial }{\partial y_{1}},\cdots ,\frac{%
\partial }{\partial y_{k}}$ form a frame on $U\times \mathbb{R}^{k}$.
Because $I\left( x\right) $ depends semialgebraically on $x\in U$, we may
express

\begin{itemize}
\item[\refstepcounter{equation}\text{(\theequation)}\label{sb-1}] $I\left(
x\right) =\left\{ \left( P_{1},\cdots ,P_{D}\right) \in \mathcal{P}%
^{D}:\left. \sum_{\substack{ \left\vert \alpha \right\vert +\left\vert \beta
\right\vert \leq m  \\ j=1,\cdots ,D}}\tilde{A}_{j\alpha \beta }^{i}\left(
x\right) \left( X^{\alpha }\partial _{y}^{\beta }P_{j}\right) \right\vert _{%
\tilde{\psi}\left( x\right) }=0\text{, for }i=1,\cdots ,i_{\max }\right\} $
for semialgebraic $\tilde{A}_{j\alpha \beta }^{i}$ on $U$.
\end{itemize}

We take $U_{\text{bad}}^{1}$ to be the union of the singular sets of the $%
\tilde{A}_{j\alpha \beta }^{i}$. Then $U_{\text{bad}}^{1}$ is a
semialgebraic set of dimension $<n$ in $\mathbb{R}^{n}$, and the $\tilde{A}%
_{j\alpha \beta }^{i}$ are real-analytic on $U\setminus U_{\text{bad}}^{1}$.

We may therefore rewrite the equation in \eqref{sb-1} in the form
\begin{equation*}
\left. \sum_{\substack{ \left\vert \alpha \right\vert +\left\vert \beta
\right\vert \leq m  \\ j=1,\cdots ,D}}\left( X^{\alpha }\left\{ A_{j\alpha
\beta }^{i}\left( x\right) \partial _{y}^{\beta }P_{j}\right\} \right)
\right\vert _{\hat{\psi}\left( x\right) }=0\text{.}
\end{equation*}%
The $A_{j\alpha \beta }^{i}$ are Nash on $%
U\setminus U_{\text{bad}}^{1}$. Thus, for any closed ball $B\subset
U\setminus U_{\text{bad}}^{1}$ the following holds. (We set $\hat{B}=\hat{%
\psi}\left( B\right) $.)

A given $F=\left( F_{1},\cdots ,F_{D}\right) \in C^m_{loc}\left( \mathbb{R}^{n}\times
\mathbb{R}^{k},\mathbb{R}^{D}\right) $ is a section of $\left( I\left(
z\right) \right) _{z\in \hat{B}}$ if and only if
\begin{equation*}
\sum_{\left\vert \alpha \right\vert \leq m}X^{\alpha }\left\{
\sum_{\left\vert \beta \right\vert \leq m-\left\vert \alpha \right\vert
}A_{j\alpha \beta }^{i}\left( x\right) \partial _{y}^{\beta }F_{j}\left(
x,y\right) \right\} =0\text{ on }\hat{B}\text{ for all }i\text{.}
\end{equation*}

We look for integers $s\geq 0$ for which there exist Nash functions $A_{j\alpha \beta }^{i}$ on $U\setminus U_{\text{bad}%
}^{1}$ with the following property (\textquotedblleft Property $\prod \left(
s\right) $"):

Let $B\subset U\setminus U_{\text{bad}}^{1}$ be a closed ball; set $\hat{B}=%
\hat{\psi}\left( B\right) $. Then $\left( F_{1},\cdots ,F_{D}\right) \in
C^m_{loc}\left( \mathbb{R}^{n}\times \mathbb{R}^{k},\mathbb{R}^{D}\right) $ is a
section of $\left( I\left( z\right) \right) _{z\in \hat{B}}$ if and only if
\begin{equation}
\sum_{\left\vert \alpha \right\vert \leq s}X^{\alpha }\left\{
\sum_{\left\vert \beta \right\vert \leq m-\left\vert \alpha \right\vert
}\sum_{j=1}^{D}A_{j\alpha \beta }^{i}\left( x\right) \partial _{y}^{\beta
}F_{j}\left( x,y\right) \right\} =0\text{ on }\hat{B}\text{ for all }i\text{.%
}  \label{sb2}
\end{equation}

We have seen that we can achieve Property $\prod \left( m\right) $.

\begin{claim}
\label{sb-claim} Let $s$ be the smallest possible integer $\geq 0$ for which
we can achieve Property $\prod \left( s\right) $, and let $A_{j\alpha \beta
}^{i}$ be as in Property $\prod \left( s\right) $. Then $s=0$. In other
words, Property $\prod (0)$ holds.
\end{claim}

\begin{proof}[Proof of Claim \protect\ref{sb-claim}]
Assuming $s \geq 1$, we will achieve Property $\prod (s-1)$, contradicting
the fact that $s$ is as small as possible.

Fix $B \subset U \setminus U_{\text{bad}}^1$ a closed ball, and let $%
(F_1,\cdots,F_D) \in C^m_{loc}(\mathbb{R}^n \times \mathbb{R}^k, \mathbb{R}^D)$ be
a section of $(I(z))_{z \in \hat{B}}$. (As always, $\hat{B}=\psi(B)$.) Fix $%
x_0 \in B$ and fix a multiindex $\alpha_0$ with $|\alpha_0| =s$. For $%
j=1,\cdots, D$, define functions on $\mathbb{R}^n \times \mathbb{R}^k$ by
setting $F_j^{\#}(z) = \theta \cdot F_j (z)$ where $\theta \in C_0^\infty(%
\mathbb{R}^n \times \mathbb{R}^k)$ with jet $(J_{\hat{\psi}(x_0)}\theta
)(x,y) = (x-x_0)^{\alpha_0}$.

Then $(F_1^{\#}, \cdots, F_D^{\#}) \in C^m_{loc}(\mathbb{R}^n \times \mathbb{R}^k,
\mathbb{R}^D)$ is a section of $(I(z))_{z \in \hat{B}}$ because each $I(z)$
is an $\mathcal{R}_z$-submodule of $\mathcal{R}_z^D$.

Applying Property $\prod (s)$ to $(F_{1}^{\#},\cdots ,F_{D}^{\#})$, we learn
that%
\begin{equation*}
\left. \sum_{\left\vert \beta \right\vert \leq m-\left\vert \alpha
_{0}\right\vert }\sum_{j=1}^{D}A_{j\alpha _{0}\beta }^{i}\left( x_{0}\right)
\left( \partial _{y}^{\beta }F_{j}\right) \right\vert _{\hat{\psi}\left(
x_{0}\right) }=0\text{ }\left( \text{all }i\right) \text{.}
\end{equation*}%
This holds for all $x_0$ and for all $\left\vert \alpha _{0}\right\vert =s$.
Thus, if $\left( F_{1},\cdots ,F_{D}\right) \in C^m_{loc}\left( \mathbb{R}%
^{n}\times \mathbb{R}^{k},\mathbb{R}^{D}\right) $ is a section of $\left(
I\left( z\right) \right) _{z\in \hat{B}}$, then
\begin{equation}
\sum_{\left\vert \beta \right\vert \leq m-\left\vert \alpha \right\vert
}\sum_{j=1}^{D}A_{j\alpha \beta }^{i}\left( x\right) \partial _{y}^{\beta
}F_{j}\left( x,y\right) =0  \label{sb3}
\end{equation}%
on $\hat{B}$ for all $\left\vert \alpha \right\vert =s$ and for all $i$.
Because the $X_{j}$ are tangent to $\hat{B}$, it follows from \eqref{sb3}
that
\begin{equation}
X^{\alpha }\left\{ \sum_{\left\vert \beta \right\vert \leq m-\left\vert
\alpha \right\vert }\sum_{j=1}^{D}A_{j\alpha \beta }^{i}\left( x\right)
\partial _{y}^{\beta }F_{j}\left( x,y\right) \right\} =0  \label{sb4}
\end{equation}%
on $\hat{B}$ for all $\left\vert \alpha \right\vert =s$ and for all $i$.
From \eqref{sb2} and \eqref{sb4}, we conclude that
\begin{equation}
\sum_{\left\vert \alpha \right\vert \leq s-1}X^{\alpha }\left\{
\sum_{\left\vert \beta \right\vert \leq m-\left\vert \alpha \right\vert
}\sum_{j=1}^{D}A_{j\alpha \beta }^{i}\left( x\right) \partial _{y}^{\beta
}F_{j}\left( x,y\right) \right\} =0  \label{sb5}
\end{equation}%
on $\hat{B}$ for all $i$. Thus, any section of $\left( I\left( z\right)
\right) _{z\in \hat{B}}$ satisfies \eqref{sb3} and \eqref{sb5}. Conversely,
suppose $\left( F_{1},\cdots ,F_{D}\right) \in C^m_{loc}\left( \mathbb{R}^{k}%
\mathbb{\times \mathbb{R}}^{k},\mathbb{R}^{D}\right) $ satisfies \eqref{sb3}
and \eqref{sb5}. Then, because \eqref{sb3} implies \eqref{sb4}, it follows
that \eqref{sb2} holds, and consequently $\left( F_{1},\cdots ,F_{D}\right) $
is a section of $\left( I\left( z\right) \right) _{z\in \hat{B}}$. Thus, a
given $\left( F_{1},\cdots ,F_{D}\right) \in C^m_{loc}\left( \mathbb{R}%
^{n}\times \mathbb{R}^{k},\mathbb{R}^{D}\right) $ is a section of $\left(
I\left( z\right) \right) _{z\in \hat{B}}$ if and only if \eqref{sb3} and %
\eqref{sb5} hold. If $s\geq 1,$ this implies that we have achieved Property $%
\prod \left( s-1\right) $, contradicting the minimal character of $s$, and
establishing Claim \ref{sb-claim}.
\end{proof}

We return to the proof of Lemma \ref{semialgebraic-bundle-lemma}. Because
Property $\prod (s)$ holds with $s=0$, there exist Nash functions $A_{j\beta }^{i}$ on $U\setminus U_{\text{bad}}^{1}$%
, for which the following (\textquotedblleft Property $\prod^{\ast }$")
holds:

Let $B\subset U\setminus U_{\text{bad}}^{1}$ be a closed ball. Set $\hat{B}=%
\hat{\psi}\left( B\right) $. Then a given $\left( F_{1},\cdots ,F_{D}\right)
\in C^m_{loc}\left( \mathbb{R}^{n}\times \mathbb{R}^{k},\mathbb{R}^{D}\right) $
is a section of $\left( I\left( z\right) \right) _{z\in \hat{B}}$ if and
only if
\begin{equation}
\sum_{\left\vert \beta \right\vert \leq m}\sum_{j=1}^{D}A_{j\beta
}^{i}\left( x\right) \partial _{y}^{\beta }F_{j}\left( x,y\right) =0\text{
on }\hat{B}\text{ (all }i\text{)}.  \label{sb6}
\end{equation}

We fix $A_{j\beta }^{i}$ as above.

We now return to our bundle $\mathcal{H}=\left( f^{z}+I\left( z\right)
\right) _{z\in \hat{U}}$.

(We abuse notation by writing $f^{z}$ for $f^{x}$ where $z=\hat{\psi}\left(
x\right) $.)

Let $B\subset U\setminus U_{\text{bad}}^{1}$ be a closed ball, and let $\hat{%
B}=\hat{\psi}\left( B\right) $. Let $\left( F_{1},\cdots ,F_{D}\right) $ and $%
\left( \tilde{F}_{1},\cdots ,\tilde{F}_{D}\right) \in C^m_{loc}\left( \mathbb{R}%
^{n}\times \mathbb{R}^{k},\mathbb{R}^{D}\right) $ be any two sections of $%
\mathcal{H}|_{\hat{B}}$.

Then $\left( F_{1}-\tilde{F}_{1},\cdots ,F_{D}-\tilde{F}_{D}\right) $ is a
section of $\left( I\left( z\right) \right) _{z\in \hat{B}}$, and therefore
by \eqref{sb6}, we have
\begin{equation}
\sum_{\substack{ \left\vert \beta \right\vert \leq m  \\ j=1,\cdots ,D}}%
A_{j\beta }^{i}\left( x\right) \partial _{y}^{\beta }F_{j}\left( x,y\right)
=\sum_{\substack{ \left\vert \beta \right\vert \leq m  \\ j=1,\cdots ,D}}%
A_{j\beta }^{i}\left( x\right) \partial _{y}^{\beta }\tilde{F}_{j}\left(
x,y\right) \text{ on }\hat{B}\text{ for all }i\text{.}  \label{sb7}
\end{equation}%
Moreover, given $x_{0}\in B$, we can take our section $\left( \tilde{F}%
_{1},\cdots ,\tilde{F}_{D}\right) $ above to satisfy
\begin{equation*}
J_{\hat{\psi}\left( x_{0}\right) }\tilde{F}_{j}=f_{j}^{x_{0}}\text{ }\left(
j=1,\cdots ,D\right) \text{,}
\end{equation*}%
because $\left( f_{1}^{x_{0}},\cdots ,f_{D}^{x_{0}}\right) \in H\left( \hat{%
\psi}\left( x_{0}\right) \right) $ and $\mathcal{H}|_{\hat{B}}$ is Glaeser
stable and has nonempty fibers. (See Theorem \ref{cm-glaeserthm}.) Therefore, \eqref{sb7} implies that
\begin{equation}
\sum_{\left\vert \beta \right\vert \leq m}\sum_{j=1}^{D}A_{j\beta
}^{i}\left( x\right) \partial _{y}^{\beta }F_{j}\left( x,y\right)
=G^{i}\left( x\right) \text{ }  \label{sb8}
\end{equation}%
on $\hat{B}$ for each $i$, where
\begin{equation*}
G^{i}\left( x\right) =\sum_{\left\vert \beta \right\vert \leq
m}\sum_{j=1}^{D}A_{j\beta }^{i}\left( x\right) \left( \partial _{y}^{\beta
}f^{x}\right) |_{\hat{\psi}\left( x\right) }\text{\quad }\left( x\in
U\setminus U_{\text{bad}}^{1}\right) \text{.}
\end{equation*}%
Clearly, $G^{i}\left( x\right) $ is a semialgebraic function on $U\setminus
U_{\text{bad}}^{1}$, and it is independent of the ball $B$ in the above
discussion.

Thus, we have seen that any section $\left( F_{1},\cdots ,F_{D}\right) $ of $%
\mathcal{H}|_{\hat{B}}$ must satisfy \eqref{sb8}.

Conversely, suppose $\left( F_{1},\cdots ,F_{D}\right) \in C^m_{loc}\left(
\mathbb{R}^{n}\times \mathbb{R}^{k},\mathbb{R}^{D}\right) $ satisfies %
\eqref{sb8}. Let $\left( \tilde{F}_{1},\cdots ,\tilde{F}_{D}\right) \in
C^m_{loc}\left( \mathbb{R}^{n}\times \mathbb{R}^{k},\mathbb{R}^{D}\right) $ be a
section of $\mathcal{H}|_{\hat{B}}$. (We know that a section exists because $%
\mathcal{H}|_{\hat{B}}$ is Glaeser stable and has nonempty fibers.) We know
that $\left( \tilde{F}_{1},\cdots ,\tilde{F}_{D}\right) $ satisfies %
\eqref{sb8}, hence
\begin{equation*}
\sum_{\left\vert \beta \right\vert \leq m}\sum_{j=1}^{D}A_{j\beta
}^{i}\left( x\right) \partial _{y}^{\beta }\left[ F_{j}-\tilde{F}_{j}\right]
\left( x,y\right) =0
\end{equation*}%
on $\hat{B}$ for each $i$.

Recalling Property $\prod^{\ast }$, we now see that $\left( F_{1}-\tilde{F}%
_{1},\cdots ,F_{D}-\tilde{F}_{D}\right) $ is a section of $\left( I\left(
z\right) \right) _{z\in \hat{B}}.$ Because $\left( \tilde{F}_{1},\cdots ,%
\tilde{F}_{D}\right) \in C^m_{loc}\left( \mathbb{R}^{n}\times \mathbb{R}^{k},%
\mathbb{R}^{D}\right) $ is a section of $\mathcal{H}|_{\hat{B}}=\left(
f^{z}+I\left( z\right) \right) _{z\in \hat{B}}$, we conclude that $\left(
F_{1},\cdots ,F_{D}\right) $ is a section of $\mathcal{H}|_{\hat{B}}$. Thus,
if $\left( F_{1},\cdots ,F_{D}\right) \in C^m_{loc}\left( \mathbb{R}^{n}\times
\mathbb{R}^{k},\mathbb{R}^{D}\right) $ satisfies \eqref{sb8}, then it is a
section of $\mathcal{H}|_{\hat{B}}$.

We have now seen that a given $\left( F_{1},\cdots ,F_{D}\right) \in
C^m_{loc}\left( \mathbb{R}^{n}\times \mathbb{R}^{k},\mathbb{R}^{D}\right) $ is a
section of $\mathcal{H}|_{\hat{B}}$ if and only if \eqref{sb8} holds.

Thus, all the conclusions of Lemma \ref{semialgebraic-bundle-lemma} hold,
except that perhaps the $G^{i}$ are not real-analytic.

We set $U_{\text{bad}}^{2}=$union of all the singular sets of the
semialgebraic functions $G^{i}$. That's a semialgebraic set of dimension $<n$
in $\mathbb{R}^{n}$.

We take $U_{\text{bad}}=U_{\text{bad}}^{1}\cup U_{\text{bad}}^{2}$, a
semialgebraic set of dimension $<n$ in $\mathbb{R}^{n}$.

The functions $A_{j\beta }^{i}$ and $G^{i}$ are Nash on $U\setminus U_{\text{bad}}$.

If $B\subset U\setminus U_{\text{bad}}$ is a closed ball and $\hat{B}=\psi
\left( B\right) $, then a given $\left( F_{1},\cdots ,F_{D}\right) \in
C^m_{loc}\left( \mathbb{R}^{n}\times \mathbb{R}^{k},\mathbb{R}^{D}\right) $ is a
section of $\mathcal{H}|_{\hat{B}}$ if and only if
\begin{equation*}
\sum_{\left\vert \beta \right\vert \leq m}\sum_{j=1}^{D}A_{j\beta
}^{i}\left( x\right) \left( \partial _{y}^{\beta }F_{j}\right) |_{\hat{\psi}%
\left( x\right) }=G^{i}\left( x\right)
\end{equation*}%
on $B$ for each $i$.

This completes the proof of Lemma \ref{semialgebraic-bundle-lemma}$.$
\end{proof}

\begin{remark}\label{remarklemma5.1}
Lemma \ref{semialgebraic-bundle-lemma} and its proof hold also for $k=0$. In that case, 
$\hat{\psi}$ is the identity map and there are no $y$-variables, hence no 
$y$-derivatives in the conclusion of Lemma \ref{semialgebraic-bundle-lemma}.
\end{remark}

\begin{corollary} Let $\mathcal{H}, U, \psi, \cdots$ be as in Lemma \ref{semialgebraic-bundle-lemma}.
Let $\left( F_{1},\cdots ,F_{D}\right) \in C^m_{loc}\left( \mathbb{R}^{n}\times
\mathbb{R}^{k},\mathbb{R}^{D}\right) .$ Then $\left( F_{1},\cdots
,F_{D}\right) $ is a section of $\mathcal{H}|_{\hat{U}\setminus \hat{\psi}%
\left( U_{\text{bad}}\right) }$ if and only if
\begin{equation*}
\sum_{\left\vert \beta \right\vert \leq m}\sum_{j=1}^{D}A_{j\beta
}^{i}\left( x\right) \partial _{y}^{\beta }F_{j}\left( x,y\right)
=G^{i}\left( x\right)
\end{equation*}%
on $\hat{U}\setminus \hat{\psi}\left( U_{\text{bad}}\right) $, for all $i$.
\end{corollary}

\begin{proof}
$U\setminus U_{\text{bad}}$ is a union of (infinitely many overlapping)
closed balls $B$. Applying Lemma \ref{semialgebraic-bundle-lemma} to each $B$,
we obtain the desired conclusion.
\end{proof}

\subsection{Gaussian Elimination with Parameters}

\label{section-Gaussian-Elimination}

Suppose we are given a system of linear equations

\begin{itemize}
\item[\refstepcounter{equation}\text{(\theequation)}\label{gep1}] $%
X_{i}+\sum_{j>k}A_{ij}X_{j}=b_{i}$, for $i=1,\cdots ,k$ with $\left\vert
A_{ij}\right\vert \leq 2^{k}$ for $i=1,\cdots k,$ $j=k+1,\cdots ,M$, and

\item[\refstepcounter{equation}\text{(\theequation)}\label{gep2}] $%
\sum_{j>k}C_{ij}X_{j}=g_{i}$, for $i=k+1,\cdots ,N$,
\end{itemize}
where $0\leq k\leq N,M;$ the $A_{ij}$, $C_{ij}$, $b_{i}$, $g_{i}$ are
semialgebraic functions defined on a semialgebraic set $E\subset \mathbb{R}%
^{n}$; and $X_{1},\cdots ,X_{M}$ are unknowns.

We say that this system is in \underline{$k$-echelon form} on $E$

If $k=0$, then we have simply \eqref{gep2} for $i=1,\cdots,N$, so every
system of linear equations with coefficient matrix and right-hand sides
depending semialgebraically on $x \in E$ is in $0$-echelon form on $E$.

If also $C_{ij} \equiv 0$ on $E$ for all $i=k+1,\cdots,N$, $j=k+1,\cdots,M$,
then we say that our system of equations is in \underline{echelon form} on $%
E $. In particular, a system in $k$-echelon form with $k=\min\{N,M\}$ is in
echelon form on $E$. Suppose our system is in $k$-echelon form with $k<\min
\{N,M\}$. We partition $E$ as follows. Let $E_{\text{good}} = \{x \in E:
\text{All the } C_{ij}(x) =0 \}$. For $\tilde{i}=k+1,\cdots, N$ and $\tilde{j}=k+1,\cdots,M$%
, we let $\tilde{E}(\ti,\tj)= \{x \in E: |C_{\ti \tj}| = \max_{i j} |C_{i j}| >0 \}$.
The $E_{\text{good}}$ and $\tilde{E}(i,j)$ form a covering of $E$.

We enumerate the pairs $(i,j)$ in any order and then form sets $E(i,j)$ by
removing from $\tilde{E}(i,j)$ all points contained in some $\tilde{E}%
(i^{\prime },j^{\prime })$ with $(i^{\prime },j^{\prime })$ preceding $(i,j)$%
. Then $E_{\text{good}}$ and the $E(i,j)$ form a partition of $E$ into
semialgebraic sets. On $E_{\text{good}}$, our system is in echelon form.

On
each $E(a,b)$, we will exhibit a system of linear equations in $(k+1)$%
-echelon form, equivalent to the given system \eqref{gep1}, \eqref{gep2}.
For fixed $(a,b)$, we relabel equations and unknowns so that our system
still has the form \eqref{gep1}, \eqref{gep2}, but with $|C_{k+1,k+1}|=
\max_{ij}|C_{ij}| >0$. Dividing equations \eqref{gep2} by $C_{k+1,k+1}$, we
may assume that
\begin{equation}
C_{k+1,k+1} =1  \label{gep3}
\end{equation}
and all
\begin{equation}
|C_{ij}|\leq 1.  \label{gep4}
\end{equation}

Note that $A_{ij}, C_{ij}, b_i, g_i$ still depend semialgebraically on $x$.
From each equation \eqref{gep1}, we subtract $A_{i(k+1)}$ times equation %
\eqref{gep2} with $i=k+1$. From each equation \eqref{gep2} ($i\not=k+1$), we
subtract $C_{i,k+1}$ times equation \eqref{gep2} with $i=k+1$. Thus, we
obtain equations of the form

\begin{equation}
\left[
\begin{array}{l}
X_{i}+\sum_{j>k}\tilde{A}_{ij}X_{j}=\tilde{b}_{i},\text{for }i=1,\cdots ,k
\\
X_{k+1}+\sum_{j>k+1}C_{k+1,j}X_{j}=g_{k+1}, \\
\sum_{j\geq k+1}\tilde{C}_{ij}X_{j}=\tilde{g}_{i}\text{, for }i>k+1.%
\end{array}%
\right.  \label{gp5}
\end{equation}%
Here, $\tilde{A}_{ij}=A_{ij}-A_{i\left( k+1\right) }C_{k+1,j}$ for $%
i=1,\cdots ,k$, $j\geq k+1$; and $\tilde{C}_{ij}=C_{ij}-C_{i,k+1}C_{k+1,j}$
for $i=k+2,\cdots ,N$, $j> k+1$.

In particular, $\tilde{A}_{i,k+1}=A_{i,k+1}-A_{i,k+1}\cdot C_{k+1,k+1}=0$,
and $\tilde{C}_{i,k+1}=C_{i,k+1}-C_{i,k+1}\cdot C_{k+1,k+1}=0$, thanks to %
\eqref{gep3}.

Also, $\left\vert \tilde{A}_{ij}\right\vert \leq \left\vert
A_{ij}\right\vert +\left\vert A_{i,k+1}\right\vert \cdot \left\vert
C_{k+1,j}\right\vert \leq \left\vert A_{ij}\right\vert +\left\vert
A_{i,k+1}\right\vert $ (by \eqref{gep4})$\leq 2^{k}+2^{k}$ (because our
system \eqref{gep1}, \eqref{gep2} is in $k$-echelon form)$=2^{k+1}$. Recall that $|C_{k+1,j}| \leq 1$.

These remarks show that the system of equations \eqref{gp5} is in $\left(
k+1\right) $-echelon form.

We repeat this procedure, starting with a system in $0$-echelon form,
and partition $E$ more and more finely into pieces $E_{\nu }$, on each of which an equivalent system to \eqref{gep1}, \eqref{gep2} is either in echelon
form, or in $k$-echelon form for ever higher $k$. The procedure has to stop
after at most $\min \left( N,M\right) $ steps, because a system in $k$%
-echelon form with $k=\min \left( N,M\right) $ is automatically in echelon
form.

Thus, we have proven the following result

\begin{lemma}
\label{gep-lemma} Consider a system of linear equations
\begin{equation}
\sum_{j=1}^{M}C_{ij}\left( x\right) X_{j}=g_{i}\left( x\right) \text{ }%
\left( i=1,\cdots ,N\right)  \label{gepstar}
\end{equation}%
where the $C_{ij}\left( x\right) $ and $g_{i}\left( x\right) $ are
semialgebraic functions defined on a semialgebraic set $E\subset \mathbb{R}%
^{n}$.

Then we can partition $E$ into semialgebraic sets $E_{\nu }$ $\left( \nu
=1,\cdots ,\nu _{\max }\right) $, for which the following holds for each $%
\nu $:

There exist a permutation $\pi :\left\{ 1,\cdots ,M\right\} \rightarrow
\left\{ 1,\cdots ,M\right\} $ and an integer $0\leq k\leq \min \left(
N,M\right) $ such that for each $x\in E_{\nu }$, the system \eqref{gepstar}
is equivalent to a system of the form
\begin{equation}
\left[
\begin{array}{c}
X_{\pi i}+\sum_{j>k}\tilde{A}_{ij}\left( x\right) X_{\pi j}=\tilde{g}%
_{i}\left( x\right) \text{ for }i=1,\cdots ,k \\
0=\tilde{b}_{i}\left( x\right) \text{ for }i=k+1,\cdots ,N\text{.}%
\end{array}%
\right.  \label{gepstarstar}
\end{equation}%
That is, for each $x\in E_{\nu }$ and each $\left( X_{1},\cdots
,X_{M}\right) \in \mathbb{C}^{M}$, \eqref{gepstar} holds at $x$ if and only if %
\eqref{gepstarstar} holds at $x$. Here, the $\tilde{A}_{ij},\tilde{g}_{i},$ and $%
\tilde{b}_{i}$ are semialgebraic functions on $E_{\nu }$, and $\left\vert
\tilde{A}_{ij}\left( x\right) \right\vert \leq 2^{k}$ on $E_{\nu }$.
\end{lemma}

In essence, the method for solving the system \eqref{gepstar} is just the
usual Gaussian elimination, except that we take extra care to maintain the
growth condition $\left\vert \tilde{A}_{ij}\left( x\right) \right\vert \leq
2^{k}$.

\subsection{What It Means to be a Section of a Semialgebraic Bundle}

We work with a semialgebraic bundle $\mathcal{H} = (H(x))_{x \in \mathbb{R}%
^2}$. Each $H(x)$ is a coset of an $\mathcal{R}_x$-submodule of $(\mathcal{R}%
_x)^D$, depending semialgebraically on $x$. Here, $\mathcal{R}_x$ is the
ring of the $m$-jets of functions at $x$. A function $F = (F_1, \cdots, F_D)
\in C^m_{loc} (\Omega, \mathbb{R}^D)$ ($\Omega \subset \mathbb{R}^2$ open) is a
section of $\mathcal{H}$ if for all $x \in \Omega$ the $m$-jet $J_x F$
belongs to $H(x)$. A function $F \in C^m_{loc} (\Omega, \mathbb{R}^D)$ is called a
local section near $x^0$ ($x^0 \in \Omega$) if for some small disc $B
\subset \Omega$ centered at $x^0$ we have $J_x F \in H(x)$ for all $x \in B$.

Let $\Omega = \{(x,y) \in \mathbb{R}^2: 0 \leq y \leq x \}$. Let $\mathcal{H}
= (H(x))_{x \in \mathbb{R}^2}$ be a semialgebraic bundle, with $H((0,0))=\{0
\}$. We assume that $\mathcal{H}$ has a section. We want a convenient condition on functions $F \in C^m_{loc} (\Omega, \mathbb{%
R}^D)$ that is equivalent to the assertion that $F|_{B \cap \Omega^{\text{%
interior}}}$ is a section of $\mathcal{H}$ for a small enough disc $B$
centered at the origin. We achieve (approximately) that.

To do so, we partition $\Omega$ into semialgebraic open subsets of $\mathbb{R%
}^2$, finitely many semialgebraic curves in $\mathbb{R}^2$, and finitely
many points. To start with, we partition $\Omega $ into the point $(0,0)$,
the arcs $\{(x,0):x>0\},\{(x,x): x>0 \},$ and $\Omega^{\text{interior}}$.

As we proceed, we will cut up each of our semialgebraic open sets into
finitely many semialgebraic open subsets, finitely many semialgebraic arcs,
and finitely many points. We won't keep track explicitly of the arcs and
points at first; we just discard semialgebraic subsets of $\mathbb{R}^{2}$
of dimension $\leq 1$.

We apply Lemma \ref{semialgebraic-bundle-lemma} in the case $k=0$ to $\Omega ^{\text{interior}%
} $ and $\mathcal{H}$. (See Remark \ref{remarklemma5.1}.)

Thus, we obtain a semialgebraic $V_{1}\subset \Omega ^{\text{interior}}$ of
dimension $\leq 1$, outside of which the following holds for some
semialgebraic functions $A_{ij}^{\#}(x),\phi _{i}^{\#}(x)$ for $1\leq i\leq
i_{\max },1\leq j\leq D,x\in \Omega ^{\text{interior}}\setminus V_{1}$:

Let $F=(F_1,\cdots, F_D)$ belong to $C^m_{loc}(U,\mathbb{R}^D)$ where $U$ is a neighborhood of $x^0
\in \Omega^{\text{interior}} \setminus V_1$. Then $F$ is a local section of $%
\mathcal{H}$ near $x^0$ if and only if

\begin{itemize}
\item[\refstepcounter{equation}\text{(\theequation)}\label{WIRM1}] $%
\sum_{j=1}^D A_{ij}^{\#}(x) F_j(x) = \phi_i^{\#}(x)$, for $i=1,\cdots,
i_{\max}$, for all $x$ in a neighborhood of $x^0$.
\end{itemize}

The equations \eqref{WIRM1} have a solution for each fixed $x$, because $%
\mathcal{H}$ has a section. Next, we apply Lemma \ref{gep-lemma} to the
above system of linear equations.

Thus, we obtain a partition of $\Omega ^{\text{interior}}\setminus V_{1}$
into semialgebraic sets $E_{\nu }^{\#}$ ($\nu =1,\cdots ,\nu _{\max }^{\#}$%
), for which we have integers $\tilde{k}_{\nu }\geq 0$, permutations $\tilde{%
\pi}_{\nu }:\{1,\cdots ,D\}\rightarrow \{1,\cdots ,D\}$, and semialgebraic
functions $\tilde{A}_{ij}^{\nu }(x)$ ($1\leq i\leq \tilde{k}_{\nu },\tilde{k}%
_{\nu }+1\leq j\leq D,x\in E_{\nu }^{\#}$), $\tilde{\phi}_{i}^{\nu }(x)$
such that for any $x\in E_{\nu }^{\#}$, the system of equations \eqref{WIRM1}
is equivalent to%
\begin{equation}
F_{\pi _{\nu }i}\left( x\right) +\sum_{j>\tilde{k}_{\nu }}\tilde{A}%
_{ij}^{\nu }\left( x\right) F_{\pi _{\nu }j}\left( x\right) =\tilde{\varphi}%
_{i}^{\nu }\left( x\right) \text{ for } i=1,\cdots ,\tilde{k}_{\nu }.  \label{WIR2}
\end{equation}

Moreover, the $\tilde{A}_{ij}^{\nu }\left( x\right) $ are bounded. Note that the functions $\tilde{b}_i$ in \eqref{gepstarstar} are identically $0$ because our equations \eqref{WIRM1} have a solution.

Because $\mathcal{H}$ has a section, there exists $F=\left( F_{1},\cdots
,F_{D}\right) \in C^m_{loc}\left( \Omega ,\mathbb{R}^{D}\right) $ satisfying (%
\ref{WIRM1}) for all $x\in \Omega ^{\text{interior}}\setminus V_{1}$, hence
also satisfying (\ref{WIR2}) in $E_{\nu }^{\#}$. Consequently, the left-hand
side of (\ref{WIR2}) is bounded (for bounded $x$), and thus also the $\tilde{\varphi}%
_{i}^{D}\left( x\right) $ are bounded (for bounded $x$).

Applying Theorem \ref{log-derivative-theorm}, we obtain a semialgebraic $%
V_{2}\subset \mathbb{R}^{2}$ of dimension $\leq 1$, satisfying
\begin{equation}
\left\vert \partial ^{\alpha }\tilde{\varphi}_{i}^{\nu }\left( x\right)
\right\vert ,\left\vert \partial ^{\alpha }\tilde{A}_{ij}^{\nu }\left(
x\right) \right\vert \leq C\left[ \text{dist}\left( x,V_{2}\right) \right]
^{-\left\vert \alpha \right\vert }\text{ for bounded } x \text{ outside }V_{2}\text{, for }%
\left\vert \alpha \right\vert \leq m+100\text{.}  \label{WIRM3}
\end{equation}

By adding $\partial \Omega $ to $V_{2}$ and removing from $V_{2}$ all points
outside $\Omega $, we may assume $V_{2}\subset \Omega $. (This operation
does not increase the distance from $V_{2}$ to any point of $\Omega $.)

Let $\hat{E}_{\nu }$ $\left( \nu =1,\cdots ,\nu _{\max }\right) $ be the
connected components of the interiors of the sets $E_{\nu }^{\#}\setminus
V_{2} $ ($\nu =1,\cdots ,\nu _{\max }^{\#}$).

Then $\Omega $ is partitioned into the $\hat{E}_{\nu }$
and $V_{3}$, where $V_{3}$ is a semialgebraic subset of $\Omega $ of
dimension $\leq 1$. The $\hat{E}_{\nu }$ are pairwise disjoint open
connected semialgebraic sets. Any path in $\Omega $ that does not meet $%
V_{3} $ stays entirely in a single $\hat{E}_{\nu }$. Indeed, suppose not: let $\gamma
\left( t\right) \in \Omega $ \ $\left( t\in \left[ 0,1\right] \right) $ be a
path starting at $\gamma \left( 0\right) \in \hat{E}_{\nu }$ not staying in $\hat{E}_\nu$ and not meeting
$V_{3}$. Pick $t_{\ast }=$ $\inf \left\{ t>0:\gamma \left( t\right) \not\in
\hat{E}_{\nu }\right\} $. Then $t^*>0$ since $\hat{E}_\nu$ is open. We can't have $\gamma \left( t_{\ast }\right) \in
\hat{E}_{\nu ^{\prime }}$ with $\nu ^{\prime }\not=\nu $ else $\gamma \left(
t\right) \in \hat{E}_{\nu ^{\prime }}$ (and $\in \hat{E}_{\nu }$) for $t\in
\lbrack t_{\ast }-\varepsilon ,t_{\ast })$. We can't have $\gamma(t_{\ast })$ in $E_\nu$, since that  would imply $\gamma(t)$ in $E_\nu$ for all $t$ in $[t_{\ast }, t_{\ast }+\varepsilon]$. Thus, $\gamma \left( t_{\ast
}\right) \in V_{3}$, contradicting the fact that $\gamma $ does not meet $%
V_{3}$.

Moreover, there exist integers $\hat{k}_{\nu }\geq 0$, permutations $\hat{\pi%
}_{\nu }:\left\{ 1,\cdots ,D\right\} \rightarrow \left\{ 1,\cdots ,D\right\}
$, and semialgebraic functions $\hat{A}_{ij}^{\nu }\left( x\right) $ $\left(
1\leq i\leq \hat{k}_{\nu }\text{, }\hat{k}_{\nu }+1\leq j\leq D\right) $ and
$\hat{\varphi}_{i}^{\nu }\left( x\right) $ $\left( 1\leq i\leq \hat{k}_{\nu
}\right) $ defined on $\hat{E}_{\nu }$, with the following properties

\begin{itemize}
\item[\refstepcounter{equation}\text{(\theequation)}\label{WIRM4}] $%
\left\vert \partial ^{\alpha }\hat{A}_{ij}^{\nu }\left( x\right) \right\vert
$, $\left\vert \partial ^{\alpha }\hat{\varphi}_{i}^{\nu }\left( x\right)
\right\vert \leq C\left[ \text{dist}\left( x,V_{3}\right) \right]
^{-\left\vert \alpha \right\vert }$ for bounded $x\in \hat{E}_{\nu }$, $\left\vert
\alpha \right\vert \leq m+100$, and

\item[\refstepcounter{equation}\text{(\theequation)}\label{WIRM5}] Let $%
x^{0}\in \hat{E}_{\nu }$ and let $F=\left( F_{1},\cdots ,F_{D}\right) $ be $%
C^m_{loc}$ in a neighborhood of $x^{0}$. Then $F$ is a local section of $%
\mathcal{H}$ near $x^{0}$ if and only if
\begin{equation*}
F_{\pi _{\nu }i}\left( x\right) +\sum_{j>\hat{k}_{\nu }}\hat{A}_{ij}^{\nu
}\left( x\right) F_{\pi _{\nu }j}\left( x\right) =\hat{\varphi}_{i}^{\nu
}\left( x\right)
\end{equation*}%
in a neighborhood of $x^{0}$ for each $i=1,\cdots ,\hat{k}_{\nu }$.
\end{itemize}

We partition $V_{3}\cup \left\{ \left( x,0\right) :x\geq 0\right\} \cup
\left\{ \left( x,x\right) :x\geq 0\right\} $ into finitely many
Nash open arcs (not containing their endpoints) and finitely many points.

For small enough $\delta >0$, $B\left( 0,\delta \right) \mathbb{\subset
\mathbb{R}}^{2}$ avoids all the above arcs not containing $0$ in their
closure, and all the above points except possibly for the point $0$. Taking $%
\delta $ small, we may assume that the remaining arcs have
convergent Puiseux series in $B(0,\delta)$.

Notice that our semialgebraic one-dimensional sets are all contained in $%
\Omega $; so no arcs have tangent lines at $0$ lying outside the sector $%
\Omega $. Thus, the remaining arcs have the form $\{y=\psi _{s}(x)\}$ in $%
B(0,\delta )$, where $\psi _{1},\cdots ,\psi _{s_{\max }}$ are semialgebraic
functions of one variable, with convergent Puiseux expansion in $[0,\delta ]$%
. We discard duplicates, i.e., we may assume $\psi _{s}$ is never
identically equal to $\psi _{s^{\prime }}$ for $s^{\prime }\not=s$. Note that the line segments $\{(x,0):0<x<\delta \}$ and $\{(x,x):0<x<\delta\}$ are among our arcs $\gamma_s$. Taking
$\delta >0$ smaller yet, we may assume that for each $s\not=s^{\prime }$,
either $\psi _{s}(x)<\psi _{s^{\prime }}(x)$ for all $x\in (0,\delta )$, or $%
\psi _{s}(x)>\psi _{s^{\prime }}(x)$ for all $x\in (0,\delta )$. (That's
because the $\psi _{s}$ are given by convergent Puiseux expansions.) Thus, in
$B(0,\delta )$, our curves may be labelled so that $0\equiv \psi
_{0}(x)<\psi _{1}(x)<\cdots <\psi _{s_{\max }}(x)\equiv x$ for $x\in
(0,\delta )$. The arcs are $\gamma _{s}=\{(x,\psi _{s}(x)):x\in \lbrack
0,\delta ]\}$ for $s=0,\cdots ,s_{\max }$. (Here we have thrown in the point
$0$, and taken $\delta $ small to allow ourselves to include $x=\delta $, not
just $x<\delta $.)

The sets we discarded in passing from $V_3$ to the semialgebraic arcs $%
\gamma_0, \cdots, \gamma_{s_{\max}}$ are irrelevant in the sense that $V_3
\cap B(0, \delta) \subset (\gamma_0 \cup \gamma_1 \cup \cdots \cup
\gamma_{s_{\max}})\cap B(0, \delta)$.

Let $E_s$ ($s = 1, \cdots, s_{\max}$) be the part of the $B(0,\delta)$ lying
between $\gamma_{s-1}$ and $\gamma_{s}$, i.e., $E_s= \{(x,y) \in B(0,
\delta): 0 < x < \delta, \psi_{s-1}(x) < y < \psi_s(x) \}$.

Any two points in a given $E_s$ may be joined by a path in $B(0,\delta)
\setminus \bigcup_{s=0}^{s_{\max}}\gamma_s \subset B(0,\delta) \setminus V_3$%
, hence all points in a given $E_s$ lie in the same $\hat{E}_\nu$.

Therefore, for $s = 1, \cdots, s_{\max}$, there exist $k_s \geq 0 $,
permutations $\pi_s: \{1, \cdots, D\} \rightarrow \{ 1, \cdots, D\}$, and
semialgebraic functions $A_{ij}^s(x)$, $\psi_i^s(x)$ ($1\leq i \leq k_s; j =
k_s+1, \cdots, D$) on $E_s$, with the following properties

\begin{itemize}
\item[\refstepcounter{equation}\text{(\theequation)}\label{WIRM-I}] Let $x^0
\in E_s$, and let $F = (F_1, \cdots, F_D)$ be $C^m_{loc}$ in a neighborhood of $%
x^0 $. Then $F$ is a local section of $\mathcal{H}$ near $x^0$ if and only if

\item[\refstepcounter{equation}\text{(\theequation)}\label{WIRM6}] $F_{\pi_s
i}(x) + \sum_{j>k_s}A_{ij}^s(x) F_{\pi_s j}(x) = \psi_i^s (x)$ in a
neighborhood of $x^0$ for each $i =1, \cdots, k_s$.
\end{itemize}

Moreover,

\begin{itemize}
\item[\refstepcounter{equation}\text{(\theequation)}\label{WIRM-II}] $%
|\partial^\alpha A_{ij}^s(x)|$, $|\partial^\alpha \psi_i^s(x)| \leq C \left[%
\text{dist} (x, \gamma_s \cup \gamma_{s-1}) \right]^{-|\alpha|}$ on $E_s$
for $|\alpha| \leq m+100$.
\end{itemize}

In particular, if $F = (F_1, \cdots, F_D) \in C^m_{loc} (\Omega, \mathbb{R}^D)$,
then $J_x F \in H(x)$ for all $x \in [\Omega \cap B(0, \delta)] \setminus
(\gamma_0 \cup \cdots \cup \gamma_{s_{\max}})$ if and only if for each $s=1,
\cdots, s_{\max}$, \eqref{WIRM6} holds on all of $E_s$.

Next, we apply Lemma \ref{semialgebraic-bundle-lemma} to $\mathcal{H}%
_{s}=(H(x))_{x\in \gamma _{s}}$, ($s=0,\cdots ,s_{\max }$). We obtain
semialgebraic functions for which the following holds.

Let $\left( x^{0},\psi _{s}\left( x^{0}\right) \right) \in \gamma _{s}$ be
given, and let $F=\left( F_{1},\cdots ,F_{D}\right) \in C^m_{loc}\left( U,%
\mathbb{R}^{D}\right) $, where $U$ is a neighborhood of $\gamma _{s}$ in $%
\mathbb{R}^{2}$. Then, except for finitely many bad $x^0$, we have the following equivalence:

$F$ is a local section of $\mathcal{H}_{s}$ near $\left( x^{0},\psi
_{s}\left( x^{0}\right) \right) $ if and only if
\begin{equation*}
\sum_{\substack{ 1\leq j\leq D  \\ 0\leq l\leq m}}\Theta _{jl}^{is}\left(
x\right) \partial _{y}^{l}F_{j}|_{\left( x,\psi _{s}\left( x\right) \right)
}=g^{si}\left( x\right) \quad \left( i=1,\cdots ,i_{\max }\left( s\right)
\right)
\end{equation*}%
for all $x$ in a neighborhood of $x^{0}$. Here, the $\Theta $'s and $g$'s are semialgebraic functions of one
variable. To say that $F$ is a local section of $\mathcal{H}_{s}$ near $%
\left( x^{0},\psi _{s}\left( x^{0}\right) \right) $ means that $J_{\left(
x,\psi _{s}\left( x\right) \right) }F\in H\left( x,\psi _{s}\left( x\right)
\right) $ for all $x$ in a neighborhood of $x^{0}$.

By restricting attention to $B\left( 0,\delta \right) $ and taking $\delta
>0 $ smaller, we may exclude from $B\left( 0,\delta \right) $ all these bad $%
x^{0}$, except for $x^{0}=0$.

Combining our results \eqref{WIRM-I}, \eqref{WIRM-II} on the $E_{\nu }$ with
the above result on the arcs $\gamma _{s}$, we obtain the following result.

\begin{lemma}
\label{WIRM-lemma} Let $\Omega =\left\{ \left( x,y\right) \in \mathbb{R}%
^{2}:0\leq y\leq x\leq 1\right\} $ and let $\mathcal{H=}\left( H\left(
x\right) \right) _{x\in \Omega }$ \ be a semialgebraic bundle, with each $%
H\left( x\right) $ consisting of $m$-jets at $x$ of functions from $\mathbb{R%
}^{2}$ to $\mathbb{R}^{D}$.

Assume $H\left( \left( 0,0\right) \right) =\left\{ 0\right\} $ and assume $\mathcal{H}$ has a section.

Then there exist the following objects, with properties to be specified
below:

\begin{itemize}
\item A positive number $\delta \in \left( 0,1\right) $.

\item Semialgebraic functions $0 = \psi _{0}\left( x\right) < \psi
_{1}\left( x\right) < \cdots < \psi _{s_{\max }}\left( x\right) =x$ on
$\left( 0,\delta \right) ,$ all given by convergent Puiseux expansions on $%
\left( 0,\delta \right) $.

\item Integers $k_{s}$ $\left( 0\leq k_{s}\leq D\right) $ and permutations $%
\pi _{s}:\left\{ 1,\cdots ,D\right\} \rightarrow \left\{ 1,\cdots ,D\right\}
$ for $s=1,\cdots ,D$.

\item Semialgebraic functions $A_{ij}^{s}\left( x,y\right) $ $\left(
s=1,\cdots ,s_{\max},1\leq i\leq k_{s},k_{s}<j\leq D\right) $ and $\varphi
_{i}^{s}\left( x,y\right) $ $( s=1,\cdots ,s_{\max},1\leq i\leq k_{s}) $ defined on $%
E_{s}=\left\{ \left( x,y\right) :0<x<\delta ,\psi _{s-1}\left( x\right)
<y<\psi _{s}\left( x\right) \right\} $.

\item Semialgebraic functions $\Theta _{jl}^{si}\left( x\right) $, $%
g^{si}\left( x\right) $ $(s=0,\cdots ,s_{\max },i=1,\cdots ,i_{\max }\left(
s\right) $, $j=1,\cdots ,D,$ $l=0,\cdots ,m$ defined on $\left( 0,\delta
\right) $, and given there by there by convergent Puiseux expansions.
\end{itemize}

The above objects have the following properties

\begin{itemize}
\item (Estimates) For $\left( x,y\right) \in \Omega $ with $0<x<\delta $ and
$\psi _{s-1}\left( x\right) <y<\psi _{s}\left( x\right) $, we have $%
\left\vert \partial ^{\alpha }A_{ij}^{s}\left( x,y\right) \right\vert $, $%
\left\vert \partial ^{\alpha }\varphi _{i}^{s}\left( x,y\right) \right\vert
\leq C\left[ \min \left( \left\vert y-\psi _{s}\left( x\right) \right\vert
,\left\vert y-\psi _{s-1}\left( x\right) \right\vert \right) \right]
^{-\left\vert \alpha \right\vert }$ for $\left\vert \alpha \right\vert \leq
m+100$.

\item (Condition for sections) Let $F=(F_1,...,F_D)\in C^m_{loc}(\Omega, \R^D)$, and
suppose $J_{x}F\in H\left( x\right) $ for all $x\in \Omega $.

\begin{itemize}
\item[\refstepcounter{equation}\text{(\theequation)}\label{WIRM-box}] Then
for $s=1,\cdots ,s_{\max }$, $i=1,\cdots ,k_{s}$, $x\in \left( 0,\delta
\right) $, $\psi _{s-1}\left( x\right) <y<\psi _{s}\left( x\right) $, we
have
\begin{equation*}
F_{\pi _{s}i}\left( x,y\right) +\sum_{D\geq j>k_{s}}A_{ij}^{s}\left(
x,y\right) F_{\pi _{s}j}\left( x,y\right) =\varphi _{i}^{s}\left( x,y\right)
\text{;}
\end{equation*}%
and for $s=0,1,\cdots ,s_{\max }$, $i=1,\cdots ,i_{\max }\left( s\right) $, $%
x\in \left( 0,\delta \right) $, we have
\begin{equation*}
\sum_{j=1}^{D}\sum_{l=0}^{m}\Theta _{jl}^{si}\left( x\right) \partial
_{y}^{l}F_{j}\left( x,\psi _{s}\left( x\right) \right) =g^{si}\left(
x\right) \text{;}
\end{equation*}%
and $J_{\left( 0,0\right) }F_j=0$ for all $j$.
\end{itemize}

Conversely, if $F=(F_1,...,F_D)\in C^m_{loc}(\Omega, \R^D) $ and the conditions in %
\eqref{WIRM-box} are satisfied, then $J_{z}F$ $\in H\left( z\right) $ for
all $z=\left( x,y\right) \in \Omega $ \ with $0\leq x<\delta $.
\end{itemize}
\end{lemma}

\section{A Second Main Lemma}
This section is devoted to the proof of the following lemma. See (A) and (B) in the Introduction.

\begin{lemma}[Second Main Lemma]
\label{lemma-pit} Let $\mathcal{H}=(H(z))_{z \in \Omega}$ with $\Omega =
\left\{ (x,y) \in \mathbb{R}^2: 0 \leq y \leq x \leq 1 \right\}$ and suppose $%
H(z)$ depends semialgebraically on $z$. (As usual, $H(z) \subset \mathcal{R}_z^D$ is a
coset of an $\mathcal{R}_z$-submodule.)

Suppose $\mathcal{H}$ has a section, and suppose $\mathcal{H}((0,0)) =
\left\{ 0 \right\}$. Then there exist semialgebraic functions $%
\theta_{jl}^{si} (x)$, $g^{si} (x)$, $\tilde{\theta}_{jl}^{si}(x)$, $\tilde{g%
}^{si}(x)$ of one variable, and $0 = \psi_0 (x) < \cdots< \psi_{s_{\max}}(x)
= x$, also semialgebraic, for which the following hold.

Suppose $F = (F_1, \cdots, F_D) \in C^m (\Omega,\R^D)$ is a section of $\mathcal{H%
}$. Let $f_{jl}^s(x) = \partial_y^l F_j (x, \psi_s(x))$ for $0 \leq s \leq
s_{\max}$, $0 \leq l \leq m $, $1 \leq j \leq D$.

Then

\begin{enumerate}
\item[\refstepcounter{equation}\text{(\theequation)}\label{pit-star}] $%
\sum_{j,l}\theta _{jl}^{si}(x)f_{jl}^{s}(x)=g^{si}(x)$ on $(0,\delta )$ for
some $\delta >0$ for each $s,i$; and $\sum_{j,l}\tilde{\theta}%
_{jl}^{si}(x)f_{jl}^{s}(x)=\tilde{g}^{si}(x)+o(1)$ as $x\rightarrow 0^{+}$,
each $s$, $i$; and $f_{jl}^{s}(x)=\sum_{k=0}^{m-l}\frac{1}{k!}%
f_{j(l+k)}^{s-1}(x)\cdot \left( \psi _{s}\left( x\right) -\psi _{s-1}\left(
x\right) \right) ^{k}+o\left( \left[ \psi _{s}\left( x\right) -\psi
_{s-1}\left( x\right) \right] ^{m-l}\right) $ as $x\rightarrow 0^{+}$, each $%
s$, $j$, $l$.
\end{enumerate}
\begin{enumerate}
\item[\LA{conversely}]
Conversely, if $f_{jl}^{s}\left( x\right) $ are semialgebraic functions
satisfying (\ref{pit-star}), then there exists a semialgebraic $C^m$
section $F=\left( F_{1},\cdots ,F_{D}\right) $ of $\mathcal{H}$ over $\Omega
_{\delta^{\prime} }=\left\{ \left( x,y\right) :0\leq y\leq x\leq \delta ^{\prime
}\right\} $ (some $\delta ^{\prime }>0$) such that $\partial
_{y}^{l}F_{j}\left( x,\psi _{s}\left( x\right) \right) =f_{jl}^{s}\left(
x\right) $ for $0<x<\delta ^{\prime }$.
\end{enumerate}
\end{lemma}

We call the curves $y=\psi_{s}(x)$ ``critical curves".

\subsection{The Jet of a Section at a Critical Curve}\label{section-notation-prelim}
Fix $m\geq 1$. Recall that $\mathcal{P}$ denotes the space of polynomials of degree $%
\leq m$ on $\mathbb{R}^{2}$, and $J_{z}F\in \mathcal{P}$ denotes the $m$-jet
of $F$ at $z\in \mathbb{R}^{2}$. $\odot _{z}$ denotes multiplication of jets
at $z$. We write $\mathfrak{p}$ to denote the space of polynomials of degree $\leq m$
on $\mathbb{R}$. If $F(x,y)$ is a $C^m_{loc}$ function in a neighborhood of $(%
\bar{x},0)$, then $j_{\bar{x}}F\in \mathfrak{p}$ is the $m$-jet at $0$ of
the function $y\mapsto F(\bar{x},y)$. We write $\boxdot $ to denote multiplication of $%
m $-jets at $0$ of $C^m_{loc}$ functions of one variable.

If $\vec{F}=(F_1,\cdots,F_{j_{\max}})$ is a vector of $C^m_{loc}$ functions on $%
\mathbb{R}^2$, then $J_z\vec{F}$ denotes
\begin{equation*}
(J_zF_1,\cdots,J_zF_{j_{\max}}) \in \mathcal{P}^{j_{\max}}.
\end{equation*}

Similarly, $j_{\bar{x}}\vec{F}$ denotes $(j_{\bar{x}}F_1,\cdots, j_{\bar{x}%
}F_{j_{\max}}) \in \mathfrak{p}^{j_{\max}}.$

A function $F^\#:(0,\delta) \rightarrow \mathfrak{p}$
may be regarded as a function of $(x,y) \in (0,\delta) \times \mathbb{R}$
such that for fixed $x$, the function $y \mapsto F^\#(x,y)$ is a polynomial of degree at most $m$.

Fix positive integers $i_{\max},j_{\max}$. Let $\text{Aff}$ denote the
vector space of all affine functions defined on $\mathfrak{p}%
^{j_{\max}+i_{\max}}$. We make the following assumptions:

\begin{itemize}
\item We are given $C^\infty$ semialgebraic functions $A_{ij},B_{i},
(i=1,\cdots,i_{\max}, j=1,\cdots,j_{\max})$ defined on $\Omega_1$, where for
$\delta>0$, $\Omega_\delta=\{(x,y)\in \mathbb{R}^2:0<x<\delta,0<y<\psi(x)\}$%
, and $\psi:(0,1) \rightarrow (0,\infty)$ is a semialgebraic function
satisfying $0< \psi(x)\leq x$ for $x\in (0,1)$.

\item We assume that $\partial^\alpha A_{ij}, \partial^\alpha B_i$ extend to
continuous functions on $\Omega_1^+$ for $|\alpha|\leq m$, where, for $\delta>0$, $%
\Omega_\delta^+=\{(x,y)\in \mathbb{R}^2: 0<x \leq \delta, 0<y\leq \psi(x)\}$.

\item We suppose that
\begin{eqnarray*}
|\partial^\alpha A_{ij}(x,y)| &\leq& C y ^{-|\alpha|}, \text{ and} \\
|\partial^\alpha B_{i}(x,y)|&\leq& Cy^{-|\alpha|}
\end{eqnarray*}
on $\Omega^+_1$ for $|\alpha|\leq m$.
\end{itemize}

\begin{lemma}
\label{main-lemma} Under the above assumptions, there exist $\delta \in
(0,1) $ and semialgebraic maps $\lambda_1,\cdots,\lambda_{k_{\max}},\mu_1,%
\cdots,\mu_{l_{\max}}:(0,\delta) \rightarrow \text{Aff}$ such that the
following hold:

\begin{itemize}
\item[\LA{assertionI}] Suppose $\vec{F}=(F_1,\cdots,F_{j_{\max}})$ and $\vec{G}%
=(G_1,\cdots,G_{i_{\max}})$ belong to $C^m(\Omega_\delta^\text{closure},\R^{j_{\max}})$ and $C^m(\Omega_\delta^\text{closure},\R^{i_{\max}})$ respectively,
with $J_{(0,0)}\vec{F}=0,J_{(0,0)}\vec{G}=0$. Suppose also that $G_i =
\sum_j A_{ij}F_j+B_i$ for each $i$. Then $[\lambda_k(\bar{x})](j_{\bar{x}}
\vec{F},j_{\bar{x}}\vec{G})=0$ for $k = 1,\cdots,k_{\max}, \bar{x} \in
(0,\delta)$, and $[\mu_l(\bar{x})](j_{\bar{x}} \vec{F},j_{\bar{x}}\vec{G})$
is bounded on $(0,\delta)$ and tends to zero as $\bar{x} \rightarrow 0$, for
each $l=1,\cdots,l_{\max}$. We do not assume $\vec{F}$ or $\vec{G}$ is
semialgebraic.

\item[\LA{assertionII}] Suppose there exists an $(\vec{F},\vec{G})$ as in \eqref{assertionI}. Let $\vec{%
F}^{\#}=(F_{1}^{\#},\cdots ,F_{j_{\max }}^{\#})$, $\vec{G}%
^{\#}=(G_{1}^{\#},\cdots ,G_{i_{_{\max }}}^{\#})$, where the $F_{j}^{\#}$
and $G_{i}^{\#}$ are semialgebraic maps from $(0,\delta )\rightarrow
\mathfrak{p}$. Suppose that
\begin{equation*}
\lbrack \lambda _{k}(\bar{x})](\vec{F}^{\#}(\bar{x}),\vec{G}^{\#}(\bar{x}%
))=0,
\end{equation*}%
for $k=1,\cdots ,k_{\max },\bar{x}\in (0,\delta )$; and that $[\mu _{l}(\bar{%
x})](\vec{F}^{\#}(\bar{x}),\vec{G}^{\#}(\bar{x}))$ is bounded on $(0,\delta
) $ and tends to zero as $\bar{x}\rightarrow 0$. Then there exist $\delta
^{\prime }>0$ and $\vec{F}=(F_{1},\cdots ,F_{j_{\max }})$, $\vec{G}%
=(G_{1},\cdots ,G_{i_{\max }})$ semialgebraic and in $C^m(\Omega _{\delta
^{\prime }}^{\text{closure}},\R^{j_{\max}})$ and $C^m(\Omega _{\delta
^{\prime }}^{\text{closure}},\R^{i_{\max}})$ respectively, with $J_{(0,0)}\vec{F}=0,J_{(0,0)}\vec{G}=0$%
, $G_{i}=\sum_{j}A_{ij}F_{j}+B_{i}$ and $j_{\bar{x}}\vec{F}=\vec{F}^{\#}(\bar{x}),j_{%
\bar{x}}\vec{G}=\vec{G}^{\#}(\bar{x})$, for all $\bar{x}\in (0,\delta ^{\prime })$.
(Note that here we have passed from $\delta$ to a smaller $\delta^{\prime }$%
.)
\end{itemize}
\end{lemma}

The remainder of this section is devoted to a proof of Lemma \ref%
{main-lemma}.

Let $\delta>0$ be small enough to be picked below,

\begin{definition}
\label{bundledef} We define a bundle $\mathcal{H}$ over $[0,1]\times\{0%
\} \subset \mathbb{R}^2$. Here, $\mathcal{H}=(H(\bar{x},0))_{\bar{x}\in[0,1]%
} $, with $H(\bar{x},0) \subset \mathcal{P}^{j_{\max}+i_{\max}}$ defined as
follows.

\begin{itemize}
\item $H(0,0)=\{0\}$.

\item If $\bar{x} \in (0,1]$, then $(\vec{P},\vec{Q})=(P_1,\cdots,P_{j_{%
\max}},Q_1,\cdots,Q_{i_{\max}}) \in H(\bar{x},0)$ if and only if
\begin{equation*}
y^{|\alpha|-m}\partial^{\alpha}\left\{\sum_j A_{ij}P_j+B_i-Q_i \right\}(\bar{%
x},y) \rightarrow 0
\end{equation*}
as $y \rightarrow 0^+$, for each $|\alpha|\leq m$ and each $i$.
\end{itemize}
\end{definition}

We will show
that $\mathcal{H}$ is a bundle, i.e., $H(z)$ is a translate of an $\mathcal{R%
}_z$-submodule of $\mathcal{R}_z^{j_{\max}+i_{\max}}$ for each $z\in
[0,\delta]\times\{0\}$; and we will show that $J_{(\bar{x},0)}(\vec{F},\vec{G%
}) \in H(\bar{x},0)$ (each $\bar{x} \in [0,\delta]$) if $\vec{F},\vec{G}$
are as in \eqref{assertionI}.

Suppose $J_{(0,0)}(\vec{F},\vec{G})=0$, $\vec{F},\vec{G}$ are $C^m$ on $%
\Omega_{\delta}^{\text{closure}}$, $G_i=\sum_{j}A_{ij}F_j+B_i$ on $%
\Omega_{\delta}$. Let $\bar{x}\in (0,\delta]$. Then
\begin{equation*}
\partial^\alpha[A_{ij}(F_j-J_{(\bar{x},0)}F_j)](\bar{x},y)=o(y^{m-|\alpha|})
\end{equation*}
and
\begin{equation*}
\partial^\alpha[G_i-J_{(\bar{x},0)}G_i](\bar{x},y)=o(y^{m-|\alpha|})
\end{equation*}
on $\Omega_\delta$ for $|\alpha|\leq m$, by Taylor's theorem and our
estimates for $\partial^\alpha A_{ij}$. The above remarks imply that $%
\partial^\alpha \{\sum_j A_{ij} J_{(\bar{x},0)}F_j +B_i -J_{(\bar{x}%
,0)}G_i\}(\bar{x},0) =o(y^{m-|\alpha|}) $.

Therefore, $J_{(\bar{x},0)}(\vec{F},\vec{G})\in H(\bar{x},0)$ for $\bar{x}%
\in (0,\delta ]$. For $\bar{x}=0$, we just note that $J_{(0,0)}(\vec{F},\vec{%
G})=0\in H(0,0)$. That proves our assertion about $J_{(\bar{x},0)}(\vec{F},%
\vec{G})$.

Note that for $\bar{x}\not=0,$ $H\left( \bar{x}, 0\right) $ is a translate in $%
\mathcal{P}$ of
\begin{equation*}
I\left( \bar{x}\right) =\left\{ \left( \vec{P},\vec{Q}\right) :\partial
^{\alpha }\left( \sum_{j}A_{ij}P_{i}-Q_{i}\right) \left( \bar{x},y\right)
=o\left( y^{m-\left\vert \alpha \right\vert }\right) \text{, as }%
y\rightarrow 0^{+}\text{, }\left\vert \alpha \right\vert \leq m\right\}
\text{.}
\end{equation*}%
Let $\left( \vec{P},\vec{Q}\right) \in I\left( \bar{x}\right) $ and let $%
S\in \mathcal{P}$. Then for $\left\vert \alpha \right\vert \leq m,$ we have
\begin{equation*}
\partial ^{\alpha }\left( S\cdot \left[ \sum_{j}A_{ij}P_{j}-Q_{i}\right]
\right) \left( \bar{x},y\right) =o\left( y^{m-\left\vert \alpha \right\vert
}\right) ,
\end{equation*}%
hence
\begin{equation}
\partial ^{\alpha }\left( \sum_{j}A_{ij}\left( SP_{j}\right) -\left(
SQ_{i}\right) \right) \left( \bar{x},y\right) =o\left( y^{m-\left\vert
\alpha \right\vert }\right) \text{, as }y\rightarrow 0^{+}\text{.}
\label{+1}
\end{equation}

Also, our estimates on $\partial ^{\alpha }A_{ij},$ together with Taylor's
theorem, give
\begin{equation*}
\partial ^{\alpha }\left( A_{ij}\left( SP_{i}-J_{\left( \bar{x},0\right)
}\left( SP_{j}\right) \right) \right) \left( \bar{x},0\right) =o\left(
y^{m-\left\vert \alpha \right\vert }\right)
\end{equation*}%
and
\begin{equation*}
\partial ^{\alpha }\left( SQ_{i}-J_{\left( \bar{x},0\right) }\left(
SQ_{i}\right) \right) \left( \bar{x},0\right) =o\left( y^{m-\left\vert
\alpha \right\vert }\right) \text{ as }y\rightarrow 0^{+}\text{ for }%
\left\vert \alpha \right\vert \leq m\text{.}
\end{equation*}%
That is,
\begin{equation}
\partial ^{\alpha }\left( A_{ij}\left( SP_{j}-S\odot _{\left( \bar{x}%
,0\right) }P_{j}\right) \right) \left( \bar{x},y\right) =o\left(
y^{m-\left\vert \alpha \right\vert }\right)  \label{+2}
\end{equation}%
and
\begin{equation}
\partial ^{\alpha }\left( SQ_{i}-S\odot _{\left( \bar{x},0\right)
}Q_{i}\right) \left( \bar{x},0\right) =o\left( y^{m-\left\vert \alpha
\right\vert }\right) \text{ as }y\rightarrow 0^{+}\text{ for }\left\vert
\alpha \right\vert \leq m.  \label{+3}
\end{equation}

It now follows from (\ref{+1}), (\ref{+2}), and (\ref{+3}) that
\begin{equation*}
\partial ^{\alpha }\left( \sum_{j}A_{ij}\left[ S\odot _{\left( \bar{x}%
,0\right) }P_{j}\right] -\left[ S\odot _{\left( \bar{x},0\right) }Q_{i}%
\right] \right) \left( \bar{x},y\right) =o\left( y^{m-\left\vert \alpha
\right\vert }\right)
\end{equation*}%
as $y\rightarrow 0^{+}$, for each $\left\vert \alpha \right\vert \leq m$.

This completes the proof that the $I\left( \bar{x}\right) $ is a submodule, when $%
\bar{x}\not=0$.

For $\bar{x}=0,$ we just note that $\left\{ 0\right\} $ is an $\mathcal{R}%
_{\left( 0,0\right) }$-submodule of $\mathcal{R}_{\left( 0,0\right)
}^{j_{\max }+i_{\max }}.$

We have now shown that

\begin{itemize}
\item $\mathcal{H}=(H(\bar{x},0))_{\bar{x} \in [0,\delta]}$ is a bundle.

\item If $(\vec{F},\vec{G})$ is as in (I) of Lemma \ref{main-lemma}, then $(%
\vec{F},\vec{G})$ is a section of $\mathcal{H}$.

\item $H(\bar{x},0)\subset \mathcal{P}^{j_{\max }+i_{\max }}$ depends
semialgebraically on $\bar{x}$, since $A_{ij}$ and
$B_{i}$ are semialgebraic.
\end{itemize}

\begin{lemma}
\label{lemma-on-semialgebraic-bundles} Let $\mathcal{H}=(H(\bar{x},0))_{(%
\bar{x},0)\in \lbrack 0,\delta ]\times \{0\}}$ be a semialgebraic bundle, $%
\mathcal{H}=H(\bar{x},0)\subset \mathcal{P}^{j_{\max }+i_{\max }}$. Then
there exist semialgebraic functions $\lambda _{1},\cdots ,\lambda _{k_{\max
}}:(0,\delta )\rightarrow \text{Aff}$, and a finite set of bad points $\{%
\bar{\bar{x}}_{1}^{\text{bad}},\cdots ,\bar{\bar{x}}_{S}^{\text{bad}}\}$
such that the following holds for any $\bar{\bar{x}}\in \left( 0,\delta
\right) $ other than the bad points. Let $\left( \vec{F},\vec{G}\right)
=\left( F_{1},\cdots ,F_{j_{\max }},G_{1},\cdots ,G_{i_{\max }}\right) $ be $%
C^m$ in a neighborhood of $\left( \bar{\bar{x}},0\right) $ in $\mathbb{R}%
^{2}$. Then
\begin{equation*}
J_{\left( \bar{x},0\right) }\left( \vec{F},\vec{G}\right) \in H\left( \bar{x}%
,0\right) \text{ for all }\bar{x}\text{ in some neighborhood of }\bar{\bar{x}%
}
\end{equation*}%
if and only if
\begin{equation*}
\left[ \lambda _{k}\left( \bar{x}\right) \right] \left( j_{\bar{x}}\vec{F}%
,j_{\bar{x}}\vec{G}\right) =0\text{ for all }\bar{x}\text{ in some
neighborhood of }\bar{\bar{x}}, (k=1,\cdots,k_{\max}) \text{.}
\end{equation*}
\end{lemma}

\begin{proof}
This is a 1 dimensional case of Lemma \ref%
{semialgebraic-bundle-lemma}, whose proof can be found in Section \ref%
{section-semialgebraic-bundle}.
\end{proof}

\begin{proof}[Proof of Lemma \ref{main-lemma}]
We apply Lemma \ref{lemma-on-semialgebraic-bundles} to the bundle $\mathcal{H%
}$ defined in Definition \ref{bundledef}. By making $\delta $ smaller, we
may assume there are no bad points $\bar{\bar{x}}_{^{\text{bad}}}$. Thus, we
have achieved the following: There exist semialgebraic functions $\lambda
_{1},\cdots ,\lambda _{k_{\max }}:(0,\delta ]\rightarrow \text{Aff}$ such that for any $%
\bar{\bar{x}}\in (0,\delta )$ and any $(\vec{F},\vec{G})$ that is $%
C^m$ in a neighborhood of $(\bar{\bar{x}},0)$, we have
\begin{equation*}
J_{\left( \bar{x},0\right) }(\vec{F},\vec{G})\in H(\bar{x},0)\text{ for all }%
\bar{x}\text{ in some neighborhood of }\bar{\bar{x}}
\end{equation*}%
if and only if
\begin{equation*}
\left[ \lambda _{k}\left( \bar{x}\right) \right] j_{\bar{x}}\left( \vec{F},%
\vec{G}\right) =0\text{ for all }\bar{x}\text{ in some neighborhood of }\bar{%
\bar{x}}, (k=1,\cdots,k_{\max})\text{.}
\end{equation*}%
In particular, if $\left( \vec{F},\vec{G}\right) $ is as in \eqref{assertionI}, then
\begin{equation*}
\left[ \lambda _{k}\left( \bar{x}\right) \right] j_{\bar{x}}\left( \vec{F},%
\vec{G}\right) =0\text{ for all }\bar{x}\in (0,\delta ), (k=1,\cdots,k_{\max})\text{.}
\end{equation*}

Next, we apply Theorem \ref{helly-theorem} in Section \ref{hltv}.

Recall $H(\bar{x},0)$ is an affine space, so $\mathbb{R}\cdot H(\bar{x},0)$
is a vector space.

We regard $\mathbb{R}\cdot H(\bar{x},0)$ as the space of all $(\vec{P},\vec{Q%
},t)$ such that $\partial^\alpha \{\sum_j A_{ij}P_j+tB_i-Q_i \}(\bar{x}%
,y)=o(y^{m-|\alpha |})$ as $y \rightarrow 0^+$.

We define seminorms on $\mathbb{R}\cdot H(\bar{x},0)$ by
\begin{equation*}
|||(\vec{P},\vec{Q},t)|||_{\alpha,i,y}=\left|y^{|\alpha|-m}\partial^\alpha
\left\{\sum_j A_{ij}P_j+tB_i-Q_i \right\}(\bar{x},y)\right|
\end{equation*}
for fixed $\bar{x}$ and $0<y<\psi(\bar{x})$. Notice that on $H(\bar{x},0)$, the seminorm agrees with
\begin{equation*}
|||(\vec{P},\vec{Q})|||_{\alpha,i,y}=\left|y^{|\alpha|-m}\partial^\alpha
\left\{\sum_j A_{ij}P_j+B_i-Q_i \right\}(\bar{x},y)\right|
\end{equation*}
for fixed $\bar{x} \not=0$ and $0<y<\psi(\bar{x}), |\alpha|\leq m, i
=1,\cdots,i_{\max}$.

Note that
\begin{equation*}
\sup_{\alpha ,i,y}|||(\vec{P},\vec{Q})|||_{\alpha ,i,y}
\end{equation*}%
is bounded for fixed $\left( \vec{P},\vec{Q}\right) \in H\left( \bar{x}%
,0\right) $, by definition of $H\left( \bar{x},0\right) $.

Thus, by Theorem \ref{helly-theorem} in Section \ref{hltv}, for each $\bar{x}%
\in \left( 0,\delta \right) ,$ there exist $y_{\sigma }\in \left( 0,\psi
\left( \bar{x}\right) \right) $ $\left( \sigma =1,\cdots ,\sigma _{\max
}\right) $ with $\sigma _{\max }$ depending only on $i_{\max },j_{\max },m$
such that for any $\left( \vec{P},\vec{Q}\right) \in H\left( \bar{x}%
,0\right) $, we have
\begin{eqnarray}
&&\sup_{\substack{ 0<y<\psi \left( \bar{x}\right)  \\ \left\vert \alpha
\right\vert \leq m  \\ i=1,\cdots ,i_{\max }}}\left\vert y\right\vert
^{\left\vert \alpha \right\vert -m}\left\vert \partial ^{\alpha }\left\{
\sum_{j}A_{ij}P_{j}+B_{j}-Q_{i}\right\} \left( \bar{x},y\right) \right\vert
\notag \\
&\leq &C\max_{\substack{ \sigma =1,\cdots ,\sigma _{\max }  \\ \left\vert
\alpha \right\vert \leq m  \\ i=1,\cdots ,i_{\max }}}\left\vert y_{\sigma
}\right\vert ^{\left\vert \alpha \right\vert -m}\left\vert \partial ^{\alpha
}\left\{ \sum_{j}A_{ij}P_{j}+B_{j}-Q_{i}\right\} \left( \bar{x},y_{\sigma
}\right) \right\vert \text{.}  \label{!}
\end{eqnarray}

Moreover, (\ref{!}) is a semialgebraic condition. Therefore, we may take $%
y_{1},\cdots ,y_{\sigma }\in \left( 0,\psi \left( \bar{x}\right) \right) $
satisfying (\ref{!}) to depend semialgebraically on $\bar{x}\in \left(
0,\delta \right) $.

Because $0<y_{\sigma }\left( \bar{x}\right) <\psi \left( \bar{x}\right) \leq
\bar{x}$ for $\bar{x}\in \left( 0,\delta \right) $ and because $y_{\sigma
}\left( \bar{x}\right) $ depends semialgebraically on $\bar{x}$, we can take
$\delta$ small to achieve the estimates

\begin{itemize}
\item[\refstepcounter{equation}\text{(\theequation)}\label{Y1}] $\left\vert
\left( \frac{d}{dx}\right) ^{\alpha }y_{\sigma }\left( \bar{x}\right)
\right\vert \leq C\bar{x}^{1-\alpha }$ for $0\leq \alpha \leq m+100$, $%
\sigma =1,\cdots ,\sigma _{\max },$ $\bar{x}\in \left( 0,\delta \right) .$

\item[\refstepcounter{equation}\text{(\theequation)}\label{Y2}] $0<y_{\sigma
}(\bar{x})<\psi(\bar{x})\leq \bar{x}$ for $\sigma =1,\cdots ,\sigma _{\max }$, $\bar{x}\in
\left( 0,\delta \right) $.

\item[\refstepcounter{equation}\text{(\theequation)}\label{Y3}] $\bar{x}%
\mapsto y_{\sigma}(\bar{x})$ is a semialgebraic function.

\item[\refstepcounter{equation}\text{(\theequation)}\label{Y4}] For any $%
\bar{x}\in (0,\delta )$ and any $\left( \vec{P},\vec{Q}\right) =\left(
P_{1},\cdots ,P_{j_{\max }},Q_{1},\cdots ,Q_{i_{\max }}\right) \in H\left(
\bar{x},0\right) ,$ we have
\begin{eqnarray*}
&&\sup_{\substack{ 0<y<\psi \left( \bar{x}\right)  \\ \left\vert \alpha
\right\vert \leq m  \\ i=1,\cdots ,i_{\max }}}\left\vert y\right\vert
^{\left\vert \alpha \right\vert -m}\left\vert \partial ^{\alpha }\left\{
\sum_{j}A_{ij}P_{j}+B_{j}-Q_{i}\right\} \left( \bar{x},y\right) \right\vert
\\
&\leq &C\max_{\substack{ \sigma =1,\cdots ,\sigma _{\max }  \\ \left\vert
\alpha \right\vert \leq m  \\ i=1,\cdots ,i_{\max }}}\left\vert y_{\sigma
}\left( \bar{x}\right) \right\vert ^{\left\vert \alpha \right\vert
-m}\left\vert \partial ^{\alpha }\left\{
\sum_{j}A_{ij}P_{j}+B_{j}-Q_{i}\right\} \left( \bar{x},y_{\sigma }\left(
\bar{x}\right) \right) \right\vert \text{.}
\end{eqnarray*}%
with $C$ depending only on $i_{\max },j_{\max },m$.
\end{itemize}

Fix $\bar{x}\in (0,\delta )$, and let $(\vec{p},\vec{q})=\left( p_{1},\cdots
,p_{j_{\max }},q_{1},\cdots ,q_{i_{\max }}\right) \in \mathfrak{p}^{j_{\max
}+i_{\max }}$. Thus, each $p_{j}$ and $q_{i}$ is a polynomial in $y$ of
degree at most $m$. For $0\leq a\leq m,$ $\sigma =1,\cdots ,\sigma _{\max },$
$i=1,\cdots ,i_{\max },$ let
\begin{eqnarray*}
&&\mu _{a,\sigma ,i}^{\#}\left[ \bar{x}\right] \left( p_{1},\cdots
,p_{j_{\max }},q_{1},\cdots ,q_{i_{\max }}\right) \\
&=&\left. \left( y_{\sigma }\left( \bar{x}\right) \right) ^{a-m}\partial
_{y}^{a}\left\{ \sum_{j}A_{ij}\left( \bar{x},y\right) p_{j}\left( y\right)
+B_{i}\left( \bar{x},y\right) -q_{i}\left( y\right) \right\} \right\vert
_{y=y_{\sigma }\left( \bar{x}\right) }.
\end{eqnarray*}

Note that we don't take $x$-derivatives here, only $y$-derivatives.

The $\mu _{a,\sigma ,i}^{\#}(\bar{x})$ are affine functions from $\mathfrak{p%
}^{j_{\max }+i_{\max }}$ to $\mathbb{R}$; thus, each $\mu _{a,\sigma
,i}^{\#}(\bar{x})$ belongs to $\text{Aff}$. Let $\mu _{1}\left( \bar{x}%
\right) ,\cdots ,\mu _{l_{\max }}\left( \bar{x}\right) $ be an enumeration of
the $\mu _{a,\sigma ,i}^{\#}\left( \bar{x}\right) $, together with the
linear maps
\begin{eqnarray*}
\left( p_{1},\cdots ,p_{j_{\max }},q_{1},\cdots ,q_{i_{\max }}\right)
&\mapsto &\left( \bar{x}\right) ^{a-m}\partial _{y}^{a}p_{j}\left( 0\right)
\\
\left( p_{1},\cdots ,p_{j_{\max }},q_{1},\cdots ,q_{i_{\max }}\right)
&\mapsto &\left( \bar{x}\right) ^{a-m}\partial _{y}^{a}q_{i}\left( 0\right)
\text{.}
\end{eqnarray*}%
We will prove the following

\begin{itemize}
    \item[\LA{star}] Let $\vec{F},\vec{G}$ be as assumed in \eqref{assertionI}. Then, as $\bar{x}$ varies over $\left( 0,\delta \right)$, the $\left[ \mu _{l}\left( \bar{x}\right) \right] \left( j_{\bar{x}}%
\vec{F},j_{\bar{x}}\bar{G}\right)$ remain bounded, and these
quantities tend to zero as $\bar{x}$ tends to $0^{+}$.%
\end{itemize}

To prove (\ref{star}), we recall that
\begin{equation*}
\sum_{j}A_{ij}F_{j}+B_{i}-G_{i}=0\text{,}
\end{equation*}%
hence
\begin{eqnarray}
&&\mu _{a,\sigma ,i}^{\#}\left( \bar{x}\right) \left( j_{\bar{x}}\vec{F},j_{%
\bar{x}}\vec{G}\right)  \label{?0} \\
&=&-\left( y_{\sigma }\left( \bar{x}\right) \right) ^{a-m}\left. \partial
_{y}^{a}\left[ \sum_{i}A_{ij}\left( \bar{x},y\right) \left\{ F_{j}\left(
\bar{x},y\right) -j_{\bar{x}}F_{j}\left( y\right) \right\} -\left\{
G_{j}\left( \bar{x},y\right) -j_{\bar{x}}G_{j}\left( y\right) \right\} %
\right] \right\vert _{y=y_{\sigma }}.  \notag
\end{eqnarray}

Let $w_{F}\left( \bar{x}\right) =\max_{\left\vert \beta \right\vert
=m,j=1,\cdots ,j_{\max }}\left( \sup_{0\leq y\leq \psi \left( \bar{x}\right)
}\left[ \partial ^{\beta }F_{j}\left( \bar{x},y\right) \right] -\inf_{0\leq
y\leq \psi \left( \bar{x}\right) }\left[ \partial ^{\beta }F_{j}\left( \bar{x%
},y\right) \right] \right) $ and similarly define $w_{G}\left( \bar{x}%
\right) $ as above, with $G$ in place of $F.$

Because $\vec{F},\vec{G}$ belong to $C^m\left( {\Omega}^{\text{closure}}_{\delta
},\R^{j_{\max}}\right) $ and $C^m\left( {\Omega}^{\text{closure}}_{\delta
},\R^{i_{\max}}\right) $ respectively, while $\psi \left( \bar{x}\right) \rightarrow 0$ as $\bar{x}%
\rightarrow 0,$ we know that $w_{F}\left( \bar{x}\right) $, $w_{G}\left(
\bar{x}\right) $ are bounded as $\bar{x}$ varies over $\left( 0,\delta
\right) $, and moreover $w_{F}\left( \bar{x}\right) ,w_{G}\left( \bar{x}%
\right) \rightarrow 0$ as $\bar{x}\rightarrow 0^{+}$.

Taylor's theorem gives

\begin{itemize}
\item[\refstepcounter{equation}\text{(\theequation)}\label{?1}] $\left\vert
\partial _{y}^{a}\left[ F_{j}\left( \bar{x},y\right) -j_{\bar{x}}F_{j}\left(
y\right) \right] \right\vert \leq Cw_{F}\left( \bar{x}\right) \cdot y^{m-a}$
for $0\leq a\leq m$, $0<y<\psi \left( \bar{x}\right) ,$ $j=1,\cdots ,j_{\max
}$.

\item[\refstepcounter{equation}\text{(\theequation)}\label{?2}] $\left\vert
\partial _{y}^{a}\left\{ G_{i}\left( \bar{x},y\right) -j_{\bar{x}%
}G_{i}\left( y\right) \right\} \right\vert \leq Cw_{G}\left( \bar{x}\right)
\cdot y^{m-a}$ for $0\leq a\leq m,0<y<\psi \left( \bar{x}\right) ,i=1,\cdots
,i_{\max }$.
\end{itemize}

We recall that

\begin{itemize}
\item[\refstepcounter{equation}\text{(\theequation)}\label{?3}] $|\partial
_{y}^{a}A_{ij}(\bar{x},y)|\leq Cy^{-a}$ for $0\leq a\leq m,0<y<\psi \left(
\bar{x}\right) ,i=1,\cdots ,i_{\max },j=1,\cdots ,j_{\max }$.
\end{itemize}

Putting \eqref{?1},\eqref{?2},\eqref{?3} into \eqref{?0}, we find that%
\begin{equation*}
\left\vert \mu _{a,\sigma ,i}^{\#}\left( \bar{x}\right) \left( j_{\bar{x}}%
\vec{F},j_{\bar{x}}\vec{G}\right) \right\vert \leq Cw_{F}\left( \bar{x}%
\right) +Cw_{G}\left( \bar{x}\right) \text{,}
\end{equation*}%
hence the $\mu _{a,\sigma ,i}^{\#}\left( \bar{x}\right) \left( j_{\bar{x}}%
\vec{F},J_{\bar{x}}\vec{G}\right) $ remain bounded as $\bar{x}$ varies over $%
\left( 0,\delta \right) $, and these quantities tend to zero as $\bar{x}%
\rightarrow 0^{+}$.

Also, because $J_{\left( 0,0\right) }\vec{F}=0,J_{\left( 0,0\right) }\vec{G}%
=0,$ and $\vec{F},\vec{G}$ are in $C^m\left( {\Omega}_{\delta }^{\text{closure}},\R^{j_{\max}}\right)
$ and $C^m\left( {\Omega}_{\delta }^{\text{closure}},\R^{i_{\max}}\right)
$ respectively, we see that
\begin{equation*}
\left( \bar{x}\right) ^{a-m}\partial _{y}^{a}F_{j}\left( \bar{x},0\right)
,\left( \bar{x}\right) ^{a-m}\partial _{y}^{a}G_{i}\left( \bar{x},0\right) ,
\end{equation*}%
for $0\leq a\leq m$, remain bounded as $\bar{x}$ varies over $\left(
0,\delta \right) ,$ and these quantities tend to zero as $\bar{x}\rightarrow
0^{+}$.

Thus, all the $\mu _{l}\left( \bar{x}\right) \left( j_{\bar{x}}\vec{F},j_{%
\bar{x}}\vec{G}\right) $ remain bounded on $\left( 0,\delta \right) $ and
tend to zero as $\bar{x}\rightarrow 0^{+}$.

We have proven \eqref{star}. Thus, we have defined our $\lambda_1,\cdots,%
\lambda_{k_{\max}}$ and $\mu_1,\cdots,\mu_{l_{\max}}$ and we have proven
\eqref{assertionI}.

We now set out to prove \eqref{assertionII}.

Thus, let $\vec{F}^{\#}=\left( F_{1}^{\#},\cdots ,F_{j_{\max }}^{\#}\right) $
and $\vec{G}^{\#}=\left( G_{1}^{\#},\cdots ,G_{i_{\max }}^{\#}\right) $ be
as in \eqref{assertionII}.

Recall, each $F_{j}^{\#}$ and $G_{i}^{\#}$ is a semialgebraic map from $%
(0,\delta )$ into $\mathfrak{p}$, and moreover
\begin{equation*}
\left[ \lambda _{k}\left( \bar{x}\right) \right] \left( \vec{F}^{\#}\left(
\bar{x}\right) ,\vec{G}^{\#}\left( \bar{x}\right) \right) =0\text{ for }%
k=1,\cdots ,k_{\max },\text{ all }\bar{x}\in \left( 0,\delta \right); \text{ and}
\end{equation*}%
$\left[ \mu _{l}\left( \bar{x}\right) \right] \left( \vec{F}^{\#}\left( \bar{%
x}\right) ,\vec{G}^{\#}\left( \bar{x}\right) \right) $ is bounded as $\bar{x}
$ varies over $\left( 0,\delta \right) $ and tends to zero as $\bar{x}%
\rightarrow 0^{+}$ for each $l=1,\cdots ,l_{\max }$.

Then

\begin{itemize}
\item[\refstepcounter{equation}\text{(\theequation)}\label{Fjsharp}] $%
F_{j}^{\#}\left( \bar{x}\right) $ has the form $y\mapsto
\sum_{s=0}^{m}F_{js}\left( \bar{x}\right) y^{s}$ and

\item[\refstepcounter{equation}\text{(\theequation)}\label{Gisharp}] $%
G_{i}^{\#}\left( \bar{x}\right) $ has the form $y\mapsto
\sum_{s=0}^{m}G_{is}\left( \bar{x}\right) y^{s}$,
\end{itemize}
with $F_{js},G_{is}$
semialgebraic functions of one variable.
Taking $\delta $ small (depending on $\vec{F}^\#,\vec{G}^\#$), we may assume
the $F_{js},G_{is}$ are $C^\infty$ on $(0,\delta)$.

Now, we define $\vec{F}=\left( F_{1},\cdots ,F_{j_{\max }}\right) $, $\vec{G}%
=\left( G_{1},\cdots ,G_{i_{\max }}\right) ,\vec{G}^{\#\#}=\left(
G_{1}^{\#\#},\cdots ,G_{i_{\max }}^{\#\#}\right) ,$ where
\begin{equation}
F_{j}\left( \bar{x},y\right) =\sum_{s=0}^{m}F_{js}\left( \bar{x}\right) y^{s}
\label{fj}
\end{equation}%
for $\left( \bar{x},y\right) \in \left( 0,\delta \right) \times \mathbb{R}$,
$j=1,\cdots ,j_{\max },$%
\begin{equation}
G_{i}^{\#\#}\left( \bar{x},y\right) =\sum_{s=0}^{m}G_{is}\left( \bar{x}%
\right) y^{s}  \label{gi}
\end{equation}%
for $\left( \bar{x},y\right) \in \left( 0,\delta \right) \times \mathbb{R}$,
$i=1,\cdots ,i_{\max }$,
\begin{equation*}
G_{i}\left( \bar{x},y\right) =\sum_{j}A_{ij}\left( \bar{x},y\right)
F_{j}\left( \bar{x},y\right) +B_{i}\left( \bar{x},y\right)
\end{equation*}%
for $\left( \bar{x},y\right) \in \Omega _{\delta }$, $i=1,\cdots ,i_{\max }$.

Note that $F_{j},G_{i}^{\#\#}$ are $C^{\infty }$ functions on~$\left(
0,\delta \right) \times \mathbb{R}$ because the $F_{js},G_{is}$ are $%
C^{\infty }$ functions on $\left( 0,\delta \right) $.

The functions $F_{j},G_{i}^{\#\#},G_{i}$ are semialgebraic because $%
F_{j}^{\#},G_{i}^{\#}$ are semialgebraic.

Let $\bar{x}\in \left( 0,\delta \right) $. Then
\begin{equation}
j_{\bar{x}}F_{j} =F_{j}^{\#}\left( \bar{x}\right) \in
\mathfrak{p},j_{\bar{x}}G_{i}^{\#\#} =G_{i}^{\#}\left( \bar{x}%
\right) \in \mathfrak{p}.  \label{pml13-0}
\end{equation}

Therefore, for all $\bar{x}$ in a small neighborhood of a given $\overline{%
\overline{x}}\in \left( 0,\delta \right) $, we have
\begin{equation*}
\lambda _{k}(\bar{x})\left( j_{\bar{x}}\vec{F},j_{\bar{x}}\vec{G}^{\#\#}\right)
=\lambda _{k}(\bar{x})\left( \vec{F}^{\#}\left( \bar{x}\right) ,\vec{G}^{\#}\left(
\bar{x}\right) \right) =0
\end{equation*}%
for $k=1,\cdots ,k_{\max }$; the last equality is an assumption made in
\eqref{assertionII}.

Because $\vec{F},\vec{G}^{\#\#}$ are $C^{\infty }$ in a neighborhood of $%
\left( \overline{\overline{x}},0\right) ,$ the defining property of the $%
\lambda _{k}$ now tells us that
\begin{equation*}
\left( J_{\left( \bar{x},0\right) }\vec{F},J_{\left( \bar{x},0\right) }\vec{G%
}^{\#\#}\right) \in H\left( \bar{x},0\right)
\end{equation*}%
for all $\bar{x}$ in a small neighborhood of $\overline{\overline{x}}$.

Recalling that $\overline{\overline{x}}\in \left( 0,\delta \right) $ is
arbitrary, we conclude that%
\begin{equation}
\left( J_{\left( \bar{x},0\right) }\vec{F},J_{\left( \bar{x},0\right) }\vec{G%
}^{\#\#}\right) \in H\left( \bar{x},0\right) \text{ for all }\bar{x}\in
\left( 0,\delta \right) .  \label{pml13}
\end{equation}

By definition of $H\left( \bar{x},0\right) $ and by the estimates
\begin{eqnarray*}
\partial ^{\alpha }\left( F_{j}-J_{\left( \bar{x},0\right) }F_{j}\right)
\left( \bar{x},y\right) &=&o\left( y^{m-\left\vert \alpha \right\vert
}\right), \\
\partial ^{\alpha }\left( G_{i}^{\#\#}-J_{\left( \bar{x},0\right)
}G_{i}^{\#\#}\right) \left( \bar{x},y\right) &=&o\left( y^{m-\left\vert
\alpha \right\vert }\right), \text{ and} \\
\left\vert \partial ^{\alpha }A_{ij}\left( x,y\right) \right\vert &\leq
&Cy^{-\left\vert \alpha \right\vert }\text{,}
\end{eqnarray*}%
we therefore have the following:

\begin{itemize}
\item[\LA{newlabel}]For any $\bar{x}\in \left( 0,\delta \right) ,$ any $i=1,\cdots ,i_{\max },$
and any $\left\vert \alpha \right\vert \leq m$, the quantity
\begin{equation*}
y^{\left\vert \alpha \right\vert -m}\partial ^{\alpha }\left\{
\sum_{j}A_{ij}F_{j}+B_{i}-G_{i}^{\#\#}\right\} \left( \bar{x},y\right)
\end{equation*}%
is bounded as $y$ varies over $\left( 0,\psi \left( \bar{x}\right) \right) $
and tends to zero as $y\rightarrow 0^{+}$. \end{itemize}

We don't yet know that the above convergence is uniform in $\bar{x}.$

Next, we recall from \eqref{assertionII} the assumption that the $\mu _{l}\left(
\bar{x}\right) \left( \vec{F}^{\#}\left( \bar{x}\right) ,\vec{G}^{\#}\left(
\bar{x}\right) \right) $ remain bounded as $\bar{x}$ varies over $\left(
0,\delta \right) $ and moreover these quantities tend to zero as $\bar{x}%
\rightarrow 0^{+}$.

Thus, the quantities

\begin{equation}
\left( y_{\sigma }\left( \bar{x}\right) \right) ^{a-m}\partial
_{y}^{a}\left\{ \sum_{j}A_{ij}F_{j}+B_{i}-G_{i}^{\#\#}\right\} \left( \bar{x}%
,y_{\sigma }\left( \bar{x}\right) \right) \label{pml14}
\end{equation}%
for $0\leq a\leq m,i=1,\cdots ,i_{\max },\sigma =1,\cdots ,\sigma _{\max }$,
remain bounded as $\bar{x}$ varies over $\left( 0,\delta \right) $, and tend
to zero as $\bar{x}\rightarrow 0^{+}$.

Because those quantities are semialgebraic functions of one variable, we may
pass to a smaller $\delta $ and assert for any $b$, say $0\leq b\leq m$,
that
\begin{equation}
\left( \frac{d}{d\bar{x}}\right) ^{b}\left\{ y_{\sigma }\left( \bar{x}%
\right) ^{a-m}\partial _{y}^{a}\left[ \sum_{j}A_{ij}F_{j}+B_{i}-G_{i}^{\#\#}%
\right] \left( \bar{x},y_{\sigma }\left( \bar{x}\right) \right) \right\}
=o\left( \bar{x}^{-b}\right)  \label{t0}
\end{equation}%
as $\bar{x}\rightarrow 0^{+}$ and this quantity is bounded for $\bar{x}$
bounded away from $0$.

For $0\leq a+b\leq m,$ we will check that
\begin{equation}
\left( \bar{x}\right) ^{a+b-m}\left( \frac{d}{d\bar{x}}\right) ^{b}\left\{
\partial _{y}^{a}\left[ \sum_{j}A_{ij}F_{j}+B_{i}-G_{i}^{\#\#}\right] \left(
\bar{x},y_{\sigma }\left( \bar{x}\right) \right) \right\} =o\left( 1\right)
\label{t1}
\end{equation}%
as $\bar{x}\rightarrow 0^{+}$ and the left-hand side is bounded.

To see this, we write
\begin{eqnarray*}
&&\left( \frac{d}{d\bar{x}}\right) ^{b}\left\{ \partial _{y}^{a}\left[
\sum_{j}A_{ij}F_{j}+B_{i}-G_{i}^{\#\#}\right] \left( \bar{x},y_{\sigma
}\left( \bar{x}\right) \right) \right\} \\
&=&\left( \frac{d}{d\bar{x}}\right) ^{b}\left\{ \left( y_{\sigma }\left(
\bar{x}\right) \right) ^{m-a}\left( y_{\sigma }\left( \bar{x}\right) \right)
^{a-m}\partial _{y}^{a}\left[ \sum_{j}A_{ij}F_{j}+B_{i}-G_{i}^{\#\#}\right]
\left( \bar{x},y_{\sigma }\left( \bar{x}\right) \right) \right\} \\
&=&\sum_{b^{\prime }+b^{\prime \prime }=b}\text{coeff}\left( b^{\prime
},b^{\prime \prime }\right) \underset{\left( \dagger \right) }{\underbrace{%
\left[ \left( \frac{d}{d\bar{x}}\right) ^{b^{\prime }}\left( y_{\sigma
}\left( \bar{x}\right) \right) ^{m-a}\right] }}\cdot \\
&&\underset{\left( \ddagger \right) }{\underbrace{\left[ \left( \frac{d}{d%
\bar{x}}\right) ^{b^{\prime \prime }}\left\{ \left( y_{\sigma }\left( \bar{x}%
\right) \right) ^{a-m}\partial _{y}^{a}\left[
\sum_{j}A_{ij}F_{j}+B_{i}-G_{i}^{\#\#}\right] \left( \bar{x},y_{\sigma
}\left( \bar{x}\right) \right) \right\} \right] }}\text{.}
\end{eqnarray*}%
Since $y_{\sigma }\left( \bar{x}\right) $ is given by a Puiseux series for $%
\bar{x}\in \left( 0,\delta \right) $ (small enough $\delta $), $$\left(
\dagger \right) =O\left( y_{\sigma }\left( \bar{x}\right) \right)
^{m-a}\cdot \bar{x}^{-b^{\prime }}=O\left( y_{\sigma }\left( \bar{x}\right)
^{m-a-b^{\prime }}\right) ,$$ because $0<y_{\sigma }\left( \bar{x}\right)
<\psi \left( \bar{x}\right) \leq \bar{x}$. By \eqref{t0}, $\left( \ddagger
\right) $ is $o\left( \bar{x}^{-b^{\prime \prime }}\right) $ as $\bar{x}%
\rightarrow 0^{+}$.

So in fact, we get not only \eqref{t1} but the stronger result
\begin{equation}
\left( \frac{d}{d\bar{x}}\right) ^{b}\left\{ \partial _{y}^{a}\left[
\sum_{j}A_{ij}F_{j}+B_{i}-G_{i}^{\#\#}\right] \left( \bar{x},y_{\sigma
}\left( \bar{x}\right) \right) \right\} =o\left( y_{\sigma }\left( \bar{x}%
\right) ^{m-a}\cdot \bar{x}^{-b}\right)  \label{t2}
\end{equation}%
as $\bar{x}\rightarrow 0^{+};$ the left-hand side is bounded.

Introduce the vector field $X_{\sigma }=\frac{\partial }{\partial x}%
+y_{\sigma }^{\prime }\left( \bar{x}\right) \frac{\partial }{\partial y}$ on
$\mathbb{R}^{2}.$ We have
\begin{equation*}
\left( \frac{d}{d\bar{x}}\right) ^{b}\left\{ \mathcal{F}\left( \bar{x}%
,y_{\sigma }\left( \bar{x}\right) \right) \right\} =\left. \left( X_{\sigma
}\right) ^{b}\mathcal{F}\right\vert _{\left( \bar{x},y_{\sigma }\left( \bar{x%
}\right) \right) }\text{ for any }\mathcal{F}\in C_{loc}^{b}\left( \mathbb{R}%
^{2}\right) \text{.}
\end{equation*}%
Therefore, \eqref{t2} yields
\begin{equation}
\left( X_{\sigma }^{b}\partial _{y}^{a}\right) \left[
\sum_{j}A_{ij}F_{j}+B_{i}-G_{i}^{\#\#}\right] \left( \bar{x},y_{\sigma
}\left( \bar{x}\right) \right) =o\left( y_{\sigma }\left( \bar{x}\right)
^{m-a}\cdot \bar{x}^{-b}\right) \text{ as }\bar{x}\rightarrow 0^{+}
\label{t3}
\end{equation}%
and the left-hand side is bounded for all $\bar{x}$, for $%
a+b\leq m,\sigma =1,\cdots ,\sigma _{\max },i=1,\cdots ,i_{\max }$.

This implies that

\begin{itemize}
\item[\refstepcounter{equation}\text{(\theequation)}\label{t4}] $\left(
y_{\sigma }\left( \bar{x}\right) \right) ^{\left\vert \alpha \right\vert
-m}\partial ^{\alpha }\left[ \sum_{j}A_{ij}F_{j}+B_{i}-G_{i}^{\#\#}\right]
\left( \bar{x},y_{\sigma }\left( \bar{x}\right) \right) $ is bounded
on $(0,\delta )$ and tends to zero as $\bar{x}\rightarrow
0^{+}$, for $\left\vert \alpha \right\vert \leq m,i=1,\cdots ,i_{\max
},\sigma =1,\cdots ,\sigma _{\max }$.
\end{itemize}

Let $\alpha =\left( b,a\right) ,$ $\partial ^{\alpha }=\partial
_{x}^{b}\partial _{y}^{a}$.

We deduce \eqref{t4} from \eqref{t3} by induction on $b$. For $b=0,$ %
\eqref{t4} is the same as \eqref{t3}.

Assume we know \eqref{t4} for all $b^{\prime }<b.$ We prove \eqref{t4} for
the given $b,$ using our induction hypothesis for $b^{\prime }$, together
with \eqref{t3}.

The quantity
\begin{equation}
X_{\sigma }^{b}\partial _{y}^{a}\left\{
\sum_{j}A_{ij}F_{j}+B_{i}-G_{i}^{\#\#}\right\} \left( \bar{x},y_{\sigma
}\left( \bar{x}\right) \right)  \label{LHS}
\end{equation}%
is a sum of terms of the form
\begin{equation}
\left( \partial _{x}^{b_{1}}y_{\sigma }\left( \bar{x}\right) \right) \cdot
\cdots \cdot \left( \partial _{x}^{b_{\nu }}y_{\sigma }\left( \bar{x}\right)
\right) \cdot \partial _{x}^{\bar{b}}\partial _{y}^{a+\nu }\left\{
\sum_{j}A_{ij}F_{j}+B_{i}-G_{i}^{\#\#}\right\} \left( \bar{x},y_{\sigma
}\left( \bar{x}\right) \right)  \label{RHS}
\end{equation}%
with $b_{t}\geq 1$ each $t$, $b_{1}+\cdots +b_{\nu }+\bar{b}=b.$

Note $\bar{b}+\left( a+\nu \right) =a+\bar{b}+b_{1}+\cdots +b_{\nu }-\left(
b_{1}-1\right) -\cdots -\left( b_{\nu }-1\right) \leq a+b$.

We know that $\eqref{LHS}=o\left( y_{\sigma }\left( \bar{x}\right)
^{m-a-b}\right) $ by \eqref{t3}.

If $\bar{b}<b,$ then by our induction hypothesis, the term \eqref{RHS} is
dominated by
\begin{eqnarray*}
&&O\left( \overset{\text{Here again we use }0<y_{\sigma }<\bar{x}\text{.}}{%
\overbrace{y_{\sigma }\left( \bar{x}\right) ^{-\left[ b_{1}-1\right] -\cdots
-\left[ b_{\nu }-1\right] }}}\right) \cdot o\left( y_{\sigma }\left( \bar{x}%
\right) ^{m-\left[ a+\nu \right] -\bar{b}}\right) \\
&=&o\left( y_{\sigma }\left( \bar{x}\right) ^{m-a-\bar{b}-b_{1}-\cdots
-b_{\nu }}\right) =o\left( y_{\sigma }\left( \bar{x}\right) ^{m-a-b}\right)
\text{.}
\end{eqnarray*}%
Therefore, in the equation $\eqref{LHS}=\sum \eqref{RHS}$, all terms are $%
o\left( y_{\sigma }\left( \bar{x}\right) ^{m-a-b}\right) $, except possibly
the term arising from $\bar{b}=b$, which is
\begin{equation*}
\partial _{x}^{b}\partial _{y}^{a}\left\{
\sum_{j}A_{ij}F_{j}+B_{i}-G_{i}^{\#\#}\right\} \left( \bar{x},y_{\sigma
}\left( \bar{x}\right) \right) \text{.}
\end{equation*}%
Therefore,
\begin{equation*}
\partial _{x}^{b}\partial _{y}^{a}\left\{
\sum_{j}A_{ij}F_{j}+B_{i}-G_{i}^{\#\#}\right\} \left( \bar{x},y_{\sigma
}\left( \bar{x}\right) \right) =o\left( y_{\sigma }\left( \bar{x}\right)
^{m-a-b}\right) \text{, as }\bar{x}\rightarrow 0^{+}\text{.}
\end{equation*}%
This completes our induction on $b$, proving \eqref{t4}$.$

Thus,

\begin{itemize}
\item[\refstepcounter{equation}\text{(\theequation)}\label{pml18-1}] $\max_{\substack{ \sigma =1,\cdots ,\sigma _{\max }  \\ i=1,\cdots ,i_{\max
}  \\ \left\vert \alpha \right\vert \leq m}}\left( y_{\sigma }\left( \bar{x}%
\right) \right) ^{\left\vert \alpha \right\vert -m}\left\vert \partial
^{\alpha }\left\{ \sum_{j}A_{ij}F_{j}+B_{i}-G_{i}^{\#\#}\right\} \left( \bar{%
x},y_{\sigma }\left( \bar{x}\right) \right) \right\vert \text{ }$ is bounded on $\left( 0,\delta \right)$ and tends to zero as $\bar{x}$  tends to $0^{+}$.
\end{itemize}
Recall that our $\mu _{l}(\bar{x})$ include the affine maps $(p_{1},\cdots
,p_{j_{\max }},q_{1},\cdots ,q_{i_{\max }})\mapsto \bar{x}^{a-m}\partial
_{y}^{a}p_{j}(0)$ and $(p_{1},\cdots ,p_{j_{\max }},q_{1},\cdots ,q_{i_{\max
}})\mapsto \bar{x}^{a-m}\partial _{y}^{a}q_{i}(0)$ for $0\leq a\leq m.$ Our
assumption on the $\mu $'s made in \eqref{assertionII}
tells us therefore that $\bar{x}^{a-m}\partial _{y}^{a}\left(
F_{j}^{\#}\left( \bar{x}\right) \right) \left( 0\right) $ and $\bar{x}%
^{a-m}\partial _{y}^{a}\left( G_{i}^{\#}\left( \bar{x}\right) \right) \left(
0\right) $ are bounded on $\left( 0,\delta \right) $ and tend to zero as $%
\bar{x}\rightarrow 0^{+}$.

That is,

\begin{itemize}
\item[\refstepcounter{equation}\text{(\theequation)}\label{pml19}] $\bar{x}%
^{s-m}F_{js}\left( \bar{x}\right) ,\bar{x}^{s-m}G_{is}\left( \bar{x}\right) $
are bounded on $\left( 0,\delta \right) $ and tend to zero as $\bar{x}%
\rightarrow 0^{+}$. $\left( 0\leq s\leq m\right) $.

\item[\refstepcounter{equation}\text{(\theequation)}\label{pml20}] Because $%
F_{js},G_{js}$ are semialgebraic functions of one variable, it follows that,
for $s,t\leq m$, the functions
\begin{equation*}
\left( \frac{d}{d\bar{x}}\right) ^{t}F_{js}\left( \bar{x}\right) ,\left(
\frac{d}{d\bar{x}}\right) ^{t}G_{is}\left( \bar{x}\right)
\end{equation*}%
are bounded on $\left( 0,\delta \right) $ if $s+t\leq m$ and are $o\left(
\bar{x}^{m-s-t}\right) $ as $\bar{x}\rightarrow 0^{+}$ (even if $s+t>m$).
\end{itemize}

Recalling now the definitions of the $F_{j}$ and $G_{i}^{\#\#}$ in terms of
the $F_{j},G_{is}$ (see \eqref{fj}, \eqref{gi}), we conclude that
\begin{eqnarray*}
\partial _{\bar{x}}^{t}\partial _{y}^{s}F_{j}\left( \bar{x},y\right)
&=&\sum_{m\geq \underline{s}\geq s}\left[ \left( \frac{d}{d\bar{x}}\right)
^{t}F_{j\underline{s}}\left( \bar{x}\right) \right] \left( \text{coefficient
}\left( \underline{s},s\right) \right) \cdot y^{\underline{s}-s} \\
&=&\sum_{m\geq \underline{s}\geq s}o\left( \bar{x}^{m-t-\underline{s}%
}\right) \cdot y^{\underline{s}-s}\text{.}
\end{eqnarray*}%
If $s+t=m,$ then this is equal to $o\left( \frac{y}{\bar{x}}\right) ^{%
\underline{s}-s}=o\left( 1\right) $ for $0<y<\psi \left( \bar{x}\right) \leq
\bar{x}$.

Therefore, for $\left\vert \beta \right\vert =m,$ we have $\left\vert
\partial ^{\beta }F_{j}\left( \bar{x},y\right) \right\vert =o\left( 1\right)
$ as $\left( \bar{x},y\right) \in \Omega _{\delta }$ tends to zero.

Similarly, $\left\vert \partial ^{\beta }G_{i}^{\#\#}\left( \bar{x},y\right)
\right\vert =o\left( 1\right) $ as $\left( \bar{x},y\right) \in \Omega
_{\delta }$ tends to zero.

That is, for $\left\vert \beta \right\vert =m$, the functions $\partial
^{\beta }F_{j}\left( \bar{x},y\right) $ and $\partial ^{\beta
}G_{i}^{\#\#}\left( \bar{x},y\right) $ are bounded on $\Omega _{\delta }$
and they tend to zero as $\bar{x}\rightarrow 0^{+}$ (keeping $\left( \bar{x}%
,y\right) \in \Omega _{\delta }$).

Let $\mathcal{E}\left( \bar{x}\right) =\sup \left\{ \left\vert \partial
^{\beta }F_{j}\left( \bar{x},y\right) \right\vert ,\left\vert \partial
^{\beta }G_{i}^{\#\#}\left( \bar{x},y\right) \right\vert :\left\vert \beta
\right\vert =m,0<y<\psi \left( \bar{x}\right) \text{ (all } i, j ) \right\} .$

Then
\begin{equation}
\mathcal{E}\left( \bar{x}\right) \text{is bounded on }\left( 0,\delta
\right) \text{ and tends to zero as }\bar{x}\rightarrow 0^{+}.  \label{b1}
\end{equation}

By Taylor's theorem,
\begin{equation*}
\left\vert \partial ^{\alpha }\left\{ F_{j}-J_{\left( \bar{x},0\right)
}F_{j}\right\} \left( \bar{x},y\right) \right\vert \leq Cy^{m-\left\vert
\alpha \right\vert }\mathcal{E}\left( \bar{x}\right) \text{ for }\left\vert
\alpha \right\vert \leq m,\left( \bar{x},y\right) \in \Omega _{\delta }\text{%
.}
\end{equation*}%
Recall that
\begin{equation*}
\left\vert \partial ^{\alpha }A_{ij}\left( \bar{x},y\right) \right\vert \leq
Cy^{-\left\vert \alpha \right\vert }\text{ for }\left\vert \alpha
\right\vert \leq m\text{ and }\left( \bar{x},y\right) \in \Omega _{\delta }%
\text{.}
\end{equation*}

Just as we estimated the functions $F_{j}$ above, we have from Taylor's theorem
that
\begin{equation*}
\left\vert \partial ^{\alpha }\left\{ G_{i}^{\#\#}-J_{\left( \bar{x}%
,0\right) }G_{i}^{\#\#}\right\} \left( \bar{x},y\right) \right\vert \leq
Cy^{m-\left\vert \alpha \right\vert }\mathcal{E}\left( \bar{x}\right) \text{
for }\left\vert \alpha \right\vert \leq m,\left( \bar{x},y\right) \in \Omega
_{\delta }\text{.}
\end{equation*}%
Combining these estimates, we see that
\begin{eqnarray}
&&\left\vert \partial ^{\alpha}\left\{ \sum_{j}A_{ij}\left( F_{j}-J_{\left( \bar{x%
},0\right) }F_{j}\right) -\left( G_{i}^{\#\#}-J_{\left( \bar{x},0\right)
}G_{i}^{\#\#}\right) \right\} \left( x,y\right) \right\vert  \notag \\
&\leq &Cy^{m-\left\vert \alpha \right\vert }\mathcal{E}\left( \bar{x}\right)
\text{ for }\left\vert \alpha \right\vert \leq m,\left( \bar{x},y\right) \in
\Omega _{\delta }\text{.}  \label{b2}
\end{eqnarray}%
Combining \eqref{pml18-1}, \eqref{b1}, \eqref{b2}, we see that
\begin{eqnarray}
&&\left( y_{\sigma }\left( \bar{x}\right) \right) ^{\left\vert \alpha
\right\vert -m}\partial ^{\alpha }\left\{ \sum_{j}A_{ij}\left[ J_{\left(
\bar{x},0\right) }F_{j}\right] +B_{i}-\left[ J_{\left( \bar{x},0\right)
}G_{i}^{\#\#}\right] \right\} \left( \bar{x},y_{\sigma }\left( \bar{x}%
\right) \right)  \label{pml22} \\
&&\text{is bounded on }\left( 0,\delta \right) \text{ and tends to 0 as }%
\bar{x}\text{ tends to }0^{+}\text{.}  \notag
\end{eqnarray}

Recall that $\left( J_{(\bar{x},0)}\vec{F},J_{\left( \bar{x},0\right) }\vec{G}%
^{\#\#}\right) \in H\left( \bar{x}\right) $ for all $\bar{x}\in (0,\delta ]$
(see \eqref{pml13}).

The above results, together with the property \eqref{Y4} of the $y_{\sigma
}\left( \bar{x}\right) $ now tells us that

\begin{itemize}
\item[\refstepcounter{equation}\text{(\theequation)}\label{pml22-2}] $%
y^{\left\vert \alpha \right\vert -m}\partial ^{\alpha }\left\{
\sum_{j}A_{ij}\left( J_{\left( \bar{x},0\right) }F_{j}\right) +B_{i}-\left(
J_{\left( \bar{x},0\right) }G_{i}^{\#\#}\right) \right\} \left( \bar{x}%
,y\right) $ is bounded on $\Omega _{\delta }$ and tends to zero as $\left(
\bar{x},y\right) \in \Omega _{\delta }$ tends to zero.
\end{itemize}

Together with \eqref{b1}, \eqref{b2}, this yields the following result

\begin{itemize}
\item[\refstepcounter{equation}\text{(\theequation)}\label{pml23-1}] $%
y^{\left\vert \alpha \right\vert -m}\partial ^{\alpha }\left\{
\sum_{j}A_{ij}F_{j} +B_{i}- G_{i}^{\#\#} \right\} \left( \bar{x},y\right) $
is bounded on $\Omega _{\delta }$ and tends to zero as $\left( \bar{x}%
,y\right) \in \Omega _{\delta }$ tends to zero. Here, $i=1,\cdots,i_{\max}$
and $|\alpha|\leq m$ are arbitrary.
\end{itemize}

From \eqref{newlabel}, we have

\begin{itemize}
\item[\refstepcounter{equation}\text{(\theequation)}\label{pml23-2}] $%
\lim_{y \rightarrow 0^+} y^{|\alpha|-m}\partial^\alpha \left( \sum_j A_{ij}
F_j +B_i-G_i^{\#\#}\right)(x,y)=0$ for each fixed $x \in (0,\delta)$.
\end{itemize}

The functions $A_{ij},F_j,B_i, G_i^{\#\#}$ are semialgebraic. Therefore, by
Lemma \ref{semialgebraic-function-lemma}, there exist a positive integer $K$
and a semialgebraic function of one variable $\mathcal{A}(x)$ such that

\begin{itemize}
\item[\refstepcounter{equation}\text{(\theequation)}\label{pml23-3}] $\left|
y^{|\alpha|-m}\partial^\alpha \left( \sum_j A_{ij} F_j
+B_i-G_i^{\#\#}\right)(x,y)\right| \leq \mathcal{A}(x) \cdot y^{\frac{1}{K}}$ for
all $(x,y) \in \Omega_\delta$, $|\alpha|\leq m, i=1,\cdots,i_{\max}$.
\end{itemize}

Taking $\delta$ smaller, we may assume $\mathcal{A}(x)$ is $C^\infty$ on $%
(0,\delta]$.

Consequently, $y^{|\alpha |-m}\partial ^{\alpha }\left(
\sum_{j}A_{ij}F_{j}+B_{i}-G_{i}^{\#\#}\right) (x,y)$ tends to zero as $%
y\rightarrow 0^{+}$, uniformly as $x$ varies over $(\varepsilon ,\delta )$ for
any $\varepsilon >0$.

Recalling that $G_{i}=\sum_{j}A_{ij}F_{j}+B_{i}$, we see that for $\left\vert \alpha \right\vert \leq m, i=1,\cdots ,i_{\max },$
\begin{equation}
y^{|\alpha |-m}\partial ^{\alpha }\left\{ G_{i}-G_{i}^{\#\#}\right\} \left(
x,y\right) \rightarrow 0 \label{doublestars}
\end{equation}
as $y\rightarrow 0^{+}$ uniformly for $x$ in each
interval $\left( \varepsilon ,\delta \right) $.

Recalling that $G_{i}^{\#\#}$ belongs to $C^{\infty }$ in a neighborhood of $%
\left( x,0\right) $ (each $x\in \left( 0,\delta \right) $), we conclude that
the derivatives $\partial ^{\alpha }G_{i}\left( x,y\right) $ ($\left\vert
\alpha \right\vert \leq m,i=1,\cdots ,i_{\max }$), initially defined on $%
\Omega _{\delta }=\left\{ \left( x,y\right) :0<x<\delta ,0<y<\psi \left(
x\right) \right\} $ extend to continuous functions on
\begin{equation}
    \Omega _{\delta }^{++}\equiv \left\{ \left( x,y\right) :0<x<\delta ,0\leq
y<\psi \left( x\right) \right\}. \label{Om++}
\end{equation}

Next, recall that $F_{js}$ is $C^{\infty }$ on $\left( 0,\delta \right) $
and that we assume that $\left\vert \partial ^{\alpha }A_{ij}\left(
x,y\right) \right\vert ,\left\vert \partial ^{\alpha }B_{i}\left( x,y\right)
\right\vert \leq Cy^{-\left\vert \alpha \right\vert }$ on

\begin{equation}\Omega
^{+}=\left\{ \left( x,y\right) :0<x<\delta ,0<y\leq \psi \left( x\right)
\right\} \label{omega+} \end{equation} on which the functions $\partial ^{\alpha }A_{ij},\partial
^{\alpha }B_{i}$ are assumed to be continuous.

We defined
\begin{eqnarray*}
G_{i} &=&\sum_{j}A_{ij}F_{j}+B_{i} \\
&=&\sum_{j}A_{ij}\left( x,y\right) \left[ \sum_{s=0}^{m}F_{js}\left(
x\right) y^{s}\right] +B_{i}\left( x,y\right) \text{.}
\end{eqnarray*}

The above remarks (and the fact that $\psi \left( x\right) \not=0$ for $x\in
\left( 0,\delta \right) $) show that $\partial ^{\alpha }G_{i}$ extends to
a continuous function on $\Omega ^{+}$ (see \eqref{omega+}), for $\left\vert \alpha
\right\vert \leq m,i=1,\cdots ,i_{\max }$.

Combining our results for $\Omega ^{+}$ (see \eqref{omega+}) and for $\Omega ^{++}$ (see \eqref{Om++}), we see that $%
\partial ^{\alpha }G_{i}$ extends to a continuous function on $\Omega _{%
\frac{2\delta }{3}}^{\text{closure}}\setminus \left\{ \left( 0,0\right)
\right\} $ for each $i=1,\cdots ,i_{\max },\left\vert \alpha \right\vert
\leq m$.

Also, $\partial ^{\alpha }F_{i}$ is a continuous function on $\Omega _{\frac{%
2}{3}\delta }^{\text{closure}}\setminus \left\{ \left( 0,0\right) \right\} $
because $F_{i}$ is $C^{\infty }$ on $\left( 0,\delta \right) \times \mathbb{R%
}$.

By \eqref{pml19}, we have $G_{is}(x)=o(x^{m-s})$ ($0\leq s\leq m$) on $%
(0,\delta )$. Because $G_{is}$ is semialgebraic, it follows that after
possibly reducing $\delta $, we have%
\begin{equation*}
\left( \frac{d}{dx}\right) ^{t}G_{is}\left( x\right) =o\left(
x^{m-s-t}\right) \text{ for }0\leq t\leq m,0\leq s\leq m,i=1,\cdots ,i_{\max
}\text{.}
\end{equation*}

Because $G_{i}^{\#\#}\left( x,y\right) =\sum_{\underline{s}=0}^{m}G_{i%
\underline{s}}\left( x\right) y^{\underline{s}}$ and $0<y<\psi \left(
x\right) \leq x$ on $\Omega _{\delta }$, we have on $\Omega _{\delta }$ that
\begin{eqnarray*}
\left\vert \partial _{x}^{t}\partial _{y}^{s}G_{i}^{\#\#}\left( x,y\right)
\right\vert &=&\left\vert \sum_{\underline{s}=s}^{m}\text{coeff}\left(
\underline{s},s\right) \cdot \left( \frac{d}{dx}\right) ^{t}G_{i\underline{s}%
}\left( x\right) \cdot y^{\underline{s}-s}\right\vert \\
&=&o\left( \sum_{\underline{s}=s}^{m}x^{m-\underline{s}-t}\cdot y^{%
\underline{s}-s}\right) \\
&=&o\left(\sum_{\underline{s}=s}^{m}x^{m-\underline{s}-t}\cdot x^{\underline{s}-s}\right)
\\
&=&o\left( x^{m-s-t}\right) \text{ on }\Omega _{\delta }\text{ for }s,t\leq m%
\text{.}
\end{eqnarray*}%
In particular,

\begin{itemize}
\item[\refstepcounter{equation}\text{(\theequation)}\label{pml26}] $\partial
^{\alpha }G_{i}^{\#\#}\left( x,y\right) \rightarrow 0$ as $\left( x,y\right)
\in \Omega _{\delta }$ tends to $\left( 0,0\right) $ for $\left\vert \alpha
\right\vert \leq m,i=1,\cdots ,i_{\max }$.
\end{itemize}

On the other hand, recalling the definition $G_{i}=\sum_{j}A_{ij}F_{j}+B_{i}$%
, we see from \eqref{pml23-1} that $\partial ^{\alpha }\left(
G_{i}-G_{i}^{\#\#}\right) \left( x,y\right) \rightarrow 0$ as $\left(
x,y\right) \in \Omega _{\delta }$ tends to $\left( 0,0\right) $ for each $%
\left\vert \alpha \right\vert \leq m$. Together with \eqref{pml26}, this
shows that $\partial ^{\alpha }G_{i}\left( x,y\right) \rightarrow 0$ as $%
\left( x,y\right) \in \Omega _{\delta }$ tends to $\left( 0,0\right) $ for
each $\left\vert \alpha \right\vert \leq m$.

Next, recall from \eqref{pml19} that $F_{js}(x)=o(x^{m-s})$ for $x\in
(0,\delta )$, $j=1,\cdots ,j_{\max },s=0,\cdots ,m$.

Because the $F_{jk}$ are semialgebraic functions of one variable, we conclude
(after reducing $\delta $) that $\left( \frac{d}{dx}\right) ^{t}F_{js}\left(
x\right) =o\left( x^{m-s-t}\right) $ on $\left( 0,\delta \right) $ for $%
t\leq m$.

Now, for $s+t\leq m$ and $\left( x,y\right) \in \Omega _{\delta }$ (hence $%
0<y<\psi \left( x\right) \leq x$), we have
\begin{eqnarray*}
\left\vert \left( \frac{\partial }{\partial y}\right) ^{s}\left( \frac{%
\partial }{\partial x}\right) ^{t}F_{j}\left( x,y\right) \right\vert
&=&\left\vert \left( \frac{\partial }{\partial y}\right) ^{s}\left( \frac{%
\partial }{\partial x}\right) ^{t}\sum_{\underline{s}=0}^{m}F_{j\underline{s}%
}\left( x\right) y^{\underline{s}}\right\vert \\
&=&\left\vert \sum_{\underline{s}=s}^{m}\text{coeff}\left( \underline{s}%
,s\right) \left[ \left( \frac{d}{dx}\right) ^{t}F_{j\underline{s}}\left(
x\right) \right] \cdot y^{\underline{s}-s}\right\vert \\
&\leq &C\sum_{\underline{s}=s}^{m}\left\vert \left( \frac{d}{dx}\right)
^{t}F_{j\underline{s}}\left( x\right) \right\vert \cdot x^{\underline{s}-s}
\\
&=&o\left( \sum_{\underline{s}=0}^{m}x^{m-\underline{s}-t}x^{\underline{s}-s}\right)
=o\left( x^{m-s-t}\right) \text{.}
\end{eqnarray*}%
Thus, for $\left\vert \alpha \right\vert \leq m$, and $j=1,\cdots ,j_{\max
}, $ we have
\begin{equation*}
\partial ^{\alpha }F_{j}\left( x,y\right) \rightarrow 0\text{ as }\left(
x,y\right) \in \Omega _{\delta }\text{ tends to }\left( 0,0\right) \text{.}
\end{equation*}

We now know the following: $G_{i}=\sum_{j}A_{ij}F_{j}+B_{i}$ on $\Omega
_{\delta }.$ The $F_{j}$ and $G_{i}$ are semialgebraic on $\Omega _{\delta }$

For $\left\vert \alpha \right\vert \leq m$, the derivatives $\partial
^{\alpha }F_{j},\partial ^{\alpha }G_{i}$ extend to continuous functions on $%
\Omega _{2\delta /3}^{\text{closure}}\setminus \left\{ \left( 0,0\right)
\right\} $. For $\left\vert \alpha \right\vert \leq m,$ the derivatives $%
\partial ^{\alpha }F_{j}\left( z\right) $, $\partial ^{\alpha }G_{i}\left(
z\right) $ tend to zero as $z\in \Omega _{\delta }$ tends to zero.

It follows that the $F_{j}$ and $G_{i}$ extend from $\Omega _{\delta /2}$ to semialgebraic
functions in $C^m\left( \Omega _{\delta /2}^{\text{closure}}\right) $ and
those functions all have $m$-jet zero at the origin. We extend $F_j, G_i$ to semialgebraic $C^m_{loc}$ functions on $\R^2$, using Corollary \ref{corollary3.2}.

Next, we show that $j_{\bar{x}}\left( \vec{F},\vec{G}\right) =\left( \vec{F}%
^{\#}\left( \bar{x}\right) ,\vec{G}^{\#}\left( \bar{x}\right) \right) $ for
$\bar{x}\in \left( 0,\delta \right) $.

From \eqref{pml13-0}, we have
\begin{equation*}
j_{\bar{x}}\left( \vec{F},\vec{G}^{\#\#}\right) =\left( \vec{F}^{\#}\left(
\bar{x}\right) ,\vec{G}^{\#}\left( \bar{x}\right) \right) \text{.}
\end{equation*}%

From \eqref{doublestars}, we see that $j_{\bar{x}}\left( G_{i}-G_{i}^{\#\#}\right)
=0$ for all $\bar{x}\in \left( 0,\delta \right) $. Therefore,
\begin{equation*}
j_{\bar{x}}\left( \vec{F},\vec{G}\right) =j_{\bar{x}}\left( \vec{F},\vec{G}%
^{\#\#}\right) =\left( \vec{F}^{\#}\left( \bar{x}\right) ,\vec{G}^{\#}\left(
\bar{x}\right) \right) \text{,}
\end{equation*}%
as desired.

Thus, we have proven \eqref{assertionII}.

The proof of Lemma \ref{main-lemma} is complete. \end{proof}

\subsection{Patching near a cusp}

\label{section-patching-near-a-cusp}

\begin{lemma}
\label{lemma-pnc1} Let $\psi (x)$ be a semialgebraic function on $[0,\delta
] $, satisfying $\psi (0)=0,0<\psi (x)\leq x$ for all $x\in (0,\delta ]$. We
set
\begin{equation*}
E_{\delta }=\{(x,y)\in \mathbb{R}^{2}:0\leq x\leq \delta ,0\leq y\leq \psi
(x)\},
\end{equation*}%
\begin{equation*}
E_{\delta }^{+}=\{(x,y)\in \mathbb{R}^{2}:0\leq x\leq \delta ,\frac{1}{3}%
\psi (x)\leq y\leq \psi (x)\}, \text{ and }
\end{equation*}%
\begin{equation*}
E_{\delta }^{-}=\{(x,y)\in \mathbb{R}^{2}:0\leq x\leq \delta ,0\leq y\leq
\frac{2}{3}\psi (x)\}.
\end{equation*}

Fix a semialgebraic function of one variable, $\theta \left( t\right) $,
satisfying $0\leq \theta \left( t\right) \leq 1$, $\theta \left( t\right) =1$
for $t\leq 1/3,$ $\theta \left( t\right) =0$ for $t\geq 2/3$, $\theta \in
C^{m+100}$.

Then set
\begin{equation*}
\theta _{-}\left( x,y\right) =\theta \left( \frac{y}{\psi \left( x\right) }%
\right) \text{, }\theta _{+}\left( x,y\right) =1-\theta _{-}\left(
x,y\right) \text{ for }\left( x,y\right) \in E_{\delta }\setminus\{(0,0)\}\text{.}
\end{equation*}%
Thus, $\theta _{+},\theta _{-}\geq 0$ and $\theta _{+}+\theta _{-}=1$ on $%
E_{\delta }\setminus\{(0,0)\}$.

Let $F^{+}\in C^m(E_{\delta }^{+})$ and $%
F^{-}\in C^m(E_{\delta }^{-})$ be semialgebraic functions, with $%
J_{(0,0)}F^{+}=J_{(0,0)}F^{-}=0$.

Suppose that
\begin{equation}
\partial _{y}^{l}F^{+}(x,\psi (x))-\sum_{j=0}^{m-l}\frac{1}{j!}\partial
_{y}^{l+j}F^{-}(x,0)\cdot (\psi (x))^{j}=o((\psi (x))^{m-l})
\label{taylorexp}
\end{equation}%
as $x\rightarrow 0^{+}$ for each $l=0,\cdots ,m$.

Define $F=\theta_+\cdot F^+ + \theta_-\cdot F^-$ on $E_\delta\setminus\{(0,0)\}, F(0,0)=0$.

Then $F$ is a $C^m$ semialgebraic function on $E_{\delta^{\prime}}$ for some small
$\delta^{\prime}$. The jet of $F$ at the origin is zero. Moreover, $F=F^+$ in a neighborhood of any point $(x, \psi(x))$, $0<x<\delta^{\prime}$; and $F=F^-$ in a
neighborhood of any point $(x,0), 0<x< \delta^{\prime}$.

\end{lemma}

\begin{proof}
Because $0\leq \psi (x)\leq x$ and $\psi $ is given near $0$ by a convergent
Puiseux series, we have $\psi ^{(k)}(x)=O(x^{1-k})$ as $x\rightarrow 0^{+}$,
for $k=0,\cdots ,m+100$. Also, because $F^{+},F^{-}$ have zero jet at $(0,0)$%
, we have, for $|\alpha |=m$, $\partial ^{\alpha }F^{+}(x,y)=o(1)$ as $%
(x,y)\in E_{\delta }^{+}$ tends to zero and $\partial ^{\alpha
}F^{-}(x,y)=o(1)$ as $(x,y)\in E_{\delta }^{-}$ tends to zero.

By induction on $\mu$, we now prove that

\begin{itemize}
\item[\refstepcounter{equation}\text{(\theequation)}\label{pnc-1}] $\partial
_{x}^{\mu }\partial _{y}^{l}F^{+}(x,\psi (x))-\sum_{j=0}^{m-l-\mu }\frac{1}{%
j!}\partial _{x}^{\mu }\partial _{y}^{l+j}F^{-}(x,0)\cdot (\psi
(x))^{j}=o((\psi (x))^{m-\mu -l})$ as $x\rightarrow 0^{+}$ for $\mu +l\leq m$%
.
\end{itemize}

For $\mu =0$, \eqref{pnc-1} is a hypothesis of our lemma. Assuming %
\eqref{pnc-1} for $\mu $, we prove it for $\mu +1$. Thus, fix $l$ satisfying
$(\mu +1)+l\leq m$. Recalling that $\partial
_{x}^{\mu }\partial _{y}^{l+j}F^{-}(x,0)=o(1)$ when $\mu +(l+j)=m$, we
conclude from \eqref{pnc-1} that

\begin{itemize}
\item[\refstepcounter{equation}\text{(\theequation)}\label{pnc-1-lhs}] $%
\partial _{x}^{\mu }\partial _{y}^{l}F^{+}(x,\psi (x))-\sum_{j=0}^{m-l-\mu
-1}\frac{1}{j!}\partial _{x}^{\mu }\partial _{y}^{l+j}F^{-}(x,0)\cdot (\psi
(x))^{j}=o((\psi (x))^{m-\mu -l})$ as $x\rightarrow 0^{+}$.
\end{itemize}

Because the above functions are semialgebraic functions of one variable and
thus given near $0$ by convergent Puiseux series, it follows that $\frac{d}{%
dx}\{\eqref{pnc-1-lhs}\}=o((\psi (x))^{m-\mu -l}\cdot x^{-1})$, hence $\frac{%
d}{dx}\{\eqref{pnc-1-lhs}\}=o((\psi (x))^{m-\mu -l-1})$, because $0<\psi
(x)\leq x$. Thus,
\begin{eqnarray*}
&&\left[ \left( \partial _{x}+\psi ^{\prime }\left( x\right) \partial
_{y}\right) \left( \partial _{x}^{\mu }\partial _{y}^{l}F^{+}\right) \right]
\left( x,\psi \left( x\right) \right) -\sum_{j=0}^{m-l-\mu -1}\frac{1}{j!}%
\partial _{x}^{\mu +1}\partial _{y}^{l+j}F^{-}\left( x,0\right) \left( \psi
\left( x\right) \right) ^{j} \\
&&-\sum_{j=1}^{m-l-\mu -1}\frac{1}{j!}\partial _{x}^{\mu }\partial
_{y}^{l+j}F^{-}\left( x,0\right) j\left( \psi \left( x\right) \right)
^{j-1}\psi ^{\prime }\left( x\right) \\
&=&o\left( \left( \psi \left( x\right) \right) ^{m-\mu -l-1}\right) \text{.}
\end{eqnarray*}%
It follows that
\begin{eqnarray}
&&\left[ \partial _{x}^{\mu +1}\partial _{y}^{l}F^{+}\left( x,\psi \left(
x\right) \right) -\sum_{j=0}^{m-l-\left( \mu +1\right) }\frac{1}{j!}\partial
_{x}^{\mu +1}\partial _{y}^{l+j}F^{-}\left( x,0\right) \left( \psi \left(
x\right) \right) ^{j}\right]  \notag \\
&&+\psi ^{\prime }\left( x\right) \left[ \partial _{x}^{\mu }\partial
_{y}^{l+1}F^{+}\left( x,\psi \left( x\right) \right) -\sum_{j=0}^{m-l-\mu -2}%
\frac{1}{j!}\partial _{x}^{\mu }\partial _{y}^{l+1+j}F^{-}\left( x,0\right)
\left( \psi \left( x\right) \right) ^{j}\right]  \label{pncbracket} \\
&=&o\left( \left( \psi \left( x\right) \right) ^{m-\left( \mu +1\right)
-l}\right) \text{.}  \notag
\end{eqnarray}%
For $j=m-l-\mu -1$, we have $\partial _{x}^{\mu }\partial
_{y}^{l+1+j}F^{-}\left( x,0\right) =o\left( 1\right) $, hence inductive
hypothesis \eqref{pnc-1} for $\left( l+1\right) $ in place of $l$ tells us
that the second term in square brackets in \eqref{pncbracket} is $o\left(
\left( \psi \left( x\right) \right) ^{m-\left( \mu +1\right) -l}\right) $.
Also, $\left\vert \psi ^{\prime }\left( x\right) \right\vert =O\left(
1\right) $.

Consequently, the first term in square brackets in \eqref{pncbracket} is $%
o\left( \left( \psi \left( x\right) \right) ^{m-\left( \mu +1\right)
-l}\right) $, proving the analogue of \eqref{pnc-1} for $\mu +1,$ thus
completing the induction and establishing \eqref{pnc-1}$.$

We bring in the cutoff functions $\theta_+$ and $\theta_-$. Note that $\theta _{+}$ is supported in $E_{\delta }^{+}$ and $\theta _{-}$ is
supported in $E_{\delta }^{-}$.

We will estimate the derivatives of $\theta _{+}$, $\theta _{-}$ on $%
E_{\delta }$.

We have
\begin{equation*}
\left( \frac{d}{dx}\right) ^{k}\frac{1}{\psi \left( x\right) }=O\left( \frac{%
1}{\psi \left( x\right) }x^{-k}\right) \text{ as }x\rightarrow 0^{+},
\end{equation*}%
because $\psi $ is given by a convergent Puiseux series.

Because $0<\psi \left( x\right) \leq x$ for $x\in \left( 0,\delta \right) $
and $0\leq y\leq \psi \left( x\right) $ in $E_{\delta }$, it follows that $$%
\partial _{x}^{l}\partial _{y}^{k}\left( \frac{y}{\psi \left( x\right) }%
\right) =O\left( \left( \psi \left( x\right) \right) ^{-k-l}\right) $$ as $%
\left( x,y\right) \in E_{\delta }\rightarrow 0$, for all $k,l\geq 0$.

Now, $\partial _{x,y}^{\alpha }\theta _{-}\left( x,y\right) $ is a sum of
terms $\theta ^{\left( s\right) }\left( \frac{y}{\psi \left( x\right) }%
\right) \cdot \prod_{\sigma =1}^{s}\left[ \partial _{x,y}^{\alpha _{\sigma
}}\left( \frac{y}{\psi \left( x\right) }\right) \right] $ with $\alpha
_{1}+\cdots +\alpha _{s}=\alpha $, $s\leq \left\vert \alpha \right\vert $.

Each such term is $O\left( \prod_{\sigma =1}^{s}\left( \frac{1}{\psi \left(
x\right) }\right) ^{\left\vert \alpha _{\sigma }\right\vert }\right)
=O\left( \left( \frac{1}{\psi \left( x\right) }\right) ^{\left\vert \alpha
\right\vert }\right) $.

Thus,
\begin{equation}
\left\vert \partial _{x,y}^{\alpha }\theta _{-}\left( x,y\right) \right\vert
,\left\vert \partial _{x,y}^{\alpha }\theta _{+}\left( x,y\right)
\right\vert \leq \frac{C_{\alpha }}{\left( \psi \left( x\right) \right)
^{\left\vert \alpha \right\vert }}\text{ on }E_{\delta }\text{ (smaller }%
\delta \text{) for }\left\vert \alpha \right\vert \leq m+100\text{.}
\label{pnc2}
\end{equation}%
Next, we return to $F^{+},F^{-}$, and prove the following estimate
\begin{equation}
\partial _{x}^{\mu }\partial _{y}^{l}\left( F^{+}-F^{-}\right) \left(
x,y\right) =o\left( \left[ \psi \left( x\right) \right] ^{m-\mu -l}\right)
\text{ as }\left( x,y\right) \in E_{\delta }^{+}\cap E_{\delta
}^{-}\rightarrow 0  \label{pnc3}
\end{equation}%
for each $\mu ,l$ with $\mu +l\leq m$.

To see this, fix $\mu $, $0\leq \mu \leq m$, and look at the polynomials
\begin{eqnarray*}
P_{x}^{+}\left( y\right) &=&\sum_{j=0}^{m-\mu }\frac{1}{j!}\left[ \partial
_{y}^{j}\partial _{x}^{\mu }F^{+}\left( x,\psi \left( x\right) \right) %
\right] \cdot \left( y-\psi \left( x\right) \right) ^{j}\text{,} \\
P_{x}^{-}\left( y\right) &=&\sum_{j=0}^{m-\mu }\frac{1}{j!}\left[ \partial
_{y}^{j}\partial _{x}^{\mu }F^{-}\left( x,0\right) %
\right] \cdot y^{j}\text{.}
\end{eqnarray*}%
Estimate \eqref{pnc-1} shows that
\begin{equation}
\partial _{y}^{l}\left( P_{x}^{+}-P_{x}^{-}\right) |_{y=\psi \left( x\right)
}=o\left( \left( \psi \left( x\right) \right) ^{m-\mu -l}\right) \text{ for }%
l=0,\cdots ,m-\mu \text{.}  \label{pncstar}
\end{equation}%
For $y$ satisfying $\left( x,y\right) \in E_{\delta }^{+}\cap E_{\delta
}^{-} $, we have $\left\vert y\right\vert ,\left\vert y-\psi \left( x\right)
\right\vert \leq \psi \left( x\right) $ and therefore \eqref{pncstar} yields
\begin{equation*}
\partial _{y}^{l}\left( P_{x}^{+}-P_{x}^{-}\right) \left( x,y\right)
=o\left( \left( \psi \left( x\right) \right) ^{m-\mu -l}\right)
\end{equation*}%
as $\left( x,y\right) \in E_{\delta }^{+}\cap E_{\delta }^{-}$ tends to zero.

On the other hand, Taylor's theorem gives for $\left( x,y\right) \in
E_{\delta }^{+}\cap E_{\delta }^{-}\setminus\{(0,0)\}$ the estimates
\begin{equation*}
\partial _{y}^{l}\left[ \partial _{x}^{\mu }F^{+}-P_{x}^{+}\right] \left(
x,y\right) =O\left( \left( \psi \left(
x\right) \right) ^{m-\mu -l}\cdot\max_{\bar{y}\in \left[ \frac{1}{3}\psi \left( x\right)
,\psi \left( x\right) \right] }\left\vert \partial _{y}^{m-\mu }\partial
_{x}^{\mu }F^{+}\left( x,\bar{y}\right) \right\vert \right)
\end{equation*}%
and
\begin{equation*}
\partial _{y}^{l}\left[ \partial _{x}^{\mu }F^{-}-P_{x}^{-}\right] \left(
x,y\right) =O\left(\left( \psi \left( x\right)
\right) ^{m-\mu -l}\cdot \max_{\bar{y}\in \left[ 0,\frac{2}{3}\psi \left(
x\right) \right] }\left\vert \partial _{y}^{m-\mu }\partial _{x}^{\mu
}F^{-}\left( x,\bar{y}\right) \right\vert  \right) \text{.}
\end{equation*}%
The maxima in these last two estimates are $o\left( 1\right) $, because $%
J_{\left( 0,0\right) }F^{+}=J_{\left( 0,0\right) }F^{-}=0$.

Thus, as $\left( x,y\right) \in E_{\delta }^{+}\cap E_{\delta }^{-}\setminus\{(0,0)\}$
approaches zero, the quantities $\partial _{y}^{l}\left[ \partial _{x}^{\mu
}F^{+}-P_{x}^{+}\right] \left( x,y\right) $, \newline
$\partial _{y}^{l}\left[ \partial _{x}^{\mu }F^{-}-P_{x}^{-}\right] \left(
x,y\right)$, $\partial _{y}^{l}\left[ P_{x}^{+}-P_{x}^{-}\right] \left(
x,y\right) $ are all $o\left( \left( \psi \left( x\right) \right) ^{m-\mu
-l}\right) $.

Consequently, $\left( \partial _{y}^{l}\partial _{x}^{\mu }F^{+}-\partial
_{y}^{l}\partial _{x}^{\mu }F^{-}\right) \left( x,y\right) =o\left( \left(
\psi \left( x\right) \right) ^{m-\mu -l}\right) $ as $\left( x,y\right) \in
E_{\delta }^{+}\cap E_{\delta }^{-}\setminus\{(0,0)\}$ approaches zero, completing the proof
of \eqref{pnc3}.

We now set $F=\theta _{+}F^{+}+\theta _{-}F^{-}$ on $E_{\delta }\setminus\{(0,0)\}$ and $F(0,0)=0$.

Evidently, $F$ is $C^m$ away from the origin, and semialgebraic; moreover, $%
F=F^{+}$ in a neighborhood of any point $\left( x^{0},\psi \left(
x^{0}\right) \right) $ in $E_{\delta }$ $\left( x^{0}\not=0\right) $ and $%
F=F^{-}$ in a neighborhood of any point $\left( x^{0},0\right) \in E_{\delta
}$ $\left( x^{0}\not=0\right) $.

It remains to check that $F\in C^m\left( E_{\delta }\right) $ near $0$ and
that $J_{\left( 0,0\right) }F=0$. That amounts to showing that
\begin{equation}
\partial _{x,y}^{\alpha }F\left( x,y\right) =o\left( x^{m-\left\vert \alpha
\right\vert }\right) \text{ as }\left( x,y\right) \in E_{\delta }\setminus\{(0,0)\}\text{
approaches } (0,0) \text{ (all }\left\vert \alpha \right\vert \leq m\text{).}
\label{pnc4}
\end{equation}

To prove \eqref{pnc4}, we may assume $\left( x,y\right) \in E_{\delta
}^{+}\cap E_{\delta }^{-}\setminus\{(0,0)\}$, because otherwise the left-hand side of \eqref{pnc4} is $%
\partial _{x,y}^{\alpha }F^{+}$ for $\left( x,y\right) \in E_{\delta }^{+}\setminus\{(0,0)\}$
or else $\partial _{x,y}^{\alpha }F^{-}$ for $\left( x,y\right) \in
E_{\delta }^{-}\setminus\{(0,0)\}$, in which case \eqref{pnc4} holds because $J_{\left(
0,0\right) }F^{+}=J_{\left( 0,0\right) }F^{-}=0$.

For $\left( x,y\right) \in E_{\delta }^{+}\cap E_{\delta }^{-}\setminus\{(0,0)\}$, we have
\begin{equation}
F=F^{-}+\theta _{+}\left( F^{+}-F^{-}\right) \text{.}  \label{pnc5}
\end{equation}%
Because $J_{\left( 0,0\right) }F^{-}=0$, we have
\begin{equation}
\partial _{x,y}^{\alpha }F^{-}\left( x,y\right) =o\left( x^{m-\left\vert
\alpha \right\vert }\right) \text{ as }\left( x,y\right) \in E_{\delta
}^{+}\cap E_{\delta }^{-}\setminus\{(0,0)\}\text{ tends to }(0,0)\text{, for }\left\vert \alpha
\right\vert \leq m\text{.}  \label{pnc6}
\end{equation}%
We recall that $\partial _{x,y}^{\alpha }\theta _{+}\left( x,y\right)
=O\left( \left( \psi \left( x\right) \right) ^{-\left\vert \alpha
\right\vert }\right) $ for $\left\vert \alpha \right\vert \leq m$ and that $%
\partial _{x,y}^{\alpha }\left( F^{+}-F^{-}\right) \left( x,y\right)
=o\left( \left( \psi \left( x\right) \right) ^{m-\left\vert \alpha
\right\vert }\right) $ for $\left\vert \alpha \right\vert \leq m$ as $\left(
x,y\right) \in E_{\delta }^{+}\cap E_{\delta }^{-}\setminus\{(0,0)\}$ tends to $(0,0)$, for $%
\left\vert \alpha \right\vert \leq m$.

Therefore, for $\left\vert \alpha \right\vert \leq m$, as $\left( x,y\right)
\in E_{\delta }^{+}\cap E_{\delta }^{-}\setminus\{(0,0)\}$ tends to $(0,0)$, we have
\begin{equation*}
\partial _{x,y}^{\alpha }\left\{ \theta _{+}\left( F^{+}-F^{-}\right) \left(
x,y\right) \right\} =o\left( \left( \psi \left( x\right) \right)
^{m-\left\vert \alpha \right\vert }\right) \text{,}
\end{equation*}%
hence
\begin{equation}
\partial _{x,y}^{\alpha }\left\{ \theta _{+}\left( F^{+}-F^{-}\right) \left(
x,y\right) \right\} =o\left( x^{m-\left\vert \alpha \right\vert }\right)
\text{,}  \label{pnc7}
\end{equation}%
because $0<\psi \left( x\right) \leq x$. Putting \eqref{pnc6}, \eqref{pnc7}
into \eqref{pnc5}, we see that
\begin{equation*}
\partial _{x,y}^{\alpha }F\left( x,y\right) =o\left( x^{m-\left\vert \alpha
\right\vert }\right)
\end{equation*}%
as $\left( x,y\right) \in E_{\delta }^{+}\cap E_{\delta }^{-}\setminus\{(0,0)\}$ tends to $(0,0)$,
for $\left\vert \alpha \right\vert \leq m$.

Thus, \eqref{pnc4} holds. The proof of Lemma \ref{lemma-pnc1} is complete.
\end{proof}

Next, we introduce a change of variables in a neighborhood of $0$ in $%
\mathbb{R}_{+}^{2}=\left\{ \left( x,y\right) :x> 0\right\} $ of the form
\begin{equation}
\bar{x}=x,\bar{y}=y+\tilde{\psi}\left( x\right) \text{,}  \label{pnc+}
\end{equation}%
where $\tilde{\psi}\left( x\right) $ is semialgebraic and satisfies $%
\left\vert \tilde{\psi}\left( x\right) \right\vert \leq Cx$ for $x\in \left(
0,\delta \right) $.

The inverse change of variables is of course
\begin{equation*}
x=\bar{x},y=\bar{y}-\tilde{\psi}\left( \bar{x}\right) \text{.}
\end{equation*}%
Note that $\partial _{x,y}^{\alpha }\left( \bar{x},\bar{y}\right) =O\left(
x^{1-\left\vert \alpha \right\vert }\right) $ for $\left\vert y\right\vert
\leq Cx\ll 1$ because $\tilde{\psi}$ is given near $0$ as \ a convergent Puiseux
series, hence $\left\vert \tilde{\psi}\left( x\right) \right\vert \leq Cx$
implies $\left\vert \tilde{\psi}^{\left( k\right) }\right\vert \leq
C_{k}x^{1-k}$ for small $x$.

The change of variables \eqref{pnc+} does not preserve $C^m$, but it does
preserve $C^m$ functions whose jets at $0$ are equal to zero.

Indeed, suppose $F\left( \bar{x},\bar{y}\right) \in C^m \left( \bar{E}\right) $ for
$\bar{E}\subset \left\{ \left( \bar{x},\bar{y}\right) :\left\vert \bar{y}%
\right\vert \leq C\bar{x}\right\} $, with $0\in \bar{E}$ and $J_{0}F=0$.

Then $\bar{E}$ corresponds under \eqref{pnc+} to a set $E\subset \left\{
\left( x,y\right) :\left\vert y\right\vert \leq C^{\prime }x\right\} $, $%
0\in E$.

We may regard $F$ as a function of $\left( x,y\right) $, and for $\left\vert
\alpha \right\vert \leq m$, $\partial _{x,y}^{\alpha }F\left( x,y\right) $
is a sum of terms $\left\vert \partial _{\bar{x},\bar{y}}^{\beta }F\left(
\bar{x},\bar{y}\right) \right\vert \cdot \prod_{\nu =1}^{\left\vert \beta
\right\vert }\left[ \partial _{x,y}^{\alpha _{\nu }}\left( \bar{x},\bar{y}%
\right) \right] $ with $|\beta|\leq m$ and $\sum_{\nu }\alpha _{\nu }=\alpha $. If $J_{\left(
0,0\right) }F=0$ as a function of $\left( \bar{x},\bar{y}\right) $, then $%
\partial _{\bar{x},\bar{y}}^{\beta }F\left( \bar{x},\bar{y}\right) =o\left(
\bar{x}^{m-\left\vert \beta \right\vert }\right) $ on $\bar{E}$, hence $%
\partial _{\bar{x},\bar{y}}^{\beta }F\left( \bar{x},\bar{y}\right) =o\left(
x^{m-\left\vert \beta \right\vert }\right) $ on $E$. Also, on $E,$
\begin{equation*}
\prod_{\nu =1}^{\left\vert \beta \right\vert }\left[ \partial _{x,y}^{\alpha
_{\nu }}\left( \bar{x},\bar{y}\right) \right] =\prod_{\nu =1}^{\left\vert
\beta \right\vert }O\left( x^{1-\left\vert \alpha _{\nu }\right\vert
}\right) =O\left( x^{\left\vert \beta \right\vert -\sum_{\nu }\left\vert
\alpha _{\nu }\right\vert }\right) =O\left( x^{\left\vert\beta\right\vert -\left\vert \alpha
\right\vert }\right) \text{.}
\end{equation*}%
Consequently, $\partial _{x,y}^{\alpha }F\left( x,y\right) =o\left(
x^{m-\left\vert \alpha \right\vert }\right) $ on $E\setminus\{(0,0)\}$, for $\left\vert \alpha
\right\vert \leq m$. Thus, as claimed, $F\in C^m\left( E\right) $ and $%
J_{\left( 0,0\right) }F=0$.

The following generalization of Lemma \ref{lemma-pnc1} is reduced to Lemma %
\ref{lemma-pnc1} by means of the change of variables discussed above.

\begin{lemma}
\label{lemma-pnc-2}Let $0\leq \psi _{-}(x)\leq \psi _{+}\left( x\right) \leq
x$ be semialgebraic functions on $[0,\delta ]$, with $\psi_- < \psi_+$ on $(0, \delta]$. We set
\begin{equation*}
E_{\delta }=\{(x,y)\in \mathbb{R}^{2}:0\leq x\leq \delta ,\psi _{-}\left(
x\right) \leq y\leq \psi _{+}(x)\},
\end{equation*}%
\begin{equation*}
E_{\delta }^{+}=\{(x,y)\in \mathbb{R}^{2}:0\leq x\leq \delta ,0\leq \psi
_{+}(x)-y\leq \frac{2}{3}\left( \psi _{+}(x)-\psi _{-}\left( x\right)
\right) \},\text{ and}
\end{equation*}%
\begin{equation*}
E_{\delta }^{-}=\{(x,y)\in \mathbb{R}^{2}:0\leq x\leq \delta ,0\leq y-\psi
_{-}\left( x\right) \leq \frac{2}{3}\left( \psi _{+}(x)-\psi _{-}\left(
x\right) \right) \}.
\end{equation*}

Fix a semialgebraic function of one variable, $\theta \left( t\right) $,
satisfying $0\leq \theta \left( t\right) \leq 1$, $\theta \left( t\right) =1$
for $t\leq 1/3,$ $\theta \left( t\right) =0$ for $t\geq 2/3$, $\theta \in
C^{m+100}$.

Then set
\begin{equation*}
\theta _{-}\left( x,y\right) =\theta \left( \frac{y-\psi_-(x)}{(\psi_+-\psi_-) \left( x\right) }%
\right) \text{, }\theta _{+}\left( x,y\right) =1-\theta _{-}\left(
x,y\right) \text{ for }\left( x,y\right) \in E_{\delta }\setminus\{(0,0)\}\text{.}
\end{equation*}%
Thus, $\theta _{+},\theta _{-}\geq 0$ and $\theta _{+}+\theta _{-}=1$ on $%
E_{\delta }\setminus\{(0,0)\}$.

Let $F^{+}\in C^m(E_{\delta }^{+})$ and $%
F^{-}\in C^m(E_{\delta }^{-})$ be semialgebraic functions, with $%
J_{(0,0)}F^{+}=J_{(0,0)}F^{-}=0$.

Suppose that
\begin{equation*}
\partial _{y}^{l}F^{+}(x,\psi_+ (x))-\sum_{j=0}^{m-l}\frac{1}{j!}\partial
_{y}^{l+j}F^{-}(x,\psi_-(x))\cdot (\psi _{+}(x)-\psi _{-}\left( x\right)
)^{j}=o((\psi _{+}(x)-\psi _{-}\left( x\right) )^{m-l})
\end{equation*}%
as $x\rightarrow 0^{+}$ for each $l=0,\cdots ,m$.

Define $F=\theta_+\cdot F^+ + \theta_-\cdot F^-$ on $E_\delta\setminus\{(0,0)\}, F(0,0)=0$.

Then $F$ is a $C^m$ semialgebraic function on $E_{\delta^{\prime} }$ for some small $\delta^{\prime} $. The jet of $F$ at $(0,0)$ is zero. Moreover, $F=F^{+}$ in a neighborhood of any point $(x,\psi _{+}(x))$, $0<x<\delta^{\prime}$, and $F=F^{-}$ in a neighborhood of any point $(x,\psi
_{-}(x))$, $0<x<\delta^{\prime} $.
\end{lemma}

\subsection{Proof of Lemma \ref{lemma-pit}\label{section-putting-it-together}}

Let $\mathcal{H}=(H(z))_{z\in \mathbb{R}^{2}}$ be a semialgebraic bundle
with a $C^m_{loc}$ section. Each $H(z)$ is a coset of an $\mathcal{R}_{z}$
submodule in $\mathcal{R}_{z}^{D}$. Assume $H((0,0))=\{0\}$. Let $\Omega
_{\delta }=\{(x,y)\in \mathbb{R}^{2}:0\leq x\leq \delta ,0\leq y\leq x\}$
for $\delta >0$. We look for semialgebraic $C^m_{loc}$ sections of $\mathcal{H}|_{\Omega _{\delta }}$, for some small $\delta $ (which will keep
shrinking as we discuss further).

We apply Lemma \ref{WIRM-lemma}. Thus, we obtain the following

\begin{itemize}
\item Semialgebraic functions $0\leq \psi _{0}\left( x\right) \leq \psi
_{1}\left( x\right) \leq \cdots \leq \psi _{s_{\max }}\left( x\right) =x$ on
$\left( 0,\delta \right) ,$ all given by convergent Puiseux expansions on $%
\left( 0,\delta \right) $.

\item Integers $k_{s}$ $\left( 0\leq k_{s}\leq D\right) $ and permutations $%
\pi _{s}:\left\{ 1,\cdots ,D\right\} \rightarrow \left\{ 1,\cdots ,D\right\}
$ for $s=1,\cdots ,s_{\max}$.

\item Semialgebraic functions $A_{ij}^{s}\left( x,y\right) $ $( s=1,\cdots ,s_{\max}$%
, $1\leq i\leq k_{s},k_{s}<j\leq D) $ and $\varphi _{i}^{s}\left( x,y\right) $
$\left( s=1,\cdots ,s_{\max},1\leq i\leq k_{s}\right) $ defined on $E_{s}=\left\{
\left( x,y\right) :0<x<\delta ,\psi _{s-1}\left( x\right) <y<\psi _{s}\left(
x\right) \right\} $.

\item Semialgebraic functions $\theta _{jl}^{si}\left( x\right) $, $%
g^{si}\left( x\right) $ $(s=0,\cdots ,s_{\max },i=1,\cdots ,i_{\max }\left(
s\right) $, $j=1,\cdots ,D,$ $l=0,\cdots ,m)$ defined on $\left( 0,\delta
\right) $, and given there by convergent Puiseux expansions.
\end{itemize}

The above objects have the following properties

\begin{itemize}
\item (Estimates) For $\left( x,y\right) \in \Omega_1 $ with $0<x<\delta $ and
$\psi _{s-1}\left( x\right) <y<\psi _{s}\left( x\right) $, we have $%
\left\vert \partial ^{\alpha }A_{ij}^{s}\left( x,y\right) \right\vert $, $%
\left\vert \partial ^{\alpha }\varphi _{i}^{s}\left( x,y\right) \right\vert
\leq C\left[ \min \left( \left\vert y-\psi _{s}\left( x\right) \right\vert
,\left\vert y-\psi _{s-1}\left( x\right) \right\vert \right) \right]
^{-\left\vert \alpha \right\vert }$ for $\left\vert \alpha \right\vert \leq
m+100$.

\item (Condition for sections) Let $F=(F_1,\cdots,F_D)\in C^m\left( \Omega_1, \mathbb{R}^D \right) $, and
suppose $J_{x}F\in H\left( x\right) $ for all $x\in \Omega_1 $.

Then for $s=1,\cdots ,s_{\max }$, $i=1,\cdots ,k_{s}$, $x\in \left( 0,\delta
\right) $, $\psi _{s-1}\left( x\right) <y<\psi _{s}\left( x\right) $, we
have
\begin{equation}
F_{\pi _{s}i}\left( x,y\right) +\sum_{D\geq j>k_{s}}A_{ij}^{s}\left(
x,y\right) F_{\pi _{s}j}\left( x,y\right) =\varphi _{i}^{s}\left( x,y\right)
\text{;}  \label{pit-1}
\end{equation}%
and for $s=0,1,\cdots ,s_{\max }$, $i=1,\cdots ,i_{\max }\left( s\right) $, $%
x\in \left( 0,\delta \right) $, we have
\begin{equation}
\sum_{j=1}^{D}\sum_{l=0}^{m}\theta _{jl}^{si}\left( x\right) \partial
_{y}^{l}F_{j}\left( x,\psi _{s}\left( x\right) \right) =g^{si}\left(
x\right) \text{;}  \label{pit-2}
\end{equation}%
and
\begin{equation}
J_{\left( 0,0\right) }F_j=0  \label{pit-3}
\end{equation}
for all $j$. 

Conversely, if $F=(F_{j})_{j=1,\cdots ,D}\in C^m_{loc}\left( \mathbb{R}^2,\R^D \right) $ satisfies \eqref{pit-1}, \eqref{pit-2}, %
\eqref{pit-3}, then $F$ is a section of $\mathcal{H}$ over $\Omega _{\delta
}^{\text{closure}}$.
\end{itemize}

Next, we set (for $s=1,\cdots ,s_{\max })$:%
\begin{equation*}
E_{s}^{+}=\left\{ \left( x,y\right) \in \mathbb{R}^{2}:0\leq x\leq \delta ,%
\text{ }0\leq \psi _{s}\left( x\right) -y\leq \frac{2}{3}\left( \psi
_{s}-\psi _{s-1}\left( x\right) \right) \right\}
\end{equation*}%
and
\begin{equation*}
E_{s}^{-}=\left\{ \left( x,y\right) \in \mathbb{R}^{2}:0\leq x\leq \delta
\text{, }0\leq y-\psi _{s-1}\left( x\right) \leq \frac{2}{3}\left( \psi
_{s}\left( x\right) -\psi _{s-1}\left( x\right) \right) \right\} \text{.}
\end{equation*}%
Then $E_{s}^{+,\text{interior}}\cup E_{s}^{-,\text{interior}}=E_{s}$. On $E_{s}^{+,\text{interior}}$ we have
$\left\vert \partial ^{\alpha }A_{ij}^{s}\left( x\right) \right\vert $, $%
\left\vert \partial ^{\alpha }\varphi _{i}^{s}\left( x,y\right) \right\vert
\leq C\left( \psi _{s}\left( x\right) -y\right) ^{-\left\vert \alpha
\right\vert }$ for $\left\vert \alpha \right\vert \leq m+100$, and on $%
E_{s}^{-\text{, interior}}$ we have $\left\vert \partial ^{\alpha
}A_{ij}^{s}\left( x\right) \right\vert $, $\left\vert \partial ^{\alpha
}\varphi _{i}^{s}\left( x,y\right) \right\vert \leq C\left( y-\psi
_{s-1}\left( x\right) \right) ^{-\left\vert \alpha \right\vert }$ for $%
\left\vert \alpha \right\vert \leq m+100$.

We may apply Lemma \ref%
{main-lemma} after a change of variables of the form $(\bar{x},\bar{y})=(x,\pm(y-\psi(x))).$

Thus, we obtain the following objects, with properties described below.

\begin{itemize}
\item Semialgebraic functions $\theta _{jl}^{+,si}\left( x\right) $, $%
g^{+,si}\left( x\right) $, $i=1,\cdots ,i_{\max }^{+}\left( s\right) $, $%
\theta _{jl}^{-,si}\left( x\right) $, $g^{-,si}\left( x\right) $, $%
i=1,\cdots ,i_{\max }^{-}\left( s\right) $, $l=0,\cdots,m,$ defined on $\left( 0,\delta
\right) $ (smaller $\delta $).

\item Semialgebraic functions $\tilde{\theta}_{jl}^{+,si}\left( x\right) $, $%
\tilde{g}^{+,si}\left( x\right) $, $i=1,\cdots ,\tilde{\imath}_{\max
}^{+}\left( s\right) $, $\tilde{\theta}_{jl}^{-,si}\left( x\right) $, $%
\tilde{g}^{-,si}\left( x\right) $, $i=1,\cdots ,\tilde{\imath}_{\max
}^{-}\left( s\right) $, $l=0,\cdots,m,$ defined on $\left( 0,\delta \right) $ (smaller $%
\delta $).
\end{itemize}

The properties for these functions are as follows.

Let $F=(F_1,\cdots,F_D)\in C^m_{loc}\left( \R^2, \R^D\right) $
satisfy \eqref{pit-1} in $E_{s}^{+,\text{interior}}$ and $J_{\left(
0,0\right) }F=0$. Then
\begin{equation}
\sum_{\substack{ 1\leq j\leq D  \\ 0\leq l\leq m}}\theta
_{jl}^{+,si}\partial _{y}^{l}F_{j}\left( x,\psi _{s}\left( x\right) \right)
=g^{+,si}\left( x\right)  \label{pit-4}
\end{equation}%
for $x\in \left( 0,\delta \right) $ and all $i$, and%
\begin{equation}
\sum_{\substack{ 1\leq j\leq D  \\ 0\leq l\leq m}}\tilde{\theta}%
_{jl}^{+,si}\partial _{y}^{l}F_{j}\left( x,\psi _{s}\left( x\right) \right) =%
\tilde{g}^{+,si}\left( x\right) +o\left( 1\right) \text{ as }x\rightarrow
0^{+}  \label{pit-5}
\end{equation}%
for $x\in \left( 0,\delta \right) $ and all $i$.

Similarly, let $F=(F_1,\cdots,F_D) \in C^m_{loc}\left( \R^2, \R^D\right) $ satisfy \eqref{pit-1} in $E_{s}^{-,\text{interior}}$ and $%
J_{\left( 0,0\right) }F=0$. Then
\begin{equation}
\sum_{\substack{ 1\leq j\leq D  \\ 0\leq l\leq m}}\theta
_{jl}^{-,si}\partial _{y}^{l}F_{j}\left( x,\psi _{s-1}\left( x\right)
\right) =g^{-,si}\left( x\right)  \label{pit-6}
\end{equation}%
for $x\in \left( 0,\delta \right) $ and all $i$, and%
\begin{equation}
\sum_{\substack{ 1\leq j\leq D  \\ 0\leq l\leq m}}\tilde{\theta}%
_{jl}^{-,si}\partial _{y}^{l}F_{j}\left( x,\psi _{s-1}\left( x\right)
\right) =\tilde{g}^{-,si}\left( x\right) +o\left( 1\right) \text{ as }%
x\rightarrow 0^{+}  \label{pit-7}
\end{equation}%
for all $i$.

\begin{itemize}
\item[\refstepcounter{equation}\text{(\theequation)}\label{pit-8}] %
Conversely, fix $s$ and suppose we are given semialgebraic functions $%
f_{jl}^{+,s}\left( x\right) $ on $\left( 0,\delta \right) $ satisfying
\begin{equation*}
\sum_{\substack{ 1\leq j\leq D  \\ 0\leq l\leq m}}\theta
_{jl}^{+,si}f_{jl}^{+,s}\left( x
\right) =g^{+,si}\left( x\right) \text{ (all }i\text{)}
\end{equation*}%
and
\begin{equation*}
\sum_{\substack{ 1\leq j\leq D  \\ 0\leq l\leq m}}\tilde{\theta}%
_{jl}^{+,si}f_{jl}^{+,s}\left( x
\right) =\tilde{g}^{+,si}\left( x\right) +o\left( 1\right) \text{ as }%
x\rightarrow 0^{+}\text{ (all }i\text{)}.
\end{equation*}%
Then there exists a semialgebraic function $F=\left( F_{1},\cdots
,F_{D}\right) \in C^m\left( E_{s}^{+},\R^D\right) $ such that \eqref{pit-1}
holds in $E_{s}^{+,\text{interior}}$ and $\partial _{y}^{l}F_{j}\left(
x,\psi _{s}\left( x\right) \right) =f_{jl}^{+,s}\left( x\right) $ and $%
J_{\left( 0,0\right) }F_{j}=0$ for all $j$.

\item[\refstepcounter{equation}\text{(\theequation)}\label{pit-9}] %
Similarly, fix $s$ and suppose we are given we are given semialgebraic
functions $f_{jl}^{-,s}\left( x\right) $ on $\left( 0,\delta \right) $
satisfying
\begin{equation*}
\sum_{\substack{ 1\leq j\leq D  \\ 0\leq l\leq m}}\theta
_{jl}^{-,si}f_{jl}^{-,s}\left( x
\right) =g^{-,si}\left( x\right) \text{ (all }i\text{)}
\end{equation*}%
and
\begin{equation*}
\sum_{\substack{ 1\leq j\leq D  \\ 0\leq l\leq m}}\tilde{\theta}%
_{jl}^{-,si}f_{jl}^{-,s}\left( x
\right) =\tilde{g}^{-,si}\left( x\right) +o\left( 1\right) \text{ as }%
x\rightarrow 0^{+}\text{ (all }i\text{)}.
\end{equation*}%
Then there exists a semialgebraic function $F=\left( F_{1},\cdots
,F_{D}\right) \in C^m\left( E_{s}^{-},\R^D\right) $ such that \eqref{pit-1}
holds in $E_{s}^{-,\text{interior}}$ and $\partial _{y}^{l}F_{j}\left(
x,\psi _{s}\left( x\right) \right) =f_{jl}^{-,s}\left( x\right) $ and $%
J_{\left( 0,0\right) }F_{j}=0$ for all $j$.

\item[\refstepcounter{equation}\text{(\theequation)}\label{pit-10}] %
Moreover, if $F=(F_1,\cdots,F_D)\in C^m\left( E_{s}^{\text{closure}},\R^D\right) $ with $%
J_{\left( 0,0\right) }F=0$, then $f_{jl}^{+,s}=\partial _{y}^{l}F_{j}\left(
x,\psi _{s}\left( x\right) \right) $ and $f_{jl}^{-,s}=\partial
_{y}^{l}F_{j}\left( x,\psi _{s-1}\left( x\right) \right) $ satisfy the key hypothesis of Lemma \ref%
{lemma-pnc-2}, namely,
\begin{equation*}
f_{jl}^{+,s}\left( x\right) -\sum_{k=0}^{m-l}\frac{1}{k!}f_{j(l+k)}^{-,s}%
\left( x\right) \left( \psi _{s}\left( x\right) -\psi _{s-1}\left( x\right)
\right) ^{k}=o\left( \left[ \psi _{s}\left( x\right) -\psi _{s-1}\left(
x\right) \right] ^{m-l}\right) \text{ as }x\rightarrow 0^{+}
\end{equation*}%
by Taylor's theorem.
\end{itemize}

Now, suppose $F=(F_1,\cdots,F_D)\in C^m_{loc}\left( \mathbb{R}^2, \R^D\right) $ is a section of $%
\mathcal{H}$ over $\Omega _{\delta }$. Then, setting $f_{jl}^{s}\left(
x\right) =\partial _{y}^{l}F_{j}\left( x,\psi _{s}\left( x\right) \right) $
for $x\in \left( 0,\delta \right) $ (smaller $\delta $), we learn that
(because the $F_{j}$ satisfy \eqref{pit-1}, \eqref{pit-2}, %
\eqref{pit-3}), properties \eqref{pit-2}$\cdots$\eqref{pit-7} yield a collection of assertions
of the form
\begin{equation}
\sum_{\substack{ j=1,\cdots ,D  \\ l=0,\cdots ,m}}\theta _{jl}^{\#,si}\left(
x\right) f_{jl}^{s}\left( x\right) =g^{\#,si}\left( x\right) \text{ on }%
\left( 0,\delta \right)  \label{pit-A}
\end{equation}%
and
\begin{equation}
\sum_{\substack{ j=1,\cdots ,D  \\ l=0,\cdots ,m}}\tilde{\theta}%
_{jl}^{\#,si}\left( x\right) f_{jl}^{s}\left( x\right) =\tilde{g}%
^{\#,si}\left( x\right) +o\left( 1\right) \text{ as }x\rightarrow 0^{+}\text{%
;}  \label{pit-B}
\end{equation}%
and also from \eqref{pit-10} we have
\begin{equation}
f_{jl}^{s}\left( x\right) =\sum_{k=0}^{m-l}\frac{1}{k!}f_{j\left( l+k\right)
}^{s-1}\left( x\right) \left[ \psi _{s}\left( x\right) -\psi _{s-1}\left(
x\right) \right] ^{k}+o\left( \left[ \psi _{s}\left( x\right) -\psi
_{s-1}\left( x\right) \right] ^{m-l}\right) \text{ as }x\rightarrow 0^{+}%
\text{.}  \label{pit-C}
\end{equation}

Conversely, if the $f_{jl}^{s}\left( x\right) $ are semialgebraic functions
of one variable, satisfying \eqref{pit-A}, \eqref{pit-B}, and \eqref{pit-C},
then for each $s=1,\cdots ,s_{\max }$ there exist $F_{+}^{s}=(F_{+,1}^{s},\cdots,F_{+,D}^{s})\in C^m\left(
E_{+}^{s\text{, closure}},\R^D\right) $, $F_{-}^{s}=(F_{-,1}^{s},\cdots,F_{-,D}^{s})\in C^m\left( E_{-}^{s\text{%
, closure}},\R^D\right) $ semialgebraic such that \eqref{pit-1}, \eqref{pit-2}, %
\eqref{pit-3} hold in $%
E_{s}^{+}$, $E_{s}^{-}$, respectively and $\partial
_{y}^{l}F_{+,j}^{s}\left( x,\psi _{s}\left( x\right) \right)
=f_{jl}^{s}\left( x\right) $, $\partial _{y}^{l}F_{-,j}^{s}\left( x,\psi
_{s-1}\left( x\right) \right) =f_{jl}^{s-1}\left( x\right) $ and $J_{\left(
0,0\right) }F_{+}^{s}=J_{\left( 0,0\right) }F_{-}^{s}=0$.

Note that $F_{+}^{s}$ is a section of $\mathcal{H}$ over $E_{s}^{+}$, and $%
F_{-}^{s}$ is a section of $\mathcal{H}$ over $E_{s}^{-}$.

Thanks to \eqref{pit-C} and Lemma \ref{lemma-pnc-2}, we may patch together $%
F_+^s$, $F_{-}^s$ into a semialgebraic $F_s =(F_{s,1},\cdots,F_{s,D})\in C^m(E_s^{\text{closure}},\R^D)$
such that $J_{(0,0)} F_s =0$, $F_s$ is a section of $\mathcal{H}$ over $E_s^{\text{closure}}$%
, and $\partial_y^l F_{sj} (x,\psi(x)) = f_{jl}^s (x)$ and $\partial_y^l
F_{sj}(x,\psi_{s-1}(x)) = f_{jl}^{s-1} (x)$.

Because of these conditions, the $F_s$ ($s=1,\cdots, s_{\max}$) fit together (their transverse derivatives up to order $m$ match at the boundaries
where the $E_s$ meet), so using also Corollary \ref{corollary3.2}, we obtain from the $F_s$ a single
semialgebraic $F=(F_1,\cdots,F_D) \in C^m_{loc} (\mathbb{R}^2, \mathbb{R}^D)$ such that $J_{(0,0)}
F =0$, and $F$ is a section of $\mathcal{H}$ over $\Omega_\delta$.

Thus, we have proven Lemma \ref{lemma-pit}.

\section{Proof of Lemma \ref{main-proposition} (Main Lemma)}\label{proofmainlemma}

From the Second Main Lemma (Lemma \ref{lemma-pit}), we can easily deduce Lemma \ref{main-proposition}.

Indeed, suppose $\mathcal{H=}\left( H\left( x,y\right) \right) _{\left(
x,y\right) \in \Omega _{\delta }}$ is as in the hypotheses of Lemma \ref{main-proposition}.

Let $\theta _{jl}^{si},g^{si},\tilde{\theta}_{jl}^{si},\tilde{g}^{si},\psi
_{s}$ be as in Lemma \ref{lemma-pit}.

For $x\in \left( 0,\delta \right) $ with $\delta $ small enough, we
introduce the following objects:
\begin{eqnarray*}
W\left( x\right)  &=&\left\{ \left( \xi _{jl}^{s}\right) _{\substack{ 0\leq
s\leq s_{\max } \\ 0\leq l\leq m \\ 1\leq j\leq D}}\in \mathbb{R}^{\left(
s_{\max }+1\right) \cdot \left( m+1\right) \cdot D}:\sum_{j,l}\theta
_{jl}^{si}\left( x\right) \xi _{jl}^{s}=g^{si}\left( x\right) \text{, each }%
s,i\right\} \text{,} \\
\mathcal{F}\left( \left( \xi _{jl}^{s}\right) ,x\right)
&=&\sum_{s,i}\left\vert \sum_{j,l}\tilde{\theta}_{jl}^{si}\left( x\right)
\xi _{jl}^{s}-\tilde{g}^{si}\left( x\right) \right\vert  \\
&&+\sum_{s\not=0}\sum_{j,l}\frac{\left\vert \xi _{jl}^{s}-\sum_{k=0}^{m-l}%
\frac{1}{k!}\xi _{j\left( l+k\right) }^{s-1}\cdot \left( \psi _{s}\left(
x\right) -\psi _{s-1}\left( x\right) \right) ^{k}\right\vert }{\left[ \psi
_{s}\left( x\right) -\psi _{s-1}\left( x\right) \right] ^{m-l}}\text{,} \\
\mathcal{F}_{\min }\left( x\right)  &=&\inf \left\{ \mathcal{F}\left( \left(
\xi _{jl}^{s}\right) ,x\right) :\left( \xi _{jl}^{s}\right) \in W\left(
x\right) \right\} \text{, and} \\
\Xi _{OK}\left( x\right)  &=&\left\{ \left( \xi _{jl}^{s}\right) \in W\left(
x\right) :\mathcal{F}\left( \left( \xi _{jl}^{s}\right) ,x\right) \leq
\mathcal{F}_{\min }\left( x\right) +x\right\} \text{.}
\end{eqnarray*}

Because $\theta _{jl}^{si},g^{si},\tilde{\theta}_{jl}^{si},\tilde{g}%
^{si},\psi _{s}$ are semialgebraic, the objects defined above depend
semialgebraically on $x$. Thanks to conclusion \eqref{pit-star} of Lemma \ref{lemma-pit}, each $%
W\left( x\right) $ and each $\Xi_{OK}(x)$ is non-empty, and
\begin{equation}
\mathcal{F}_{\min }\left( x\right) \rightarrow 0\text{ as }x\rightarrow 0^{+}%
\text{.}  \label{star1}
\end{equation}

From Theorem \ref{semialgebraic-selection} we obtain

\begin{itemize}
\item[\LA{star2}] Semialgebraic functions $\xi _{jl}^{s}\left( x\right) $
on $\left( 0,\delta \right) $ such that $\left( \xi _{jl}^{s}\left( x\right)
\right) \in \Xi _{OK}\left( x\right) $ for each $x\in \left( 0,\delta
\right) $.
\end{itemize}

In particular, for $x\in \left( 0,\delta \right) $, we have%
\begin{eqnarray}
\sum_{j,l}\theta _{jl}^{s,i}\left( x\right) \xi _{jl}^{s}\left( x\right)
&=&g^{si}\left( x\right) \text{ for each }s,i,j;  \label{star3} \\
\left\vert \sum_{j,l}\tilde{\theta}_{jl}^{si}\left( x\right) \xi _{jl}^{s}\left(
x\right) -\tilde{g}%
^{si}\left( x\right) \right\vert &\leq &\left[ \mathcal{F}_{\min }\left(
x\right) +x\right] \text{ for each }s,i;  \label{star4}
\end{eqnarray}%
and
\begin{eqnarray}
&&\left\vert \xi _{jl}^{s}\left( x\right) -\sum_{k=0}^{m-l}\frac{1}{k!}\xi
_{j\left( l+k\right) }^{s-1}(x)\cdot \left( \psi _{s}\left( x\right) -\psi
_{s-1}\left( x\right) \right) ^{k}\right\vert  \notag \\
&\leq &\left[ \mathcal{F}_{\min }\left( x\right) +x\right] \cdot \left( \psi
_{s}\left( x\right) -\psi _{s-1}\left( x\right) \right) ^{m-l}\text{, for
each }s,j,l\text{ }\left( s\not=0\right) \text{.}  \label{star5}
\end{eqnarray}

From \eqref{star1}, \eqref{star4}, \eqref{star5}, we see that
\begin{equation}
\sum_{j,l}\tilde{\theta}_{jl}^{si}\left( x\right) \xi _{jl}^{s}\left(
x\right) =\tilde{g}^{si}\left( x\right) +o\left( 1\right) \text{ as }%
x\rightarrow 0^{+}\text{,}  \label{star6}
\end{equation}%
and
\begin{eqnarray}
&&\xi _{jl}^{s}\left( x\right) -\sum_{k=0}^{m-l}\frac{1}{k!}\xi _{j\left(
l+k\right) }^{s-1}(x)\cdot \left( \psi _{s}\left( x\right) -\psi _{s-1}\left(
x\right) \right) ^{k}  \notag \\
&=&o\left( \left[ \psi _{s}\left( x\right) -\psi _{s-1}\left( x\right) %
\right] ^{m-l}\right) \text{ as }x\rightarrow 0^{+}\text{.}  \label{star7}
\end{eqnarray}

Finally, from \eqref{star2}, \eqref{star3}, \eqref{star6}, \eqref{star7}, and the assertion \eqref{conversely} in Lemma \ref{lemma-pit},
we conclude that $\mathcal{H}|_{\Omega _{\delta' }}$ has a $C^m_{loc}$
semialgebraic section for some $\delta ^{\prime }<\delta $.

This completes the proof of Lemma \ref{main-proposition} and that of Theorem \ref{main-theorem}.

\renewcommand{\refname}{\begin{center}References\end{center}}
\bibliography{papers.bib}

\end{document}